\documentclass[reqno]{amsart}
\usepackage{amsmath, amssymb,amsthm,bm,mathrsfs}
\usepackage[utf8]{inputenc}
\usepackage{enumerate}
\usepackage{hyperref}
\usepackage{graphicx}
\usepackage{pgf,tikz,pgfplots}
\pgfplotsset{compat=1.16}
\usetikzlibrary{arrows.meta,patterns,matrix,decorations.pathreplacing,calc,positioning}

\usepackage{todonotes}

\newtheorem{theorem}{Theorem}
\newtheorem{lemma}{Lemma}
\newtheorem{corollary}{Corollary}

\newtheorem{proposition}[lemma]{Proposition}

\newtheorem{manualtheoreminner}{Theorem}
\newenvironment{manualtheorem}[1]{%
  \renewcommand\themanualtheoreminner{#1}%
  \manualtheoreminner
}{\endmanualtheoreminner}

\theoremstyle{definition}
\newtheorem{definition}{Definition}

\theoremstyle{remark}
\newtheorem{remark}{Remark}
\newtheorem{example}{Example}

\title[A restricted $2$-plane transform]{A restricted $2$-plane transform related to Fourier Restriction for surfaces of codimension $2$}

\author{Spyridon Dendrinos}
\address[S. Dendrinos]{School of Mathematical Sciences, University College Cork, Western Gateway Building, Western Road, Cork, Ireland }
\email{sd@ucc.ie}

\author{Andrei Mustata}
\address[A. Mustata]{School of Mathematical Sciences, University College Cork, Western Gateway Building, Western Road, Cork, Ireland }
\email{andrei.mustata@ucc.ie}

\author{Marco Vitturi}
\address[M. Vitturi]{School of Mathematical Sciences, University College Cork, Western Gateway Building, Western Road, Cork, Ireland}
\email{marco.vitturi@ucc.ie}

\begin{document}
\begin{abstract}
We draw a connection between the affine invariant surface measures constructed by P.\@ Gressman in \cite{Gressman2019} and the boundedness of a certain geometric averaging operator associated to surfaces of codimension $2$ and related to the Fourier Restriction Problem for such surfaces. For a surface given by $(\xi, Q_1(\xi), Q_2(\xi))$, with $Q_1,Q_2$ quadratic forms on $\mathbb{R}^d$, the particular operator in question is the $2$-plane transform restricted to directions normal to the surface, that is
\[ \mathcal{T}f(x,\xi) := \iint_{|s|,|t| \leq 1} f(x - s \nabla Q_1(\xi) - t \nabla Q_2(\xi), s, t)\,ds\,dt, \]
where $x,\xi \in \mathbb{R}^d$. We show that when the surface is well-curved in the sense of Gressman (that is, the associated affine invariant surface measure does not vanish) the operator satisfies sharp $L^p \to L^q$ inequalities for $p,q$ up to the critical point. We also show that the well-curvedness assumption is necessary to obtain the full range of estimates. The proof relies on two main ingredients: a characterisation of well-curvedness in terms of properties of the polynomial $\det(s \nabla^2 Q_1 + t \nabla^2 Q_2)$, obtained with Geometric Invariant Theory techniques, and Christ's Method of Refinements. With the latter, matters are reduced to a sublevel set estimate, which is proven by a linear programming argument.
\end{abstract}

\maketitle

\section{Introduction}
The $k$-plane transform in $\mathbb{R}^n$ is the operator $T_{n,k}$ defined by 
\[ T_{n,k} f(\pi) := \int_{\pi} f \,d\mathcal{L}_{\pi}, \]
where $\pi$ is any affine $k$-plane in $\mathbb{R}^n$ and $d\mathcal{L}_{\pi}$ denotes the Lebesgue measure on $\pi$. Such operators are generalisations of the X-ray transform and of the Radon transform, with which they coincide when $k=1$ and $k=n-1$ respectively. The strongest results for the boundedness of $T_{n,k}$ for $(n,k)$ generic have been obtained by M.\@ Christ in \cite{Christ84}, in which he proved a range of mixed-norm estimates (building upon work of S.\@ W.\@ Drury in \cite{Drury1983,Drury1984}); see also \cite{ROberlin2010} for some improvements for a subset of $(n,k)$ values and \cite{DuoandikoetxeaNaibo} for a survey of further developments. The particular case of $k=1$ has been the object of considerable attention due to its relationship with the Kakeya maximal function -- see T.\@ Wolff's influential paper \cite{Wolff1998} for the $n=3$ case, \cite{LabaTao} for generic $n$ and again \cite{ROberlin2010} for other improvements.\par 
In this paper we will be concerned with the restriction of the $2$-plane transform to particular sets of directions -- ones that arise as normals to surfaces of codimension $2$ that are ``well-curved'', in a sense that will be made precise later on (we regard the identification of the correct notion of well-curvedness as one of the main aims of this paper). A number of instances of restricted $T_{n,k}$ transforms exist in the literature, particularly when $k=1$: 
\begin{enumerate}[i)]
\item The restriction of the X-ray transform $T_{n,1}$ to a one-dimensional set of directions of the form $(\gamma(t),1)$, with $\gamma : [-1,1] \to \mathbb{R}^{n-1}$ a curve, has first been considered by M.\@ Christ and B.\@ Erdo\u{g}an in \cite{ChristErdogan} for the moment curve $(t,t^2, \ldots, t^{n-1})$; those results have later been extended to the sharp mixed-norm range by the first author and B.\@ Stovall in \cite{DendrinosStovall,DendrinosStovall2}. In this case, in order to obtain estimates for the largest range of exponents it is vital to assume that the curve $\gamma$ is well-curved in the sense of having non-vanishing torsion. The latter condition is equivalent to the non-vanishing of the affine invariant surface measure on $\gamma$ as introduced by Gressman in \cite{Gressman2019}.\footnote{This measure further coincides with the well-known affine arclength from Affine Geometry.}
\item The restriction of the X-ray transform $T_{n,1}$ to two-dimensional sets of directions was studied by B.\@ Erdo\u{g}an and R.\@ Oberlin in \cite{ErdoganOberlin}; there the authors considered directions of the form $(\varphi(u,v),1)$ for various examples of maps $\varphi: [-1,1]^2 \to \mathbb{R}^{n-1}$. It can be verified by the methods of \cite{Gressman2019} (in particular, by Theorem 6 in that paper) that in all their examples the affine invariant surface measure on the surface $\varphi([-1,1]^2)$ is non-vanishing. %
\item The restriction of the X-ray transform $T_{n,1}$ to the $(n-2)$-dimensional set of directions given by light-rays (that is, directions of the form $(\omega, 1)$ with $\omega \in \mathbb{S}^{n-2}$) was studied by T.\@ Wolff in \cite{Wolff2001}, in which mixed-norm estimates were proven in a certain range (not believed to be sharp). In this case the set of directions possesses curvature because the sphere $\mathbb{S}^{n-2}$ is curved. 
\item The restriction of the Radon transform $T_{n,n-1}$ to hyperplanes orthogonal to directions of the form $(\Gamma(\xi),1)$, with $\Gamma : [-1,1]^m \to \mathbb{R}^{n-1}$ the parametrisation of an $m$-dimensional submanifold of $\mathbb{R}^{n-1}$, has been considered by P.\@ Gressman in \cite{Gressman2022}.\footnote{More precisely, the operator here described is the dual operator to the one described in Example 3, Section 6 of \cite{Gressman2022}.} Combining the methods of that paper with those of \cite{Gressman2019}, one obtains non-trivial $L^p \to L^q$ estimates under the assumption that the image of $\Gamma$ has affine invariant surface measure (as per \cite{Gressman2019}) that is non-vanishing.
\end{enumerate}
We are not aware of restrictions of $T_{n,k}$ transforms for $k$ other than $1$ or $n-1$ that have been studied in the literature;\footnote{Save perhaps for \cite{DOberlin14}, which however has a measure-theoretic flavour rather than the geometric flavour we are interested in.} ours seems to be the first such instance.\par 
We will now introduce the restriction of the $2$-plane transform $T_{n,2}$ that we are going to consider in this paper. Besides fitting in well within the aforementioned literature, the operators we are about to introduce arise naturally in the study of Fourier Restriction for surfaces of codimension 2, as will be illustrated in Section \ref{section:motivation}. Let $d\geq 2$ and take a compact quadratic surface of codimension $2$ in $\mathbb{R}^{d+2}$, given as a graph by the parametrisation
\[ \phi(\xi) :=  (\xi, Q_1(\xi), Q_2(\xi)), \qquad \xi \in [-1,1]^d, \]
where $Q_1, Q_2$ are quadratic forms on $\mathbb{R}^d$; we use $\Sigma(Q_1,Q_2)$ to denote the surface $\phi([-1,1]^d)$. It will also be convenient to introduce the real symmetric $d \times d$ matrices $A,B$ that correspond to the Hessians $\nabla^2 Q_1, \nabla^2 Q_2$, that is, the matrices given by 
\[ A \xi := \nabla Q_1 (\xi), \quad B\xi := \nabla Q_2(\xi), \quad \text{ for all } \xi \in \mathbb{R}^d. \]
\begin{remark}
We concentrate on quadratic surfaces for simplicity of exposition, but the main result that will be given in Section \ref{section:main_results} (Theorem \ref{main_theorem}) holds for more general surfaces, as will be explained there.
\end{remark}
To any such pair of quadratic forms (or equivalently, to any surface $\Sigma(Q_1,Q_2)$) we associate the operator $\mathcal{T} = \mathcal{T}_{Q_1, Q_2}$, acting on (Schwartz) functions $f : \mathbb{R}^{d+2} \to \mathbb{C}$, given by 
\begin{equation}\label{eq:definition_restricted_2-plane_transform}
\mathcal{T}f(x,\xi) := \iint_{|s|,|t|\leq 1} f(x - s\, \nabla Q_1(\xi) - t  \, \nabla Q_2(\xi), s, t)\,ds\,dt, 
\end{equation}
where $x \in \mathbb{R}^d, \xi \in [-1,1]^d$. The operator $\mathcal{T}$ is a (local) $2$-plane transform in a restricted set of directions parametrised by $\xi$: indeed, the $2$-plane in question is given by 
\[ \pi_{x,\xi} := \{ (x,0,0) + s(-\nabla Q_1(\xi), 1, 0) + t(- \nabla Q_2(\xi),0,1) : s,t \in \mathbb{R}\}; \]
moreover, it is readily verified that $\pi_{x,\xi}$ is normal to the tangent plane of $\Sigma(Q_1,Q_2)$ at the point $\phi(\xi)$. Notice that in \eqref{eq:definition_restricted_2-plane_transform} we are not integrating with respect to the Lebesgue measure on the $2$-plane as one does in $T_{n,2}$, but the $ds\,dt$ measure is nevertheless comparable to it since $\xi$ is bounded, so that, if we were to extend the integration in \eqref{eq:definition_restricted_2-plane_transform} to all $s, t\in\mathbb{R}$, we would have 
\[ \mathcal{T}f(x, \xi) \leq T_{d+2,2}f(\pi_{x,\xi}) \lesssim_{Q_1,Q_2} \mathcal{T}f(x,\xi). \]
To gauge the severity of the restriction in directions, notice that the Grassmannian $\operatorname{Gr}(2,d+2)$ of $2$-dimensional linear subspaces of $\mathbb{R}^{d+2}$ has dimension $2d$, whereas the submanifold of $\operatorname{Gr}(2,d+2)$ given by the directions of the family of $2$-planes $\pi_{x,\xi}$ above is parametrised by $\xi$ and thus has dimension at most $d$.\par 
We are interested in the boundedness properties of $\mathcal{T}$ and how these relate to how well-curved the surface $\Sigma(Q_1,Q_2)$ is. A general collection of estimates one can study are the mixed-norm ones, which we are now going to introduce. Let $q,r \geq 1$ and define for any $F : \mathbb{R}^d \times [-1,1]^d \to \mathbb{C}$ its $L^q(L^r)$ mixed-norm\footnote{With this order of integration, this is sometimes called the \emph{Kakeya-order} mixed-norm.} to be
\begin{equation}\label{eq:mixed_norm_definition}
\|F\|_{L^q(L^r)} := \Big(\int_{[-1,1]^d} \Big(\int_{\mathbb{R}^d} |F(x,\xi)|^r \,dx\Big)^{q/r} \,d\xi\Big)^{1/q} 
\end{equation}
(notice that when $q=r$ the $L^q(L^q)$-norm is simply the usual $L^q$-norm). For exponents $p,q,r \geq 1$, we say that $\mathcal{T}$ satisfies the mixed-norm estimate $L^p \to L^q(L^r)$ if we have the a-priori estimate
\begin{equation}\label{eq:generic_mixed_norm_estimate}
\|\mathcal{T} f\|_{L^q(L^r)} \lesssim_{p,q,r} \|f\|_{L^p}. 
\end{equation}
For the rest of the paper we will make the assumption that $f$ is supported on, say, $B(0,C) \times [-1,1]^{2}$ for some $C>0$. Due to the local nature of the operator $\mathcal{T}$, this assumption can be removed when $r\ge q \ge p$ by a standard localisation argument.
\begin{remark}
By a standard duality and discretisation argument, any estimate of the form \eqref{eq:generic_mixed_norm_estimate} translates into a Kakeya-type bound for collections of $\delta \times \ldots \times \delta \times 1 \times 1$ slabs associated to $\Sigma(Q_1,Q_2)$; see Section \ref{section:motivation} for details (in particular Corollary \ref{corollary:Kakeya_2_slabs} there) and an application.
\end{remark}
Testing the mixed-norm inequalities \eqref{eq:generic_mixed_norm_estimate} against some simple geometric examples leads to a conjectural range of boundedness as will now be described. Let $0 < \delta < 1$ and let $B_n(r)$ denote the $n$-dimensional ball of radius $r$ centred at the origin. We use $A,B$ in place of $\nabla^2 Q_1, \nabla^2 Q_2$ for convenience. Observe that for $|s|\lesssim \|A\|^{-1} \delta$ and $|t|\lesssim \|B\|^{-1} \delta$ we have $|x - sA \xi - tB\xi|\leq \delta$ for all $|x| \lesssim \delta$ and all $\xi \in [-1,1]^d$. Therefore 
\[ \mathcal{T}\mathbf{1}_{B_{d+2}(\delta)}(x,\xi) \gtrsim \delta^2 \, \mathbf{1}_{B_{d}(O(\delta))}(x) \, \mathbf{1}_{[-1,1]^d}(\xi), \]
so that for \eqref{eq:generic_mixed_norm_estimate} to hold as $\delta \to 0$ we see with a simple computation that we must have 
\[ 2 + \frac{d}{r} \geq \frac{d+2}{p}. \]
For our second example, let $S_\delta$ denote the ``slab''
\[ S_\delta := \{ (x - sA\xi - tB\xi, s,t) : |s|,|t|\sim 1,\, x,\xi \in B_d(\delta)\}, \]
and observe that $|S_\delta| \lesssim \delta^d$ by similar considerations as above. Clearly we have 
\[ \mathcal{T}\mathbf{1}_{S_\delta}(x,\xi) \gtrsim \mathbf{1}_{B_d(\delta)}(x) \, \mathbf{1}_{B_d(\delta)}(\xi), \]
and thus if estimate \eqref{eq:generic_mixed_norm_estimate} is to hold as  $\delta \to 0$ we obtain a second necessary condition. The two conditions together are then 
\begin{equation}\label{eq:mixed_necessary_conditions}
\begin{cases}
\begin{aligned}
2 + \dfrac{d}{r} &\geq \dfrac{d+2}{p}, \\
\dfrac{1}{r} + \dfrac{1}{q} &\geq \dfrac{1}{p}. 
\end{aligned}
\end{cases}
\end{equation}\par
%
%
\begin{remark}\label{remark:trivial_estimates} We record the following trivial facts about certain exponents in the range allowed by \eqref{eq:mixed_necessary_conditions}:
\begin{enumerate}[i)]
\item inequality \eqref{eq:generic_mixed_norm_estimate} is certainly satisfied for $p=\infty$ and for every $1 \leq q, r \leq \infty$ (recall that we are assuming $f$ is supported in $B(0,C) \times [-1,1]^2$);
\item inequality \eqref{eq:generic_mixed_norm_estimate} is certainly satisfied for $p=r=1$ and for every $1 \leq q \leq \infty$; 
\item if inequality \eqref{eq:generic_mixed_norm_estimate} holds for exponents $(p,q,r)$ then it also holds for any exponents $(p,\tilde{q},r)$ with $1 \leq \tilde{q} \leq q$ (by H\"{o}lder inequality).
\end{enumerate} 
\end{remark}
We conjecture that when $\Sigma(Q_1,Q_2)$ is well-curved (in a sense to be made precise shortly; see Definition \ref{defn:well_curved} of next subsection) then the necessary conditions \eqref{eq:mixed_necessary_conditions} are also sufficient, with the possible exception of the endpoint $L^{(d+2)/2} \to L^{(d+2)/2}(L^\infty)$. In this paper we will concern ourselves mainly with non-mixed-norm estimates, that is estimates with $q = r$ (this is because mixed-norm estimates are not accessible with the methods we employ, at least not without significant reworking); in this case the necessary conditions are rewritten as 
\[ 2 + \frac{d}{q} \geq \frac{d+2}{p}, \qquad \frac{2}{q} \geq \frac{1}{p}. \]
As described in the next subsection, we are able to confirm the conjecture in the non-mixed-norm range given by these conditions, with the exclusion of a critical line. By interpolation with the trivial inequalities observed above, one also obtains a range of mixed-norm inequalities as a consequence.

\subsection{Main results}\label{section:main_results}
In order to state our main results, we will now clarify the notion of curvature that we are going to employ. It is based upon P.\@ Gressman's work in \cite{Gressman2019}, in which a construction was provided that, given a submanifold $\mathcal{M}$ of $\mathbb{R}^n$, produces a unique (up to multiplicative constants) surface measure $\nu_{\mathcal{M}}$ (that is, a measure with support on $\mathcal{M}$ and absolutely continuous with respect to the standard surface measure) which is equi-affine invariant.\footnote{That is, if $T$ is an affine transformation of $\mathbb{R}^n$ that preserves volumes one has $\nu_{T(\mathcal{M})}(T(E)) = \nu_{\mathcal{M}}(E)$ for all Borel sets $E$.} Moreover, the measure $\nu_{\mathcal{M}}$ satisfies an affine curvature condition of the form $\nu_{\mathcal{M}}(R) \lesssim |R|^{\alpha}$ for every rectangle $R$ in $\mathbb{R}^n$ (for a specific value of $\alpha$ that depends only on $n$ and $\dim \mathcal{M}$), and is the largest such measure up to multiplicative constants. Details on Gressman's construction will be provided in Section \ref{section:affine_curvature_GIT}.
\begin{definition}\label{defn:well_curved}
We say that a submanifold $\mathcal{M}$ of $\mathbb{R}^n$ is \emph{well-curved} if the density of its affine invariant surface measure $\nu_{\mathcal{M}}$ (with respect to the standard surface measure $d\sigma$) does not vanish anywhere on $\mathcal{M}$. If the density of $\nu_{\mathcal{M}}$ vanishes identically, we say that $\mathcal{M}$ is \emph{flat}.
\end{definition}
When the submanifold $\mathcal{M} \subset \mathbb{R}^n$ has codimension $1$ or $n-1$, the measure $\nu_{\mathcal{M}}$ corresponds respectively to the affine hypersurface measure and the affine arclength (see Theorem 1, part (4) of \cite{Gressman2019}). In these two extremal cases, the submanifold is then well-curved if the Gaussian curvature is non-vanishing or if the torsion is non-vanishing, respectively -- thus recovering the common notions of well-curvedness for such codimensions present in the literature. Definition \ref{defn:well_curved} should also be compared to the curvature assumptions present in the examples of restricted $k$-plane transforms listed at the beginning of this section.\par 
In the case of the compact quadratic surfaces $\mathcal{M} = \Sigma(Q_1,Q_2)$ we have that $d\xi/d\sigma$ is bounded away from zero, and therefore $\Sigma(Q_1,Q_2)$ is well-curved according to our definition if and only if $d\nu_{\mathcal{M}}/d\xi$ does not vanish. However, it is shown in \cite{Gressman2019} (see also Section \ref{section:affine_curvature_GIT}) that, for a surface in such a form, the density $d\nu_{\mathcal{M}}/d\xi$ is actually a constant that depends only on $Q_1,Q_2$, and therefore $\Sigma(Q_1,Q_2)$ is well-curved if and only if that constant is non-zero -- and if it is zero, then the surface is flat. Thus in our quadratic case the well-curved/flat distinction of Definition \ref{defn:well_curved} will be a perfect dichotomy.\par
We can now state our main result, which connects the boundedness properties of operators \eqref{eq:definition_restricted_2-plane_transform} to the curvature of $\Sigma(Q_1,Q_2)$.
\begin{theorem}[well-curved surfaces]\label{main_theorem}
Let $Q_1, Q_2$ be quadratic forms on $\mathbb{R}^d$ and suppose that the quadratic surface $\Sigma(Q_1,Q_2)$ is well-curved. Then, for every $1\leq p,q \leq \infty$ such that 
\[ 2 + \frac{d}{q} > \frac{d+2}{p} \quad \text{ and } \quad \frac{2}{q} \geq \frac{1}{p}, \]
we have 
\[ \|\mathcal{T}f\|_{L^q} \lesssim_{p,q, Q_1, Q_2} \|f\|_{L^p} \]
for every function $f$ supported in $B(0,C) \times [-1,1]^2$.\par
If instead the surface $\Sigma(Q_1,Q_2)$ is not well-curved (hence flat), then every $L^p \to L^q$ estimate with $(p,q)$ sufficiently close to the endpoint $\big(\tfrac{d+4}{4}, \tfrac{d+4}{2}\big)$ is false.
\end{theorem}
The examples that yield the conjectural range \eqref{eq:mixed_necessary_conditions} show that the range of exponents in the theorem above is sharp, save perhaps for the missing critical line $2 + d/q = (d+2)/p$. The theorem is obtained by interpolating the trivial $L^p \to L^1$ and $L^{\infty} \to L^q$ estimates from Remark \ref{remark:trivial_estimates} with restricted weak-type estimates along the critical line $2/q = 1/p$ and arbitrarily near the endpoint estimate $L^{(d+4)/4} \to L^{(d+4)/2}$. The latter are obtained using Christ's Method of Refinements, but alternative proofs can be given using techniques of Gressman from either \cite{Gressman2022} or \cite{Gressman2022b}; see Remark \ref{remark:other_proofs} in this regard.\par 
We observe that in general it is possible with our methods to obtain the restricted weak-type endpoint estimate as well, unless the surface $\Sigma(Q_1,Q_2)$ belongs to a certain class that can be described explicitly; this description relies upon Theorem \ref{thm:well_curvedness_characterisation} below and will be given in Remark \ref{remark:surfaces_without_log_loss} of Section \ref{section:sublevel_set_estimates}. As stated, the range of exponents is also sharp in the curvature condition, in the sense that the range of true estimates is necessarily smaller when the surface is flat (this will be proven in Section \ref{section:flat_surfaces} -- see also Theorem \ref{thm:flat_surfaces} below). In particular, Theorem \ref{main_theorem} shows that any $L^p \to L^q$ estimate for $\mathcal{T}$ with $(p,q)$ near the endpoint is equivalent to the well-curvedness of $\Sigma(Q_1,Q_2)$ (see \cite{IosevichLu} for a result of similar flavour in the context of Fourier Restriction for hypersurfaces).\par
Our methods are sufficiently stable under perturbation that we are also able to extend Theorem \ref{main_theorem} to more general codimension $2$ surfaces. Indeed, let $\varphi_1,\varphi_2 : \mathbb{R}^d \to \mathbb{R}$ be $C^2$ functions such that $\nabla \varphi_1(0) = \nabla \varphi_2(0) = 0$ and let $\Sigma(\varphi_1,\varphi_2)$ denote the surface parametrised by 
\[ (\xi, \varphi_1(\xi), \varphi_2(\xi)), \qquad \xi \in [-\epsilon, \epsilon]^d, \]
where $\epsilon > 0$ is sufficiently small depending on $\varphi_1, \varphi_2$. The analogue of operator \eqref{eq:definition_restricted_2-plane_transform}, denoted by $\mathcal{T}_{\varphi_1,\varphi_2}$, is given by
\[ \mathcal{T}_{\varphi_1,\varphi_2}f(x,\xi) := \iint_{|s|,|t|\leq 1} f(x - s\, \nabla \varphi_1(\xi) - t  \, \nabla \varphi_2(\xi), s, t)\,ds\,dt. \]
\begin{manualtheorem}{\ref{main_theorem}${}^\prime$}[General well-curved surfaces]\label{thm:general_well_curved_surfaces}
Let $\varphi_1,\varphi_2$ be as above and suppose that $\Sigma(\varphi_1, \varphi_2)$ is well-curved at $\xi=0$. Then, for every $1\leq p,q \leq \infty$ such that 
\[ 2 + \frac{d}{q} > \frac{d+2}{p} \quad \text{ and } \quad \frac{2}{q} \geq \frac{1}{p}, \]
we have 
\[ \|\mathcal{T}_{\varphi_1,\varphi_2} f\|_{L^q} \lesssim_{p,q, \varphi_1, \varphi_2} \|f\|_{L^p} \]
for every function $f$ supported in $B(0,C) \times [-1,1]^2$.
\end{manualtheorem}
The range of exponents above is identical to the one given in Theorem \ref{main_theorem}. To show Theorem \ref{thm:general_well_curved_surfaces}, only small adjustments need to be made to the argument for the quadratic surface case -- the necessary modifications will be sketched in Appendix \ref{appendix:modification_general_surface}.\par 
By standard interpolation theory for mixed-norm spaces (see e.g. \cite{BenedekPanzone}), one obtains from the strong-type inequalities of Theorem \ref{main_theorem} a whole range of mixed-norm estimates of the form \eqref{eq:generic_mixed_norm_estimate}, upon interpolation with the (strong-type) trivial estimates in Remark \ref{remark:trivial_estimates}.
\begin{corollary}[mixed-norm range]
Let $Q_1, Q_2$ be quadratic forms on $\mathbb{R}^d$ and suppose that the quadratic surface $\Sigma(Q_1,Q_2)$ is well-curved. Then for every $1\leq p,q,r \leq \infty$ such that 
\[ 2 + \frac{d}{r} > \frac{d+2}{p}, \qquad \frac{1}{r} + \frac{1}{q} \geq \frac{1}{p}, \quad \text{ and } \quad \frac{2}{r} \geq \frac{1}{p}, \]
we have 
\[ \|\mathcal{T}f\|_{L^q(L^r)} \lesssim_{p,q,r, Q_1, Q_2} \|f\|_{L^p} \]
for every function $f$ supported in $B(0,C) \times [-1,1]^d$.
\end{corollary}
The proof of Theorem \ref{main_theorem} rests on an algebraic characterisation of well-curvedness which is enabled by a connection between Gressman's affine invariant measures and Geometric Invariant Theory; it is of independent interest. Specifically, we prove the following fact.
\begin{theorem}\label{thm:well_curvedness_characterisation}
Let $Q_1, Q_2$ be quadratic forms on $\mathbb{R}^d$. The quadratic surface $\Sigma(Q_1,Q_2)$ is well-curved if and only if the following condition is satisfied:
\begin{equation}
\parbox{0.8\textwidth}{%
the homogeneous polynomial in $s,t$ given by $\det(s\nabla^2 Q_1 + t \nabla^2 Q_2)$ does not vanish identically and does not admit any root of multiplicity larger than $d/2$.
} \tag{\textbf{M}}\label{condition:well-curvedness}
\end{equation}
\end{theorem}
Here by \emph{root} of a homogeneous polynomial in $\mathbb{R}[s,t]$ we mean a homogeneous linear divisor $as+bt$ in $\mathbb{C}[s,t]$, and by its (algebraic) multiplicity we mean the largest power $m$ such that $(as + bt)^m$ is still a divisor. Theorem \ref{thm:well_curvedness_characterisation} is stated for quadratic forms, but it holds ``pointwise'' for arbitrary $\Sigma(\varphi_1,\varphi_2)$ surfaces: the surface is well-curved if $\det(s \nabla^2 \varphi_1(\xi) + t \nabla^2 \varphi_2(\xi))$ satisfies \eqref{condition:well-curvedness} for every $\xi$.
\begin{example}
Consider the quadratic surfaces $\Sigma(Q_1,Q_2)$ given by 
\[ Q_1(\xi) := \frac{1}{2}\sum_{j=1}^{d} \lambda_j \xi_j^2, \quad Q_2(\xi) := \frac{1}{2}\sum_{j=1}^{d} \mu_j \xi_j^2, \]
where the $\lambda_j, \mu_j$ are real coefficients that for any $j$ are not simultaneously zero. We have 
\[ \det(s\nabla^2 Q_1 + t \nabla^2 Q_2) = \prod_{j=1}^{d} (s \lambda_j + t \mu_j) \]
and thus by Theorem \eqref{thm:well_curvedness_characterisation} the surface $\Sigma(Q_1,Q_2)$ is well-curved if $\#\{ j : [\lambda_j : \mu_j] = [\lambda : \mu] \} \leq d/2$ for all $[\lambda : \mu] \in \mathbb{P}(\mathbb{R}^2)$.
\end{example}
This is not the first instance in which the object $\det(s\nabla^2 Q_1 + t \nabla^2 Q_2)$ and condition \eqref{condition:well-curvedness} have made their appearance in Harmonic Analysis: readers familiar with M.\@ Christ's PhD thesis \cite{ChristThesis} will recognise \eqref{condition:well-curvedness} above as being precisely the condition that yields the sharp $L^p \to L^2$ estimates for the operator of Fourier Restriction to surfaces $\Sigma(Q_1,Q_2)$. Thus, in light of Theorem \ref{thm:well_curvedness_characterisation}, M.\@ Christ's result can be retroactively reformulated as saying that the Fourier Restriction operator $R f := \widehat{f} \,\big|_{\Sigma(Q_1,Q_2)}$ satisfies optimal $L^p \to L^2$ estimates if and only if $\Sigma(Q_1,Q_2)$ is well-curved in the sense of Definition \ref{defn:well_curved}. See Section \ref{section:motivation} for additional details.\par
The characterisation of well-curvedness provided above is quantitative to some extent, and in particular it gives us a way to gauge the ``flatness'' of surfaces which are not well-curved. Indeed, a flat $\Sigma(Q_1,Q_2)$ surface must be such that $\det(s \nabla^2 Q_1 + t \nabla^2 Q_2)$ has a root of multiplicity $m_{\ast} > d/2$, which in particular is the largest of all the root multiplicities. Intuitively, we expect that as the largest multiplicity $m_{\ast}$ increases, the surface gets flatter (with the most extreme case being that in which $\det(s\nabla^2 Q_1 + t\nabla^2 Q_2)$ vanishes identically); consequently, we expect the $L^p \to L^q$ mapping properties of operator \eqref{eq:definition_restricted_2-plane_transform} to worsen. It turns out that indeed this largest multiplicity $m_{\ast}$ controls the surviving range of boundedness of the operators \eqref{eq:definition_restricted_2-plane_transform}, particularly along the critical line $2/q = 1/p$. We have the following partial analogue of Theorem \ref{main_theorem} for flat surfaces.
\begin{theorem}[flat surfaces]\label{thm:flat_surfaces}
Let $Q_1, Q_2$ be quadratic forms on $\mathbb{R}^d$ and suppose that the quadratic surface $\Sigma(Q_1,Q_2)$ is flat but $\det(s\nabla^2 Q_1 + t \nabla^2 Q_2)$ is not identically vanishing. Let $m_{\ast} > d/2$ denote the largest multiplicity among its roots. Then for every $1\leq p,q \leq \infty$ such that 
\[ 1 + \frac{m_{\ast}}{q} \geq \frac{m_{\ast}+1}{p} \quad \text{ and } \quad \frac{2}{q} \geq \frac{1}{p}, \]
with the exception of $p = (m_{\ast}+2)/2$, $q = m_{\ast}+2$, we have 
\[ \|\mathcal{T}f\|_{L^q} \lesssim_{p,q, Q_1, Q_2} \|f\|_{L^p} \]
for every function $f$ supported in $B(0,C) \times [-1,1]^d$. Moreover, every $L^p \to L^q$ estimate with $1 + m_{\ast} / q < (m_{\ast}+1)/p$ and $2/q = 1/p$ is false.\par 
If instead $\det(s\nabla^2 Q_1 + t \nabla^2 Q_2)$ vanishes identically, then there is an $\epsilon = \epsilon_{Q_1,Q_2}$ with $0 < \epsilon < 1$ such that every $L^p \to L^q$ estimate with $\tfrac{2 - \epsilon}{q} < \tfrac{1}{p}$ is false (this includes in particular estimates with $2/q = 1/p$ for $(p,q) \neq (\infty, \infty)$).
\end{theorem}
The statement above does not paint the full picture: our counterexamples rule out a range of exponents beyond those on the line $2/q = 1/p$; however, which exponents we are able to rule out depends on properties of $Q_1, Q_2$ (or rather, of the associated Hessian matrices $A,B$) that go beyond the single value $m_{\ast}$. We direct the reader to Section \ref{section:flat_surfaces} for the more precise picture, and particularly to condition \eqref{eq:flat_surfaces_necessary_condition_1} and Figure \ref{figure:flat_surfaces_range} there.\par 
The ranges given in Theorem \ref{thm:flat_surfaces} are strict subsets of that given in Theorem \ref{main_theorem}, and the aforementioned counterexamples of Section \ref{section:flat_surfaces} show that this is necessarily the case. Moreover, these ranges become smaller as $m_{\ast}$ increases. We do not know whether the given ranges are sharp for all flat surfaces $\Sigma(Q_1,Q_2)$ outside of the line $2/q = 1/p$, but we are able to show that they are for some classes of surfaces. This will also be detailed in Section \ref{section:flat_surfaces}.
\subsection{Structure of the paper} In Section \ref{section:motivation} we provide context for the study of operators \eqref{eq:definition_restricted_2-plane_transform} by describing how they relate to the Fourier Restriction problem for surfaces of codimension $2$ such as $\Sigma(Q_1,Q_2)$; the connection passes through Kakeya-type estimates, and some application of these is also discussed. In Section \ref{section:affine_curvature_GIT} we recall Gressman's construction of affine invariant surface measures from \cite{Gressman2019} in the special case of a surface of codimension $2$, and we describe how the well-curvedness of such surfaces can be interpreted in algebraic terms via Geometric Invariant Theory. In Section \ref{section:semistability} we harness this connection to prove an algebraic characterisation of well-curvedness in terms of the multiplicity of the roots of polynomials $\det(s\nabla^2 Q_1 + t \nabla^2 Q_2)$ -- this is Theorem \ref{thm:well_curvedness_characterisation}. The argument is split in two parts, as the case in which $\det(s\nabla^2 Q_1 + t \nabla^2 Q_2)$ vanishes identically needs to be treated separately. With this preliminary work done, in Section \ref{section:main_proof} we prove Theorems \ref{main_theorem} and \ref{thm:flat_surfaces} with a particularly simple instance of Christ's Method of Refinements from \cite{Christ98}. The latter reduces matters to proving sharp sublevel set estimates for the polynomial $\det(s\nabla^2 Q_1 + t \nabla^2 Q_2)$, which are the subject of Section \ref{section:sublevel_set_estimates}. The proof is somewhat unusual in that it employs a simple linear programming argument; it might be of independent interest. In Section \ref{section:flat_surfaces} we discuss the case of flat surfaces of codimension $2$; we provide counterexamples that rule out various $L^p \to L^q$ estimates that are instead true for well-curved surfaces. Finally, in the appendix we sketch the modification needed to prove Theorem \ref{thm:general_well_curved_surfaces}.
\subsection*{Notation} For $M$ a matrix, we let $M^\top$ denote its transpose and $\|M\|$ denote its operator norm. For $E \subset \mathbb{R}^n$ a set, we let $\mathbf{1}_E$ denote its characteristic function and $|E|$ denote its Lebesgue measure. For non-negative quantities $A,B$, we write $A \lesssim B$ if there exists a constant $C>0$ such that $A \leq C B$. If the value of the constant $C$ depends on a list of parameters $\mathcal{P}$ we write $A \lesssim_{\mathcal{P}}B$ to highlight this fact. If $A \lesssim B$ and $B \lesssim A$, we write $A \sim B$. In conditional statements we will write $A \ll B$ to denote the inequality $A \leq c B$ for some sufficiently small constant $c>0$.
\section{Motivation and applications}\label{section:motivation}
In this section we will provide motivation for the study of the operators $\mathcal{T}$ given by \eqref{eq:definition_restricted_2-plane_transform}. Such motivation arises most prominently from the study of the Fourier Restriction problem and related matters such as the study of Kakeya/Besicovitch-type sets and the Mizohata-Takeuchi conjecture; we will review these in the context of codimension $2$ surfaces, as this will allow us to compare conditions present in the literature with our definition of well-curvedness.
\subsection{Fourier Restriction}
The Fourier Restriction problem for a submanifold $\mathcal{M} \subset \mathbb{R}^n$ (with surface measure $d\sigma$), in its equivalent adjoint formulation known as the Fourier Extension problem, is concerned with the boundedness properties of the Fourier Extension operator given by
\[ E_{\mathcal{M}} g(x) := \int_{\mathcal{M}} g(\xi) e^{2\pi i \xi \cdot x} \,d\sigma(\xi). \]
More specifically, one is interested in determining the full set of exponents $p,q$ for which estimates
\begin{equation}\label{eq:general_extension_estimate}
\|E_{\mathcal{M}} g\|_{L^q(\mathbb{R}^n)} \lesssim_{p,q} \|g\|_{L^p(\mathcal{M}, d\sigma)} 
\end{equation}
hold. The literature on this problem is immense (particularly in the case of codimension $1$) and we do not attempt to review it here; rather, we concentrate on (a selection of) works on the case of submanifolds of codimension $2$, which is most directly relevant to us and has been studied in a number of instances.\par
The first such instance addressing codimension $2$ specifically occurred in M.\@ Christ's PhD thesis \cite{ChristThesis}, in which he studied inequalities \eqref{eq:general_extension_estimate} for $p=2$; such results are commonly known as $L^2$-restriction theorems or as Tomas-Stein theorems. For quadratic surfaces $\Sigma(Q_1,Q_2)$ he proved\footnote{See Section 12 of \cite{ChristThesis}.} that under condition \eqref{condition:well-curvedness} of Section \ref{section:main_results} the extension operator $E_{\Sigma(Q_1,Q_2)}$ satisfies the $L^2 \to L^q$ estimates \eqref{eq:general_extension_estimate} for every $q \geq q_0 :=(2d+8)/d$ (which is sharp), with the exception of the case of $d$ even and $\det(s\nabla^2 Q_1 + t \nabla^2 Q_2)$ having a root of multiplicity exactly $d/2$, in which case $q > q_0$ instead. Moreover, he showed\footnote{See Section 3 of \cite{ChristThesis} and in particular Proposition 3.1 therein.} that \eqref{condition:well-curvedness} is also necessary, in the sense that if $\Sigma(Q_1,Q_2)$ violates the condition then the $L^2 \to L^q$ estimates are false for any $q$ sufficiently close to the endpoint $q_0$. In this work, condition \eqref{condition:well-curvedness} came about as the condition that would guarantee the appropriate decay of $\widehat{\mu}$, where $\mu$ is the measure given by 
\[ \mu(f) := \int_{[-1,1]^d} f(\xi, Q_1(\xi), Q_2(\xi))\,d\xi; \]
such decay is a fundamental ingredient in $L^2$-restriction arguments \`{a} la Tomas-Stein. In retrospect, it should come as no surprise that the endpoint or near-endpoint $L^2 \to L^{q_0}$ Fourier extension estimate -- and hence condition \eqref{condition:well-curvedness} -- is equivalent to the well-curvedness of the surface $\Sigma(Q_1,Q_2)$, as it was shown in \cite{IosevichLu} that this is also the case for hypersurfaces. The interpretation of \eqref{condition:well-curvedness} as a type of curvature condition was noted in \cite{ChristThesis}.\par 
Christ's $L^2$-restriction results were later extended by G.\@ Mockenhaupt \cite{Mockenhaupt} to flat quadratic surfaces and by L.\@ De Carli and A.\@ Iosevich \cite{DeCarliIosevich} to some flat non-quadratic surfaces. D.\@ Oberlin \cite{DOberlin04} proved Fourier restriction estimates beyond the Tomas-Stein range for $d=3$ and for the surface given by $Q_1(\xi) = \xi_1^2 + \xi_2^2$, $Q_2(\xi) = \xi_1^2 + \xi_3^2$. More recently, S.\@ Guo and C.\@ Oh \cite{GuoOh} have addressed the Fourier Restriction problem for general quadratic surfaces of codimension 2 in $\mathbb{R}^5$, proving estimates of type \eqref{eq:general_extension_estimate} that go beyond the Tomas-Stein range and are sharp for some classes of surfaces (all of them flat). Their only assumptions on the pair $(Q_1, Q_2)$ are that the quadratic forms are linearly independent and that $\ker \nabla^2 Q_1 \,\cap \, \ker \nabla^2 Q_2 = \{0\}$ -- in particular, this excludes only a rather degenerate subclass of the set of pairs $(Q_1,Q_2)$ for which $\det(s \nabla^2 Q_1 + t \nabla^2 Q_2)$ vanishes identically (see Section \ref{section:semistability_vanishing_determinant} for more general pairs with vanishing determinant). Interestingly, the range of exponents they obtain is the same for all pairs of quadratic forms considered; it is expected that a larger range could be obtained for well-curved surfaces.\par 
Having provided some context, we will now describe how the operator $\mathcal{T}$ makes its appearance in the Fourier Restriction problem. We will keep the discussion light by not worrying too much about rigour.\par
The most successful approaches to the Fourier Restriction problem to this date are all based on wavepacket decompositions. In the case of codimension 2 specifically (we will use $\Sigma$ for $\Sigma(Q_1,Q_2)$ for shortness), in order to study the extension operator $E_{\Sigma}$ one can equivalently study the modified extension operator
\[ E^{\delta}_{\Sigma}\, g(x) := \int_{\mathcal{N}_{\delta}(\Sigma)} g(\xi) e^{2\pi i \xi \cdot x} \,d\xi, \]
where $\mathcal{N}_{\delta}(\Sigma)$ is the $\delta$-neighbourhood of $\Sigma$ and $g$ is supported on this neighbourhood (thus $E^{\delta}_{\Sigma} g = \widehat{g}$). Estimates \eqref{eq:general_extension_estimate} are then replaced by local-type estimates of the form
\begin{equation}\label{eq:general_local_extension_estimate}
\|E^{\delta}_{\Sigma} \,g\|_{L^q(B(\delta^{-1}))} \lesssim_{p,q,\alpha} \delta^{2/{p'} - \alpha} \|g\|_{L^p(\mathcal{N}_{\delta}(\Sigma))} 
\end{equation}
for every $\delta \leq 1$ and every $\alpha \geq 0$, where $B(\delta^{-1})$ is the ball of radius $\delta^{-1}$ centred at $0$. These estimates are known to imply estimates of type \eqref{eq:general_extension_estimate} (see e.g.\ Section 4 of \cite{GuoOh}). The reason for passing to local-type estimates is that $\mathcal{N}_{\delta}(\Sigma)$ can be neatly partitioned into parabolic boxes adapted to the geometry of $\Sigma$, and such a partition automatically yields a geometrically meaningful way to partition $g$ and $E^{\delta}_{\Sigma} \,g$. The parabolic box that approximates $\mathcal{N}_{\delta}(\Sigma)$ in the vicinity of point $\phi(\xi) = (\xi, Q_1(\xi), Q_2(\xi))$ must have dimensions $\sim \delta^{1/2} \times \ldots \times \delta^{1/2} \times \delta \times \delta$ (this can be seen by a Taylor expansion). It can be described as the set of points given by
\[ \phi(\xi) + \sum_{j=1}^{d} \delta^{1/2} \lambda_j \bm{v}_j(\xi) + \delta \nu_1 \bm{n}_1(\xi) + \delta \nu_2 \bm{n}_2(\xi) \]
for arbitrary $|\lambda_j|, |\nu_1|,|\nu_2|\lesssim 1$ , where\footnote{$\bm{e}_j$ denotes the $j$-th element in the standard basis of $\mathbb{R}^d$.}
\begin{align*}
\bm{v}_j(\xi) &:= (\bm{e}_j, \partial_j Q_1(\xi), \partial_j Q_2(\xi)), \qquad j=1,\ldots, d \\
\bm{n}_1(\xi) &:= (-\nabla Q_1(\xi), 1, 0), \\
\bm{n}_2(\xi) &:= (-\nabla Q_2(\xi), 0, 1);
\end{align*}
here the $\bm{v}_j$ span the directions tangent to $\Sigma$ and $\bm{n}_1,\bm{n}_2$ span the normal ones. Given a collection $\mathcal{F}$ of boundedly-overlapping boxes $\theta$ of the form above covering $\mathcal{N}_\delta(\Sigma)$, one can form an associated partition of unity by smooth functions $\chi_\theta$ and consequently decompose
\[ g = \sum_{\theta \in \mathcal{F}} g_\theta := \sum_{\theta \in \mathcal{F}} g \chi_\theta. \]
By the Uncertainty Principle, $|\widehat{g_\theta}|$ (that is, $|E^{\delta}_{\Sigma} \,g_\theta|$) is approximately constant on any translate of the box dual\footnote{Recall that given a parallelepiped $P$ in $\mathbb{R}^n$ centred at $0$, its dual $P^{\ast}$ is the parallelepiped
\makebox[\linewidth]{$\displaystyle P^{\ast} := \{ \bm{u} \in \mathbb{R}^n : |\bm{u} \cdot \bm{v}| \leq 1 \text{ for all } \bm{v} \in P\}. $}} to the box $\theta$, denoted $\theta^{\ast}$, which has dimensions $\sim \delta^{-1/2} \times \ldots \times \delta^{-1/2} \times \delta^{-1} \times \delta^{-1}$ and long directions spanning the same $2$-plane as $\bm{n}_1, \bm{n}_2$. Thus geometrically $\theta^{\ast}$ is roughly the intersection of a cube of sidelength $\sim\delta^{-1}$ with the $O(\delta^{-1/2})$-neighbourhood of a $2$-plane normal to $\Sigma$ at some point; we call these objects \emph{slabs} (of length $\delta^{-1}$ and thickness $\delta^{-1/2}$). Denote by $\mathcal{S}_{\theta}$ a collection of boundedly-overlapping copies of $\theta^{\ast}$ (i.e. slabs) that covers $\mathbb{R}^{d+2}$; then we can further partition each $\widehat{g_\theta}$ by localising it\footnote{This can only be done approximately, as it is good to keep the frequency localisation intact.} to every $S \in \mathcal{S}_{\theta}$, writing $\widehat{g_\theta} = \sum_{S \in \mathcal{S}_{\theta}} \widehat{g_\theta} \chi_S$. In this way we effectively resolve $E^{\delta}_{\Sigma} g$ into wavepackets that are frequency-supported on some box $\theta$, concentrated on a translate of $\theta^{\ast}$ and approximately constant (in magnitude) there.\par
To obtain Fourier Extension estimates, the strategy typically involves controlling the interactions between different wavepackets by various means; by the observations above, such control can be achieved by studying the overlap of slabs coming from different $\mathcal{S}_\theta$'s. The celebrated Bourgain-Guth argument (also referred to as Broad/Narrow analysis), originating in \cite{BourgainGuth}, employs precisely such a strategy to prove estimates of the form \eqref{eq:general_local_extension_estimate}. It is beyond the scope of this article to present the argument in any amount of detail, but we remark that it can take as input Kakeya-type inequalities, which are functionally of the form
\[ \Big\| \sum_{S \in \mathcal{S}} \mathbf{1}_S\Big\|_{L^r} \lesssim_r \delta ^{-\beta}, \]
where $\mathcal{S}$ is (for example) a collection of slabs containing a single element from each $\mathcal{S}_{\theta}$ and $\beta \geq 0$. Such inequalities can be deduced from estimates  \eqref{eq:generic_mixed_norm_estimate} via duality and discretisation, as the proof of the following corollary will show. In order to avoid technicalities, we work with some simpler slabs which are rescaled to have length $1$ and thickness $\delta$: using $A,B$ for $\nabla^2 Q_1, \nabla^2 Q_2$, for $x \in \mathbb{R}^d$ and $\xi \in [-1,1]^d$ we let $S_{\delta}(x,\xi)$ denote the slab
\[ S_{\delta}(x,\xi):= \{(y,s,t) \in \mathbb{R}^d \times [-1,1]^2 : |y - x + sA\xi + tB\xi| < \delta\} \]
(notice that this is indeed the $O(\delta)$-neighbourhood of the $2$-plane spanned by $\bm{n}_1(\xi), \bm{n}_2(\xi)$, intersected with a cube of sidelength $\sim 1$).
\begin{corollary}[Kakeya-type estimate]\label{corollary:Kakeya_2_slabs}
Let $Q_1,Q_2$ be quadratic forms on $\mathbb{R}^d$ and suppose that the operator $\mathcal{T}$ is $L^p \to L^q$ bounded. If $(x_j, \xi_j)_{j \in J}$ are points in $\mathbb{R}^d \times [-1,1]^d$ such that the $\xi_j$ are $\delta$-separated, we have 
\begin{equation}\label{eq:generic_Kakeya_estimate}
\Big\| \sum_{j \in J} a_j \mathbf{1}_{S_{\delta}(x_j,\xi_j)}\Big\|_{L^{p'}} \lesssim_{Q_1,Q_2,p,q} \delta^{-d + 2d/{q'}} \Big(\sum_{j \in J} |a_j|^{q'}\Big)^{1/{q'}}. 
\end{equation}
In particular, if $\Sigma(Q_1,Q_2)$ is well-curved, we have for every $\epsilon > 0$
\[ \Big\| \sum_{j \in J} \mathbf{1}_{S_{\delta}(x_j,\xi_j)}\Big\|_{L^{(d+4)/d}} \lesssim_{Q_1,Q_2,\epsilon} \delta^{d^2/(d+4) - \epsilon}(\# J)^{(d+2)/(d+4)}. \]
\end{corollary} 
In the next subsection we will provide an application of Corollary \ref{corollary:Kakeya_2_slabs} to a problem in Geometric Measure Theory.
\begin{remark}
It is well-known that by a standard randomisation argument it is possible to deduce estimates such as those encountered in Corollary \ref{corollary:Kakeya_2_slabs} from Fourier restriction estimates such as \eqref{eq:general_extension_estimate} (see for instance Section 22.3 of \cite{Mattila}). However, away from the restriction endpoint these estimates are not necessarily as efficient as those deduced from $L^p \to L^q$ bounds for the operator $\mathcal{T}$. To wit, using Christ's Fourier restriction estimate one can deduce the inequality
\[ \Big\| \sum_{j \in J} \mathbf{1}_{S_{\delta}(x_j,\xi_j)}\Big\|_{L^{(d+4)/d}} \lesssim_{Q_1,Q_2,\epsilon} \delta^{d^2/(d+4) - \epsilon}(\# J), \]
which is weaker than the one obtained in Corollary \ref{corollary:Kakeya_2_slabs}.
\end{remark}
\begin{proof}[Proof of Corollary \ref{corollary:Kakeya_2_slabs}]
From hypothesis we have by duality $\|\mathcal{T}^{\ast} g\|_{L^{p'}} \lesssim \|g\|_{L^{q'}}$, where the adjoint $\mathcal{T}^{\ast}$ is given by 
\[ \mathcal{T}^{\ast} g(y,s,t) = \int_{[-1,1]^d} g(y + sA\xi + tB\xi, \xi)\,d\xi. \]
The statement is a consequence of following simple fact: with $K := \|A\|+\|B\|$, we have
\[ \mathcal{T}^{\ast} (\mathbf{1}_{B(x,2\delta)} \mathbf{1}_{B(\xi,K^{-1}\delta)}) \gtrsim_{A,B} \delta^{d} \,\mathbf{1}_{S_{\delta}(x,\xi)}. \]
Taking $g(x,\xi) = \sum_{j \in J} a_j \mathbf{1}_{B(x_j,2\delta)}(x) \mathbf{1}_{B(\xi_j,K^{-1}\delta)}(\xi)$ and using the $\delta$-separation of the $\xi_j$, estimate \eqref{eq:generic_Kakeya_estimate} follows readily from the dual estimate above.\par
For the well-curved case, apply \eqref{eq:generic_Kakeya_estimate} with $a_j = 1$ and $(p,q)$ along the $2/q = 1/p$ line and arbitrarily close to endpoint $(p,q) = \big(\tfrac{d+4}{4}, \tfrac{d+4}{2}\big)$ (these are the estimates afforded by Theorem \ref{main_theorem}). Finally, interpolate with the trivial $ \big\| \sum_{j \in J} \mathbf{1}_{S_{\delta}(x_j,\xi_j)}\big\|_{L^{\infty}} \lesssim \# J$ estimate to upgrade the norm to an $L^{(d+4)/d}$ one (this costs us a $\delta^{-\epsilon}$ loss, since $\#J \lesssim \delta^{-d}$).
\end{proof}
The above discussion thus motivates the study of restricted $2$-plane transforms \eqref{eq:definition_restricted_2-plane_transform} in the context of the Fourier Restriction problem. We plan to pursue this connection further in the near future.
\subsection{$(n,k)$-Kakeya sets}
Kakeya sets are subsets of $\mathbb{R}^n$ that contain a unit segment in every possible direction; a Kakeya set of measure zero is usually called a Besicovitch set (such sets exist). The Kakeya conjecture in Geometric Measure Theory states that Besicovitch sets in $\mathbb{R}^n$ have necessarily Hausdorff dimension equal to $n$. More in general, $(n,k)$-Kakeya sets are subsets $E \subset \mathbb{R}^n$ such that for any $k$-dimensional subspace $V$ (or ``$k$-plane'') there exists an affine translate $V + p$ such that $B(p,1) \cap (V + p) \subset E$ (where $B(p,1)$ denotes a ball in $\mathbb{R}^n$ of radius $1$ centred at $p$); Kakeya sets then coincide with $(n,1)$-Kakeya sets. Analogously, a $(n,k)$-Besicovitch set is a $(n,k)$-Kakeya set of measure zero. Even the existence of $(n,k)$-Besicovitch sets for $k>1$ is an open problem, but it is generally believed that no such sets exist, as the numerology of the dimensions involved is not favourable -- and for some $(n,k)$ pairs this has indeed been proven. We direct the reader to Chapter 24 of \cite{Mattila} for details and an overview of the problem.\par 
In order to obtain a more favourable situation, one might restrict the directions of the $k$-planes to lie in a submanifold $\mathcal{G}$ of the Grassmannian\footnote{The manifold of all linear subspaces of $\mathbb{R}^n$ of dimension $k$.} $G(n,k)$ and define a $\mathcal{G}$-Kakeya set to be a set $E \subset \mathbb{R}^n$ such that for every $V \in \mathcal{G}$ there exists an affine translate $V + p$ such that $B(p,1) \cap (V + p) \subset E$. Some works exist in this direction -- see \cite{DOberlin06}, \cite{DOberlin14} and \cite{FraserHarrisKroon} for some general types of submanifolds. Heuristically however, the most favourable situation appears to be that in which $\mathcal{G}$ satisfies $\dim \mathcal{G}+k = n$. This was the approach taken by K.\@ Rogers in \cite{Rogers}, in which he considered $\mathcal{G}$-Kakeya sets for $\mathcal{G}$ a $d$-dimensional submanifold of $G(d+2,2)$, a case that is directly relevant to us. Indeed, the set
\[ N(Q_1,Q_2) := \{ \pi_{\xi} : \xi \in [-1,1]^d \}, \]
where 
\[ \pi_{\xi} := \operatorname{Span}\{(-\nabla Q_1(\xi),1,0), (-\nabla Q_2(\xi),0,1)\}, \]
is the set of $2$-planes that are normal to $\Sigma(Q_1,Q_2)$ at some point; under the very mild assumption $\ker \nabla^2 Q_1 \,\cap \,\ker \nabla^2 Q_2 = \{0\}$, this set is precisely a $d$-dimensional submanifold of $G(d+2,2)$. Rogers proved that when the submanifold $\mathcal{G}$ satisfies a certain curvature condition (akin to the Wolff axioms\footnote{See e.g. Definition 13.1 in \cite{KatzLabaTao}.}) and $d=1$ then a $\mathcal{G}$-Kakeya set has Hausdorff dimension $3$ (thus equal to the ambient dimension $d+2$), and when $d=2$ it has Hausdorff dimension at least $7/2$. Using Corollary \ref{corollary:Kakeya_2_slabs}, we can prove a similar statement for arbitrary $d\geq 2$ and Kakeya sets with respect to directions normal to surfaces $\Sigma(Q_1,Q_2)$.
\begin{proposition}[$N(Q_1,Q_2)$-Kakeya sets]\label{prop:N-Kakeya_sets}
Let $d\geq 2$ and let $Q_1,Q_2$ be quadratic forms on $\mathbb{R}^d$ with the property that the polynomial $\det(s\nabla^2 Q_1 + t \nabla^2 Q_2)$ does not vanish identically. If $E$ is a $N(Q_1,Q_2)$-Kakeya set in $\mathbb{R}^{d+2}$, then
\[ \dim_{H} E \geq \frac{d+4}{2}. \]
\end{proposition}
\begin{proof}
We will present the argument for Minkowski dimension for simplicity of exposition -- the extension of the proof to Hausdorff dimension follows a standard argument that can be found in Section 4 of \cite{Rogers}.\par
Let $(\xi_j)_{j \in J}$ be a maximal collection of $\delta$-separated points in $[-1,1]^d$ and let $(x_j)_{j \in J}$ be arbitrary points in $\mathbb{R}^d$. It will suffice to show that to cover 
\[ E_{\delta} := \bigcup_{j \in J} S_{\delta}(x_j, \xi_j) \]
one needs at least $\gtrsim \delta^{-(d+4)/2}$ balls of radius $\delta$. Observe that 
\[ \sum_{j \in J} |S_{\delta}(x_j, \xi_j)| \sim \delta^d \# J \sim 1, \]
and therefore by H\"{o}lder inequality
\[ |E_\delta|^{1/p} \Big\| \sum_{j \in J} \mathbf{1}_{S_{\delta}(x_j,\xi_j)}\Big\|_{L^{p'}} \gtrsim 1. \]
Since $\det(s\nabla^2 Q_1 + t \nabla^2 Q_2)$ does not vanish, Theorems \ref{main_theorem} and \ref{thm:flat_surfaces} show that $\mathcal{T}$ is $L^p \to L^q$ bounded for some non-trivial $(p,q)$ on the line $2/q = 1/p$. Applying Corollary \ref{corollary:Kakeya_2_slabs} with any such estimate (and taking $a_j=1$ in \eqref{eq:generic_Kakeya_estimate}) we obtain after some rearrangement
\[ |E_{\delta}| \gtrsim \delta^{d/2}, \]
which implies the claim (since $|B_{d+2}(\delta)|\sim \delta^{d+2}$).
\end{proof}
It is natural to want to compare the curvature assumptions, and in particular to wonder whether all $N(Q_1,Q_2)$ submanifolds are curved in the sense of \cite{Rogers}. We claim that they are, under the hypotheses of Proposition \ref{prop:N-Kakeya_sets}. The curvature condition would be somewhat cumbersome to state in here, so we omit it; however, in our case it boils down to the condition that for every $V \in G(d+2,2)$ with $\dim V > 2$ one has
\[ \dim \{ \pi \in \mathcal{G} : \pi \subset V \} \leq \dim V - 2. \]
We will verify that this is the case when $\mathcal{G} = N(Q_1,Q_2)$. Let $\dim V = d + 2 - \ell$ and write $V = \{ \bm{x} \in \mathbb{R}^{d+2} : \bm{v}_1 \cdot \bm{x} = \ldots = \bm{v}_{\ell} \cdot \bm{x} = 0\}$ for some linearly independent $\bm{v}_1, \ldots, \bm{v}_{\ell}$ (the case $\ell = 0$ is trivial, so we can assume $\ell \geq 1$). Write $\bm{v}_j = (u_j, a_j, b_j) \in \mathbb{R}^d \times \mathbb{R} \times \mathbb{R}$ and observe that $\pi_{\xi} \subset V$ if and only if 
\[ A u_j \cdot \xi = a_j, \qquad B u_j \cdot \xi = b_j, \qquad \text{for all } j \in \{1, \ldots ,\ell\} \]
so that the dimension of $\{ \pi \in N(Q_1,Q_2) : \pi \subset V\}$ is the same as the dimension of the space of solutions to these equations. If the $u_1, \ldots, u_\ell$ are not linearly independent then the equations do not have a solution (as this would make the $\bm{v}_j$ linearly dependent as well); hence we can assume that they are linearly independent. Letting $U := \begin{pmatrix} u_1 & \cdots & u_\ell \end{pmatrix}$ we see that the dimension is bounded by $\dim \ker \begin{pmatrix} AU \\ BU \end{pmatrix}$. To show that this is $\leq \dim V - 2 = d - \ell$ it is equivalent to show that $\operatorname{rk}\begin{pmatrix} AU \\ BU \end{pmatrix} \geq \ell$; but by assumption there exists $(s,t)$ such that $\det(sA + tB) \neq 0$, and since $\operatorname{rk} U = \ell$ we see that $\operatorname{rk}(sA+tB)U = \ell$ and thus the rank condition is satisfied. This finishes the proof of the claim. 
\subsection{Mizohata-Takeuchi conjecture}
In this last motivational subsection we show how operators of the form \eqref{eq:definition_restricted_2-plane_transform} appear naturally in the context of the Mizohata-Takeuchi conjecture for surfaces of codimension 2.\par
The Mizohata-Takeuchi conjecture is a variant of the Fourier Restriction problem that concerns weighted $L^2$ estimates for the Fourier Extension operator (it originated in the study of dispersive and hyperbolic PDEs). For a hypersurface $\Sigma \subset \mathbb{R}^n$ with surface measure $d\sigma$ the conjecture takes the form
\[ \int_{\mathbb{R}^n} |E_{\Sigma}\,g(x)|^2 w(x) \,dx \lesssim \|Xw\|_{L^\infty} \int_{\Sigma} |g|^2 \,d\sigma, \]
where $X = T_{n,1}$ is the X-ray transform and $w$ is a non-negative function. The conjecture has been verified in the special case of $\Sigma = \mathbb{S}^{n-1}$ and weight $w$ radial, and this was done independently in \cite{CarberySoria} and \cite{BarceloRuizVega}; it can also be proven by the methods of \cite{CarberyRomeraSoria} but this was not realised at the time.\footnote{This was communicated to us by A.\@ Carbery.} The single-scale version of the result was treated in \cite{BarceloBennettCarbery}. The case of weights concentrated on a circle in the plane -- the opposite case to radial weights in some sense -- was treated in \cite{BennettCarberySoriaVargas}. The conjecture is otherwise open in all dimensions $n$, including in $n=2$, and the topic has been attracting increasing attention lately: see \cite{BennettNakamura} and \cite{BennettNakamuraShiraki} for some variants involving tomographic bounds (that is, bounds on objects such as $X(|E_{\Sigma}\,g|^2)$, where $X$ can later be transferred to the weight $w$ via the X-ray inversion formula); \cite{BezSugimoto} for connections with smoothing estimates; \cite{Shayya} for some results in $n=2$; \cite{CarberyIliopoulouWang} for a result for general $n$ but with a loss in the scale.\par
For surfaces of codimension other than $1$ one can generalise the conjecture as follows. For a submanifold $\mathcal{M} \subset \mathbb{R}^n$ of codimension $k$, denote by $N(\mathcal{M})$ the set of $k$-planes $\pi$ such that, for some point $p \in \mathcal{M}$, $\pi$ is orthogonal to $T_{p}\mathcal{M}$ (thus $N(\mathcal{M})$ is the set of normal directions of $\mathcal{M}$); then one conjectures that for every non-negative weight $w$
\[ \int_{\mathbb{R}^n} |E_{\mathcal{M}}\, g(x)|^2 w(x) \,dx \lesssim \sup_{\substack{\pi \in N(\mathcal{M}), \\ x \in \mathbb{R}^n}} |T_{n,k}w(\pi + x)| \int_{\mathcal{M}} |g|^2 \,d\sigma. \]
The first factor on the right-hand side is effectively the $L^\infty$ norm of the restriction of the $k$-plane transform $T_{n,k}$ to the set of normal directions to $\mathcal{M}$ -- which is precisely the same type of operator as \eqref{eq:definition_restricted_2-plane_transform}. We offer some modest evidence for this generalisation of the Mizohata-Takeuchi conjecture in all codimensions by proving the weak version stated in the  proposition below (which has a worse norm on the weight). We prelude some definitions: for $Q_1,\ldots, Q_k$ quadratic forms on $\mathbb{R}^d$ we let $\bm{Q}(\xi):= (Q_1(\xi), \ldots, Q_k(\xi))$; we denote by $\Sigma(\bm{Q})$ the compact quadratic surface of codimension $k$ in $\mathbb{R}^{d+k}$ parametrised by 
\[ \phi_{\bm{Q}}(\xi) := (\xi, \bm{Q}(\xi)), \qquad \xi \in [-1,1]^d. \]
We let 
\[ \mathcal{T}_{\bm{Q}}f(x,\xi) := \int_{\mathbb{R}^k} f(x - \nabla (\bm{s} \cdot \bm{Q})(\xi), \bm{s}) \,d\bm{s}, \]
where $\nabla = \nabla_{\xi}$ is applied componentwise, that is $\nabla (\bm{s} \cdot \bm{Q}) = \sum_{j=1}^{k} s_j \nabla Q_j$; notice that when $k = 2$ this is precisely the non-local version of operator \eqref{eq:definition_restricted_2-plane_transform}. This operator is pointwise comparable to the restriction of $T_{n,k}$ to directions normal to $\Sigma(\bm{Q})$. Finally, for simplicity we will work with the slightly modified Fourier Extension operator
\[ \mathsf{E}_{\Sigma(\bm{Q})}\, g(\bm{x}) := \int_{[-1,1]^d} g(\xi) e^{2\pi i \bm{x} \cdot \phi_{\bm{Q}}(\xi)} \,d\xi. \]
\begin{proposition}
Let $k \geq 1$ and let $\bm{Q} = (Q_1, \ldots, Q_k)$ be a vector of $k$ quadratic forms on $\mathbb{R}^d$. For every integrable weight $w : \mathbb{R}^{d+k} \to [0,\infty)$ we have\footnote{The mixed-norm is as in \eqref{eq:mixed_norm_definition}, that is, $L^{\infty}(L^2) = L^{\infty}_{\xi}(L^2_x)$.}
\begin{equation}\label{eq:weak_Mizohata_Takeuchi}
\int_{\mathbb{R}^{d+k}} |\mathsf{E}_{\Sigma(\bm{Q})}\, g(\bm{x})|^2 w(\bm{x}) \,d\bm{x} \lesssim \|\mathcal{T}_{\bm{Q}} w\|_{L^\infty (L^2)} \int_{[-1,1]^d} |g(\xi)|^2 \,d\xi 
\end{equation}
for every function $g \in L^2$.
\end{proposition}
\begin{proof}
We note the following Radon duality formula, which will be useful later:
\begingroup
\allowdisplaybreaks
\begin{equation}
\begin{aligned}
\mathcal{T}_{\bm{Q}}f(x,\xi) &= \int_{\mathbb{R}^k} \int_{\mathbb{R}^{d+k}} \widehat{f}(\eta, \bm{\alpha}) e^{2\pi i [\eta \cdot (x - \nabla(\bm{s} \cdot\bm{Q})(\xi)) + \bm{\alpha}\cdot \bm{s}]} \,d\eta\,d\bm{\alpha} \,d\bm{s} \\
&= \int_{\mathbb{R}^{d+k}} \widehat{f}(\eta, \bm{\alpha}) e^{2\pi i \eta \cdot x} \int_{\mathbb{R}^k} e^{2\pi i \bm{s} \cdot (\bm{\alpha} - \eta \cdot \nabla \bm{Q}(\xi))} \,d\bm{s} \,d\eta \,d\bm{\alpha} \\
&= \int_{\mathbb{R}^{d+k}} e^{2\pi i \eta \cdot x}\, \widehat{f}(\eta, \eta \cdot \nabla \bm{Q}(\xi)) \,d\eta,
\end{aligned} \label{eq:Radon_duality}
\end{equation}
\endgroup
where $\eta \cdot \nabla \bm{Q}(\xi) = (\eta \cdot \nabla Q_1(\xi), \ldots ,\eta \cdot \nabla Q_k(\xi))$.\par
Expanding the square in the left-hand side of \eqref{eq:weak_Mizohata_Takeuchi}, we have by Fubini
\begin{align*}
\int_{\mathbb{R}^{d+k}} |\mathsf{E}_{\Sigma(\bm{Q})}\, g(\bm{x})|^2 w(\bm{x}) \,d\bm{x} &= \iiint g(\eta) \overline{g(\xi)} e^{2\pi i \bm{x} \cdot (\phi_{\bm{Q}}(\eta) - \phi_{\bm{Q}}(\xi))} w(\bm{x}) \,d\eta \,d\xi \,d\bm{x} \\
&= \iint g(\eta) \overline{g(\xi)} \widehat{w}(\phi_{\bm{Q}}(\xi) - \phi_{\bm{Q}}(\eta)) \,d\eta \,d\xi.
\end{align*}
Now using the polarisation identity
\[ Q(\xi) - Q(\eta) = \frac{1}{2} (\xi - \eta) \cdot \nabla Q(\xi + \eta)  \]
we see by a change of variables that the last integral is equal to 
\[ \iint g\Big(\xi - \frac{\eta}{2}\Big) \overline{g\Big(\xi + \frac{\eta}{2}\Big)} \widehat{w}(\eta, \eta \cdot \nabla \bm{Q}(\xi)) \,d\xi \,d\eta. \] 
As $g$ is supported in $[-1,1]^d$ we see that we can insert in this expression a localisation factor $\mathbf{1}_{[-1,1]^d}(\xi)$ for free. By the Fourier inversion formula applied to $g, \overline{g}$ (which can be assumed to be Schwartz by a standard approximation argument) and a second change of variables we see that the expression can then be rearranged to be 
\[ \iint \bigg( \int \widehat{g}\Big(\frac{y}{2} - x\Big) \overline{\widehat{g} \Big( \frac{y}{2} + x\Big)} e^{2\pi i \xi \cdot y} \,dy\bigg) \Big( \int e^{2\pi i \eta \cdot x} \,\widehat{w}(\eta, \eta \cdot \nabla \bm{Q}(\xi)) \,d\eta\Big) \mathbf{1}_{[-1,1]^d}(\xi) \,d\xi \,dx. \]
In the second factor at the integrand we recognise $\mathcal{T}_{\bm{Q}} w (x,\xi)$ via the Radon duality formula \eqref{eq:Radon_duality}. For the first factor, define the bilinear operator\footnote{This operator is variously known as \emph{ambiguity function} (in Signal Processing) or as \emph{cross-Wigner distribution} (in Quantum Mechanics).} 
\[ W(F_1,F_2)(x,\xi) :=  \int F_1 \Big(\frac{y}{2} - x\Big) F_2\Big( \frac{y}{2} + x\Big) e^{2\pi i \xi \cdot y} \,dy; \]
then we see that the expression has become 
\[ \iint W\big(\widehat{g}, \overline{\widehat{g}}\big)(x,\xi) \mathcal{T}_{\bm{Q}}w(x,\xi) \mathbf{1}_{[-1,1]^d}(\xi) \,d\xi\,dx. \]
By two applications of Cauchy-Schwarz (and using the fact that $\xi$ is localised) this is bounded by 
\begin{align*}
&\leq \int \Big( \int \big|W\big(\widehat{g}, \overline{\widehat{g}}\big)(x,\xi)\big|^2 \,dx\Big)^{1/2} \Big( \int |\mathcal{T}_{\bm{Q}}w(x,\xi)|^2 \,dx\Big)^{1/2} \mathbf{1}_{[-1,1]^d}(\xi) \,d\xi \\
& \leq \| \mathcal{T}_{\bm{Q}}w\|_{L^\infty_{\xi} (L^2_x)} \int \Big( \int \big|W\big(\widehat{g}, \overline{\widehat{g}}\big)(x,\xi)\big|^2 \,dx\Big)^{1/2}\mathbf{1}_{[-1,1]^d}(\xi) \,d\xi \\
& \lesssim \| \mathcal{T}_{\bm{Q}}w\|_{L^\infty(L^2)} \big\|W\big(\widehat{g}, \overline{\widehat{g}}\big) \big\|_{L^2(L^2)}.
\end{align*}
Finally, by identifying $W(F_1, F_2)$ with a Fourier transform, we see by Plancherel that $\|W(F_1, F_2)\|_{L^2(L^2)} = \|F_1\|_{L^2} \|F_2\|_{L^2}$, so that by a further application of Plancherel we have $\big\|W\big(\widehat{g}, \overline{\widehat{g}}\big) \big\|_{L^2(L^2)} \leq \|g\|_{L^2}^2$. Inequality \eqref{eq:weak_Mizohata_Takeuchi} follows.
\end{proof}
\section{Affine Invariant Measures and GIT}\label{section:affine_curvature_GIT}
In this section we will briefly illustrate the construction of the affine invariant measures of Gressman \cite{Gressman2019} that are foundational to the definition of well-curvedness adopted here. In particular, we will explain how the non-vanishing of these measures is connected to the concept of semistability in Geometric Invariant Theory (abbreviated GIT, from here onwards).
\subsection{Construction of the affine invariant measure}
In order to keep things simple, we will describe Gressman's construction only in the context of surfaces of codimension 2. The construction here given can extend easily to surfaces of other sufficiently low codimension (see Remark \ref{remark:other_codimensions}, but for the most general construction we refer the reader to \cite{Gressman2019}.\par
The construction rests on two elements, the first being a lemma that allows one to construct a density from an arbitrary $m$-linear functional and the second being a choice of a suitable $m$-linear functional that captures curvature and enjoys affine invariance. We begin from the lemma, for which we introduce the following notation: letting $\Phi$ be a $m$-linear functional on the real finite-dimensional vector space $V$ (that is, $\Phi \in (V^{\ast})^{\otimes m}$), we denote by $\rho$ the action of the special linear group $SL(V)$ on $(V^{\ast})^{\otimes m}$ given by 
\begin{equation}\label{eq:rho_action}
(\rho_M \Phi)(\bm{v}_1, \ldots, \bm{v}_m) :=  \Phi(M^{\top}\bm{v}_1, \ldots, M^{\top}\bm{v}_m) 
\end{equation}
for any $M \in SL(V)$ and any $\bm{v}_j \in V$. For $(\bm{v}_1, \ldots, \bm{v}_d)$ an ordered choice of $d$ vectors in $V$ (where $d = \dim V$), we let 
\[ \| \Phi \|_{(\bm{v}_1, \ldots, \bm{v}_d)} := \| (\Phi(\bm{v}_{j_1}, \ldots, \bm{v}_{j_m}))_{j_1,\ldots, j_m \in \{1, \ldots, d\}}\|, \]
where $\|\cdot\|$ denotes an arbitrary norm on $\mathbb{R}^{dm}$ (say, the $\ell^2$ norm for the sake of fixing one). The lemma is then as follows.
\begin{lemma}[Prop. 1 of \cite{Gressman2019}]\label{lemma:density_construction}
Let $V$ be a real vector space with $d = \dim V$ and let $\Phi \in (V^{\ast})^{\otimes m}$ be a $m$-linear functional on $V$. Then there is a constant $c_{\Phi} \geq 0$ such that for every $\bm{v}_1,\ldots, \bm{v}_d$
\[ \inf_{M \in SL(V)} \| \rho_M \Phi\|_{(\bm{v}_1, \ldots, \bm{v}_d)}^{d/m} = c_{\Phi} |\det\begin{pmatrix}
\bm{v}_1 & \cdots & \bm{v}_d
\end{pmatrix}|. \]
\end{lemma}
The lemma comes with the important caveat that the constant $c_\Phi$ could vanish (this will correspond to the surface being ``flat'' at a point).\par
The multi-linear functional to which Lemma \ref{lemma:density_construction} will be applied is called the \emph{Affine Curvature Tensor} and in the case of surfaces of codimension 2 it is defined as follows. Let $\phi : \Omega \to \mathbb{R}^{d+2}$ be an embedding of a $d$-dimensional manifold into $\mathbb{R}^{d+2}$ and for a fixed $p \in \Omega$ consider vector fields $X_1,\ldots, X_{d}, Y_1, Y_2, Z_1, Z_2$ defined in a neighbourhood of $p$. Then we define the Affine Curvature Tensor $\mathcal{A}_{p}^{\phi}$ to be
\begin{align*}
&\mathcal{A}_{p}^{\phi}(X_1,\ldots, X_{d}, Y_1, Y_2, Z_1, Z_2) \\
& \hspace{1em} := \det \begin{pmatrix}
X_1 \phi(p) & \ldots & X_d \phi(p) & Y_1 Y_2 \phi(p) & Z_1 Z_2 \phi(p)
\end{pmatrix}.
\end{align*} 
It can be shown that $\mathcal{A}_{p}^{\phi}$ is indeed a tensor, in the sense that its value depends only on the value of the vector fields at $p$ (see Prop. 2 of \cite{Gressman2019}); therefore $\mathcal{A}_{p}^{\phi}$ can be identified with an element of $((T_p \Omega)^{\ast})^{\otimes (d+4)}$, that is, with a $(d+4)$-linear functional on the tangent space at $p$. Heuristically, the Affine Curvature Tensor probes the Taylor expansion of $\phi$ around any given point (hence the second derivatives $Y_1 Y_2 \phi$ and $Z_1 Z_2 \phi$ in the definition, which detect the quadratic terms). It has moreover the important property of being \emph{equi-affine invariant}, meaning that if $T$ is any affine transformation of $\mathbb{R}^{d+2}$ that preserves volumes, then we have $\mathcal{A}_{p}^{T\circ \phi} = \mathcal{A}_{p}^{\phi}$.\par
Combining Lemma \ref{lemma:density_construction} with the Affine Curvature Tensor one can then construct a surface measure on $\Sigma = \phi(\Omega)$ as follows. Define first of all the density 
\[ \delta_{\mathcal{A}}^{p} (X_1,\ldots, X_d) := \inf_{M \in SL(T_p \Omega)} \|\rho_M \mathcal{A}_{p}^{\phi}\|_{(X_1,\ldots,X_d)}^{d/(d+4)}. \]
Then one can define the surface measure $\nu_{\Sigma}$ via push-forward: for a ball $B \subset \mathbb{R}^d$ and a coordinate chart $\varphi : B \to \Omega$, we let
\[ \int_{\varphi(B)} g \,d\mu_{\mathcal{A}} := \int_{B} g(\varphi(y)) \,\delta_{\mathcal{A}}^{\varphi(y)}(d\varphi(\partial_{y_1}), \ldots, d\varphi(\partial_{y_d}))\,dy_1 \ldots \,dy_d; \]
finally, we define the \emph{affine invariant surface measure} $\nu_{\Sigma}$ by 
\[ \int_{\Sigma} f \,d\nu_{\Sigma} := \int_{\Omega} f \circ \phi \,d\mu_{\mathcal{A}}. \]
By Lemma \ref{lemma:density_construction}, the definition of $\mu_{\mathcal{A}}$ is consistent on overlapping charts, giving a measure on the whole $\Omega$ (and thus on the whole $\Sigma$); moreover, it is not hard to see that the definition is independent of the particular embedding and that $\nu_{\Sigma}$ inherits the equi-affine invariance of $\mathcal{A}_{p}^{\phi}$. 
\begin{remark}\label{remark:other_codimensions}
The construction above is readily extended to $d$-dimensional submanifolds of $\mathbb{R}^{d+r}$ such that the codimension satisfies $r \leq \frac{d(d+1)}{2}$. Indeed, it suffices to modify the affine curvature tensor to be 
\begin{align*}
&\mathcal{A}_{p}^{\phi}(X_1,\ldots, X_{d}, Y_1, Z_1, \ldots , Y_r, Z_r) \\
& \hspace{1em} := \det \begin{pmatrix}
X_1 \phi(p) & \ldots & X_d \phi(p) & Y_1 Z_1 \phi(p) & \cdots & Y_r Z_r \phi(p)
\end{pmatrix};
\end{align*}
then the density $ \delta_{\mathcal{A}}^{p} $ is given by 
\[  \delta_{\mathcal{A}}^{p} (X_1,\ldots, X_d) := \inf_{M \in SL(T_p \Omega)} \|\rho_M \mathcal{A}_{p}^{\phi}\|_{(X_1,\ldots,X_d)}^{d/(d+2r)} \]
and the rest of the construction is the same. The codimension condition $r \leq \frac{d(d+1)}{2}$ has to do with the Taylor expansion of $\phi$ and in particular with the fact that there are exactly $\frac{d(d+1)}{2}$ monomials of degree $2$ in $d$ many variables; to deal with higher codimensions yet, the tensor $\mathcal{A}_{p}^{\phi}$ needs to be modified by introducing derivatives of progressively higher orders. The fully general construction is presented in \cite{Gressman2019}.
\end{remark}
The case in which we are interested is $\phi(\xi) = (\xi, Q_1(\xi), Q_2(\xi))$ (with $\Omega = [-1,1]^d$); we see then that the measure $\nu_{\Sigma}$ on $\Sigma = \Sigma(Q_1,Q_2)$ is given by 
\[ \int_{\Sigma(Q_1,Q_2)} f \,d\nu_{\Sigma} = \int_{[-1,1]^d} f(\phi(\xi))\, \delta^{\xi}_{\mathcal{A}}(\partial_1, \ldots, \partial_d) \,d\xi. \]
Let us write $(M \partial)_j := M^{\top} \partial_j$; thus if $M_{ij}$ denotes the $(i,j)$-entry of $M$, we have $(M\partial)_j = \sum_{k=1}^{d} M_{jk}\partial_k$. Expanding the definitions, we have for the density $d\nu_{\Sigma}/d\xi$
\begingroup
\allowdisplaybreaks	
\begin{align*}
&\frac{d \nu_{\Sigma}}{d\xi} = \delta^{\xi}_{\mathcal{A}}(\partial_1, \ldots, \partial_d) \\
& = \bigg[ \inf_{M \in SL(\mathbb{R}^d)} \Big(\sum_{\substack{i_1, \ldots, i_{d} \\ j_1, j_2, k_1,k_2}} \big|\det \big(\begin{matrix}
(M \partial)_{i_1} \phi(\xi) & \cdots & (M\partial)_{i_d}\phi(\xi) \end{matrix} \\
& \hspace{11em} \begin{matrix} (M\partial)_{j_1}(M\partial)_{j_2} \phi(\xi) & (M\partial)_{k_1}(M\partial)_{k_2} \phi(\xi)
\end{matrix}\big)\big|^2 \Big)^{1/2} \bigg]^{d/(d+4)}.
\end{align*}
\endgroup
The expression simplifies significantly due to the special form of $\phi$. Indeed, observe that the first $d$ components of $(M\partial)_i \phi$ are simply the $i$-th column of $M^{\top}$, and the first $d$ components of $(M\partial)_{j_1}(M\partial)_{j_2} \phi$ are identically zero; therefore the determinant vanishes unless $i_1, \ldots, i_d$ is a permutation of $1,\ldots, d$. Since $\det M = 1$ we obtain for the sum of determinants in the last expression
\begingroup
\allowdisplaybreaks
\begin{align*}
& \sum_{\substack{i_1, \ldots, i_{d} \\ j_1, j_2, \\ k_1,k_2}} \begin{vmatrix}
(M \partial)_{i_1} \phi(\xi) & \cdots & (M\partial)_{i_d}\phi(\xi) & (M\partial)_{j_1}(M\partial)_{j_2} \phi(\xi) & (M\partial)_{k_1}(M\partial)_{k_2} \phi(\xi)
\end{vmatrix}^2 \\
& = d! \sum_{j_1,j_2,k_1,k_2} \begin{vmatrix}
(M\partial)_{j_1}(M\partial)_{j_2} Q_1(\xi) & (M\partial)_{k_1}(M\partial)_{k_2} Q_1(\xi) \\
(M\partial)_{j_1}(M\partial)_{j_2} Q_2(\xi) & (M\partial)_{k_1}(M\partial)_{k_2} Q_2(\xi)
\end{vmatrix}^2.
\end{align*} 
\endgroup
\begin{remark}
When $Q_1,Q_2$ are quadratic forms the last expression is clearly independent of $\xi$ and thus we see that $d\nu_{\Sigma}/d\xi$ is a constant, as claimed in Section \ref{section:main_results}. According to Definition \ref{defn:well_curved} the surface $\Sigma(Q_1,Q_2)$ is well-curved if this constant is non-zero, and flat otherwise.
\end{remark}
The expression can be massaged further: it is immediate that 
\[ (M\partial)_{j}(M\partial)_{k} Q_i(\xi) = (M \nabla^2 Q_i(\xi) M^{\top})_{j,k}, \]
and therefore the sum above coincides with 
\[  \sum_{j_1,j_2,k_1,k_2} \begin{vmatrix}
(M \nabla^2 Q_1 M^{\top})_{j_1, j_2} & (M \nabla^2 Q_1 M^{\top})_{k_1, k_2} \\
(M\nabla^2 Q_2 M^{\top})_{j_1, j_2} & (M \nabla^2 Q_2 M^{\top})_{k_1, k_2}
\end{vmatrix}^2. \]
We can summarise the above as follows. Using again $A,B$ in place of $\nabla^2 Q_1, \nabla^2 Q_2$, define the quadrilinear functional
\[ \mathscr{A}_{A,B}(Y_1, Y_2; Z_1, Z_2) := \begin{vmatrix}
\langle A Y_1, Y_2 \rangle  & \langle A Z_1, Z_2 \rangle \\
\langle B Y_1, Y_2 \rangle  & \langle B Z_1, Z_2 \rangle 
\end{vmatrix} \]
and notice that $\rho$ given by \eqref{eq:rho_action} acts on this functional by 
\[ \rho_M \mathscr{A}_{A,B} = \mathscr{A}_{MAM^{\top}, MBM^{\top}}. \]
Then the density of $\nu_{\Sigma}$ for $\Sigma = \Sigma(Q_1,Q_2)$ is given by 
\[ \frac{d\nu_{\Sigma}}{d\xi} = c_d \inf_{M \in SL(\mathbb{R}^d)} \| \mathscr{A}_{MAM^{\top}, MBM^{\top}}\|_{\partial}^{d/(d+4)}, \]
where $c_d$ is an absolute constant and we have shortened $\|\cdot\|_{\partial} := \|\cdot\|_{(\partial_1, \ldots, \partial_d)}$. The fact that this quantity depends only on the Hessians has been made explicit. The reparametrisation and equi-affine invariances have also been made explicit in the following way: firstly, it is obvious that the density, as a function of the Hessians $A,B$, is invariant with respect to the ``reparametrisation action'' of $SL(\mathbb{R}^d)$ on pairs of symmetric matrices given by (with a little abuse of notation)
\begin{equation}\label{eq:rho_action_AB}
\rho_M(A,B) := (MAM^{\top}, MBM^{\top}). 
\end{equation}
Secondly (and slightly less obviously), the density is also invariant as a function of $A,B$ with respect to the action $\sigma$ of $SL(\mathbb{R}^2)$ given by 
\begin{equation}\label{eq:sigma_action_AB}
\text{ for } N = \begin{pmatrix}
\lambda & \mu \\ \lambda' & \mu'
\end{pmatrix} \in SL(\mathbb{R}^2), \qquad \sigma_N (A,B) := (\lambda A + \mu B, \lambda'A + \mu' B); 
\end{equation}
this is a consequence of the fact that $\mathscr{A}_{\cdot, \cdot}$ itself is $\sigma$-invariant, as can be seen by a straightforward calculation. These observations about invariances lead us directly into the next subsection.
\subsection{Connection to GIT}
GIT is the branch of Algebraic Geometry that studies group actions on algebraic varieties (of which vector spaces are a particularly simple instance); it provides a way to construct well-behaved quotient spaces via the study of polynomials that are invariant under these actions. One of the deepest insights of \cite{Gressman2019} is the realisation that the non-vanishing of the affine invariant surface measure is equivalent to the concept of semistability in GIT. Below we will explain this connection, limiting ourselves to the bare minimum of theory in order not to encumber the exposition.\par 
We dive right in by stating a lemma from \cite{Gressman2019} that connects the density $d\nu_{\Sigma}/d\xi$ to certain invariant polynomials; the statement will be customised to our particular situation. Recall that the quadrilinear form $\mathscr{A}_{A,B}$ is an element of the vector space of quadrilinear functionals on $\mathbb{R}^d$, that is $V := ((\mathbb{R}^d)^{\ast})^{\otimes 4}$, and that the $\rho$ action given by \eqref{eq:rho_action} is defined over the whole of $V$. A real polynomial $P$ on $V$ (that is, a polynomial in the coefficients of elements $\mathscr{A} \in V$) is $\rho$-\emph{invariant} if for every $\mathscr{A} \in V$ and every $M \in SL(\mathbb{R}^d)$
\[ P (\rho_M \mathscr{A}) = P(\mathscr{A}). \]
These invariant polynomials form a ring, which is moreover finitely generated (this is a celebrated theorem of Hilbert \cite{Hilbert}). It turns out that one can estimate the density via any set of homogeneous generators\footnote{It is easy to see that, since $\rho$ commutes with dilations, given any set of generators one can form a set of homogeneous generators.} of the invariant polynomials.
\begin{lemma}[Lemma 2 of \cite{Gressman2019}]\label{lemma:Gressman_GIT_lemma}
Let $P_1, \ldots, P_N$ be homogeneous polynomials on $V = ((\mathbb{R}^d)^{\ast})^{\otimes 4}$ that generate the ring of $\rho$-invariant polynomials. Then for every $\mathscr{A} \in V$ we have
\[ \inf_{M \in SL(\mathbb{R}^d)} \| \rho_M \mathscr{A} \|_{\partial} \sim \max_{j \in \{1,\ldots,N\}} |P_{j}(\mathscr{A})|^{1/{\deg P_j}}. \]
\end{lemma}
By taking $\mathscr{A} = \mathscr{A}_{A,B}$ the left-hand side becomes (a multiple of) $(d\nu_{\Sigma}/d\xi)^{(d+4)/d}$, so that the lemma provides a way to estimate the density in terms of the generators via the expression at the right-hand side. The implicit constants depend on the choice of generators.\par
In proving the characterisation of well-curvedness given by Theorem \ref{thm:well_curvedness_characterisation}, we will make use of a straightforward consequence of Lemma \ref{lemma:Gressman_GIT_lemma}. Let $\operatorname{Sym}^2(\mathbb{R}^d)$ denote the space of real symmetric $d \times d$ matrices; then the actions $\rho, \sigma$, given by \eqref{eq:rho_action_AB}, \eqref{eq:sigma_action_AB} respectively, combine into an action of $SL(\mathbb{R}^d) \times SL(\mathbb{R}^2)$ on $\operatorname{Sym}^2(\mathbb{R}^d) \times \operatorname{Sym}^2(\mathbb{R}^d)$ denoted $\rho \times \sigma$ and given by 
\[ (\rho \times \sigma)_{M,N}(A,B) := \rho_{M}(\sigma_{N}(A,B)) \]
for any $M \in SL(\mathbb{R}^d), N \in SL(\mathbb{R}^2)$ (observe that $\rho$ and $\sigma$ commute, so the order is inconsequential). We say that a polynomial $Q$ on $\operatorname{Sym}^2(\mathbb{R}^d) \times \operatorname{Sym}^2(\mathbb{R}^d)$ is $(\rho \times \sigma)$-invariant if for every pair of real symmetric matrices $A,B$ and every $M \in SL(\mathbb{R}^d), N \in SL(\mathbb{R}^2)$
\[ Q((\rho \times \sigma)_{M,N}(A,B)) = Q(A,B). \]
The lemma we will use is then the following.
\begin{lemma}\label{lemma:GIT_equivalence}
Let $Q_1,Q_2$ be quadratic forms on $\mathbb{R}^d$, with associated surface $\Sigma = \Sigma(Q_1,Q_2)$, and let $A,B$ be the Hessians $\nabla^2 Q_1, \nabla^2 Q_2$ respectively. Then the density $d\nu_{\Sigma}/ d\xi$ is non-zero if and only if there exists a $(\rho \times \sigma)$-invariant polynomial $Q$ on $\operatorname{Sym}^2(\mathbb{R}^d) \times \operatorname{Sym}^2(\mathbb{R}^d)$ such that $Q(0,0)=0$ but
\[ Q(A,B) \neq 0. \]
\end{lemma}
We point out that since the density is defined pointwise the lemma extends to arbitrary surfaces parametrised by $(\xi, \varphi_1(\xi), \varphi_2(\xi))$ -- just replace $A,B$ with the Hessians of $\varphi_1, \varphi_2$ at the desired point.
\begin{proof}
By Lemma \ref{lemma:Gressman_GIT_lemma}, if the density $d\nu_{\Sigma}/ d\xi$ is non-zero then there exists a $\rho$-invariant homogeneous polynomial $P$ on $V = ((\mathbb{R}^d)^{\ast})^{\otimes 4}$ such that $P(\mathscr{A}_{A,B}) \neq 0$; but since $\mathscr{A}_{X,Y}$ is $\sigma$-invariant, we see that the polynomial $Q(X,Y) := P(\mathscr{A}_{X,Y})$ is $(\rho\times \sigma)$-invariant, $Q(0,0)=0$ and $Q(A,B) \neq 0$.\par
Conversely, assume that there exists such a polynomial $Q(X,Y)$ as per the statement. The polynomial is in particular $\sigma$-invariant, and it is a well-known fact that $\sigma$-invariant polynomials are generated by determinants 
\[ \begin{vmatrix}
X_{j_1, j_2} & X_{k_1,k_2} \\
Y_{j_1, j_2} & Y_{k_1,k_2}
\end{vmatrix} \]
(this is known as the First Fundamental Theorem for $SL(2)$-invariants, see for example Chapter II of \cite{GraceYoung}). However, the above is nothing but the coefficient $\mathscr{A}_{A,B}(\partial_{j_1}, \partial_{j_2}; \partial_{k_1}, \partial_{k_2})$, and therefore there exists some polynomial $P$ such that $Q(X,Y) = P(\mathscr{A}_{X,Y})$ for all $(X,Y) \in \operatorname{Sym}^2(\mathbb{R}^d) \times \operatorname{Sym}^2(\mathbb{R}^d)$. Since $Q$ is also $\rho$-invariant, we see that 
\begin{equation}\label{eq:partial_rho_invariance}
P(\rho_{M} \mathscr{A}_{X,Y}) = P(\mathscr{A}_{X,Y}). 
\end{equation}
Assume now by way of contradiction that $d\nu_{\Sigma}/d\xi = 0$, which in particular means that $\inf_{M \in SL(\mathbb{R}^d)} \| \rho_{M} \mathscr{A}_{A,B}\|_{\partial} = 0$. Thus there exists a sequence $(M_k)_{k \in \mathbb{N}} \subset SL(\mathbb{R}^d)$ such that $ \| \rho_{M_k} \mathscr{A}_{A,B}\|_{\partial} \to 0$ as $k \to \infty$; in particular, every component of $\rho_{M_k}\mathscr{A}_{A,B}$ tends to zero. By \eqref{eq:partial_rho_invariance} this implies by continuity that $Q(A,B) = P(\mathscr{A}_{A,B})=0$, but this is a contradiction.
\end{proof}
The existence of a non-constant invariant polynomial that does not vanish on $(A,B)$ is equivalent, in GIT language, to $(A,B)$ being semistable. More precisely, consider an affine variety $\mathscr{C}$ given as the zero set of a finite collection of homogeneous polynomials; observe that $0 \in \mathscr{C}$ and that if $x \in \mathscr{C}$ then $\lambda x \in \mathscr{C}$ for every $\lambda \in \mathbb{R}$. We will call $\mathscr{C}$ a \emph{cone}. Given an action $\theta : G \times \mathscr{C} \to \mathscr{C}$ of a linearly reductive algebraic group $G$ on the cone $\mathscr{C}$, and assuming that the action commutes with dilations,\footnote{Any such action is always assumed to be algebraic, in the sense that there exist embeddings of $G, \mathscr{C}$ as affine varieties in affine spaces such that the action is given by a polynomial map in the resulting affine coordinates.} a point $x \in \mathscr{C}$ is said to be $\theta$-\emph{semistable} if 
\[ 0 \not\in \operatorname{Cl_{Zar}}(\{ \theta_g(x) : g \in G\}), \]
that is, if $0$ is not contained in the Zariski closure of the orbit of $x$; else the point is called $\theta$-\emph{unstable}. Notice that semistability is a property of the orbit and not of the particular point. It is immediate to see that if there exists a $\theta$-invariant polynomial $P$ such that $P(x) \neq 0$ then $x$ is $\theta$-semistable; the opposite implication is also true but non-trivial, and is the content of the so-called Fundamental Theorem of GIT (see Theorem 1.1 in Chapter 1, Section 2 of \cite{MumfordFogartyKirwan} or Section 3.4.1 of \cite{Wallach}). Thus we have the equivalent definition of semistability: $x \in \mathscr{C}$ is $\theta$-semistable if and only if there exists a $\theta$-invariant polynomial $P$ on $\mathscr{C}$ such that $P(0)=0$ but $P(x) \neq 0$. 
\begin{remark}
Effectively, we could have simply defined semistability in terms of non-vanishing invariant polynomials. However, in the next section we will need to use tools from GIT that are better phrased in terms of orbits, and therefore decided to provide here the more standard definition of semistability.
\end{remark}
Since $\operatorname{Sym}^2(\mathbb{R}^d) \times \operatorname{Sym}^2(\mathbb{R}^d)$ is a cone and $SL(\mathbb{R}^d) \times SL(\mathbb{R}^2)$ is a linearly reductive group, we can rephrase Lemma \ref{lemma:GIT_equivalence} informally as 
\begin{quote}
\emph{$d\nu_{\Sigma}/d\xi$ is non-zero if and only if $(A,B)$ is $(\rho\times\sigma)$-semistable.}
\end{quote}
\section{Characterisation of well-curvedness}\label{section:semistability}
In this section we will provide the following algebraic characterisation of the semistability of a pair of symmetric matrices $(A,B)$ under the $\rho\times \sigma$ action introduced in the previous section. 
\begin{proposition}\label{prop::semistability_SL2xSLd_det}
Let $A,B \in \operatorname{Sym}^2(\mathbb{R}^d)$. The pair $(A,B)$ is $(\rho\times\sigma)$-semistable if and only if the homogeneous polynomial $s,t \mapsto \det(sA + tB)$ does not vanish identically and has no root of multiplicity $>d/2$.
\end{proposition}
Together with Lemma \ref{lemma:GIT_equivalence}, this proposition immediately implies Theorem \ref{thm:well_curvedness_characterisation}, as the root condition above is precisely condition \eqref{condition:well-curvedness} when $(A,B) = (\nabla^2 Q_1, \nabla^2 Q_2)$. The rest of the section is dedicated to the proof of the proposition, which is articulated in three subsections.
\subsection{Preliminaries} 
In the proof of Proposition \ref{prop::semistability_SL2xSLd_det} we will make use of a fundamental GIT result -- the so-called Hilbert-Mumford criterion, which provides a characterisation of semistable/unstable points. The classical Hilbert-Mumford criterion (like much of GIT) is formulated over the complex numbers: this means that below $\mathscr{C}$ is an affine variety in some $\mathbb{C}^n$ and $G$ is an algebraic subgroup\footnote{An algebraic subgroup of $GL(\mathbb{C}^n)$ is a subgroup that is also a subvariety of $GL(\mathbb{C}^n)$.} of $GL(\mathbb{C}^n)$.
\begin{lemma}[Hilbert-Mumford criterion]\label{lemma:Hilbert_Mumford_criterion}
Let $\mathscr{C}$ be a cone and let $\theta : G \times \mathscr{C} \to \mathscr{C}$ be the action of a linearly reductive group $G$, which we assume commutes with dilations. If $x \in \mathscr{C}$ is $\theta$-unstable, then there exists a one-parameter subgroup of $G$ given by an algebraic homomorphism $\eta : \mathbb{C}^{\times} \to G$ such that 
\[ \lim_{\lambda \to 0} \theta_{\eta(\lambda)}(x) = 0, \]
where the limit is taken in the standard topology of $\mathscr{C}$ (the one inherited from the standard topology of $\mathbb{C}^n$).
\end{lemma}
The real version of the Hilbert-Mumford criterion is due to Birkes \cite{Birkes}: its statement is exactly the same, but $\mathbb{C}$ is replaced everywhere by $\mathbb{R}$. An easy consequence of the real Hilbert-Mumford criterion is that $x \in \mathscr{C}$ is $\theta$-semistable if and only if it is semistable for the complexification of $\theta$ (which entails complexifying $\mathscr{C}, G$ as well). Indeed, if $x$ is $\theta$-unstable then by the real Hilbert-Mumford criterion $0$ is in the standard closure of the orbit of $x$, and therefore $0$ is also in the Zariski closure of the orbit under the complexified action; viceversa, if $x$ is $\theta$-semistable then for some $\theta$-invariant polynomial $P$ such that $P(0)=0$ we have $P(x)\neq 0$, but $P$ is also invariant with respect to the complexified action.\par
For us the above means that a pair of real symmetric matrices $(A,B)$ is semistable under the action $\rho \times \sigma$ of $SL(\mathbb{R}^d) \times SL(\mathbb{R}^2)$ if and only if it is semistable under the same action of group $SL(\mathbb{C}^d) \times SL(\mathbb{C}^2)$ instead. This will afford us some convenient technical simplifications later on, but is by no means necessary.
\begin{remark}
Lemma \ref{lemma:Gressman_GIT_lemma} is a direct consequence of the real Hilbert-Mumford criterion.
\end{remark}
Let us denote
\[ \Delta_{A,B}(s,t) := \det(sA+tB) \]
for convenience; thus $\Delta$ can be regarded as a map $\operatorname{Sym}^2(\mathbb{C}^d) \times \operatorname{Sym}^2(\mathbb{C}^d) \to \mathbb{C}[s,t]$. Some observations about the symmetries enjoyed by this map are in order. The first observation is that $\Delta$ is invariant under the action $\rho$: indeed,
\[ \det(s MAM^\top + t MBM^\top) = \det(M(sA+tB)M^\top)=\det(sA+tB); \]
therefore
\[ \Delta_{\rho_M(A,B)} = \Delta_{A,B}. \]
The second observation is that $\Delta$ is not invariant under the action $\sigma$, but it is nevertheless equivariant: indeed,
\[ \det(s(\lambda A + \mu B) + t(\lambda'A + \mu' B)) = \det((\lambda s + \lambda' t)A + (\mu s + \mu' t)B), \]
so if we let $\widetilde{\sigma}$ denote the action on polynomials of two variables defined by 
\[ \widetilde{\sigma}_{N}P\begin{pmatrix} s \\ t \end{pmatrix} := P\left(N^\top \begin{pmatrix} s \\ t \end{pmatrix}\right)  \]
for any $N \in SL(\mathbb{C}^2)$, we have
\[ \Delta_{\sigma_N (A,B)} = \widetilde{\sigma}_{N} (\Delta_{A,B}). \]\par
We are of course only interested in the action of $\widetilde{\sigma}$ on homogeneous polynomials of two variables and degree $d$. It will be very useful to identify which polynomials are semistable under this action; we can do so very easily with the Hilbert-Mumford criterion. By Lemma \ref{lemma:Hilbert_Mumford_criterion}, $P \in \mathbb{C}[s,t]$ (homogeneous of degree $d$) will be $\widetilde{\sigma}$-unstable if and only if there exists a one-parameter subgroup $(N_\lambda)_{\lambda\in\mathbb{C}^{\times}}$ of $SL(\mathbb{C}^2)$ such that
\[ \lim_{\lambda \to 0} \widetilde{\sigma}_{N_\lambda}P =0, \]
where the limit is taken in the standard vector space topology of $\mathbb{C}[s,t]$. The one-parameter (algebraic) subgroups of the special linear groups $SL(\mathbb{C}^n)$ are well-known: they are all of the form
\[ N_\lambda = G \begin{pmatrix} \lambda^{a_1} & & \\
 & \ddots & \\
 & & \lambda^{a_n}
\end{pmatrix} G^{-1}, \]
where $G \in SL(\mathbb{C}^n)$ and the exponents $a_j$ are integers that satisfy $\sum_{j=1}^{n} a_j = 0$ (but are otherwise unconstrained). In our case $n=2$, so the one-parameter subgroups are simply conjugates of $\begin{pmatrix} \lambda & \\ & \lambda^{-1}\end{pmatrix}$, and therefore if we let $\begin{pmatrix}
\hat{s} \\ \hat{t}
\end{pmatrix} = G^{-1} \begin{pmatrix}
s \\ t
\end{pmatrix}$ we can write
\[ \widetilde{\sigma}_{N_\lambda} P = \sum_{k=0}^{d} c_k \lambda^{2k-d} \hat{s}^{k} \hat{t}^{d-k}, \]
where the $c_k$ are the coefficients of $P \circ G$. This expression can only tend to zero as $\lambda \to 0$ if the coefficients $c_k$ vanish for all $k \leq d/2$; but this means in particular that $\hat{s}^m$ divides $P \circ G$ for some $m>d/2$, or in other words that $P$ has a root of multiplicity $>d/2$. The argument can be run in reverse, and therefore we have shown the following known fact.
\begin{lemma}\label{lemma:semistability_SL2}
Let $P$ be a homogeneous polynomial of degree $d$ in $\mathbb{C}[s,t]$. Then $P$ is $\widetilde{\sigma}$-semistable if and only if $P$ has no root of multiplicity $>d/2$.
\end{lemma}
In light of this lemma, we could rephrase Proposition \ref{prop::semistability_SL2xSLd_det} as 
\begin{quote}
\emph{$(A,B)$ is $(\rho\times\sigma)$-semistable if and only if $\Delta_{A,B}$ is $\widetilde{\sigma}$-semistable.}
\end{quote}\par
Now we are ready to begin the proof of Proposition \ref{prop::semistability_SL2xSLd_det}. One implication is easy: suppose that $(A,B)$ is $(\rho\times\sigma)$-unstable, and therefore by Lemma \ref{lemma:Hilbert_Mumford_criterion} there exists a one-parameter subgroup $\big((M_\lambda, N_\lambda)\big)_{\lambda \in \mathbb{C}^{\times}} \subset SL(\mathbb{C}^d) \times SL(\mathbb{C}^2)$ such that 
\[ \lim_{\lambda \to 0} \rho_{M_\lambda} \sigma_{N_{\lambda}} (A,B) = (0,0). \]
By the invariance of $\Delta$ under $\rho$ and equivariance under $\sigma$, we have then that 
\[ \lim_{\lambda \to 0} \widetilde{\sigma}_{N_\lambda} \Delta_{A,B} = 0, \]
that is, the polynomial $\Delta_{A,B}$ is $\widetilde{\sigma}$-unstable. By Lemma \ref{lemma:semistability_SL2} we have then that $\Delta_{A,B}$ has a root of multiplicity larger than $d/2$, thus proving one side of the equivalence.\par
It remains to prove the opposite implication: we will assume in the rest of the section that $\Delta_{A,B}$ has a root of multiplicity strictly larger than $d/2$, and show that this makes $(A,B)$ unstable. There is a relevant dichotomy here: either $\Delta_{A,B}$ is a non-vanishing polynomial in $s,t$ or it is identically zero. We treat each case on its own.
\subsection{Case I: $\Delta_{A,B}$ is not identically vanishing}
Since the determinant is non-vanishing, for some $(s_0,t_0)$ we have that $s_0 A + t_0 B$ is invertible. We may assume without loss of generality that $(s_0,t_0)=(0,1)$, or in other words that $\det B \neq 0$. Indeed, observe that if $s_0 \neq 0$ we can let 
\[ N_0 := \begin{pmatrix}
0 & -1/s_0 \\
s_0 & t_0
\end{pmatrix} \in SL_2(\mathbb{C}) \]
and we have 
\[ \sigma_{N_0}(A,B) = (-1/s_0 \, B, s_0 A + t_0 B); \]
$(A,B)$ is $(\rho\times\sigma)$-unstable if and only if the pair $(-1/s_0 \, B, s_0 A + t_0 B)$ is, and therefore it is just a matter of relabelling $A':= -1/s_0 \, B$, $B' := s_0 A + t_0 B$ in the arguments below.\par 
We can thus assume $\det B \neq 0$ and write 
\[ \det(sA + tB) = \det(B)\det(sAB^{-1} + t I). \]
We put $AB^{-1}$ in Jordan normal form: for any $r,\lambda$ denote by $J_r(\lambda)$ the $r\times r$ Jordan block of eigenvalue $\lambda$, that is
\[ J_r(\lambda) := \begin{pmatrix}
\lambda & 1 & & & \\
 & \lambda & 1 & & \\
 & & \ddots & \ddots & \\
 & & & \lambda & 1 \\
 & & & & \lambda \\
\end{pmatrix} \]
(if $r=1$ we have simply $J_1(\lambda) = \begin{pmatrix}\lambda \end{pmatrix}$); then there exists a matrix $Q \in GL(\mathbb{C}^d)$ such that $AB^{-1} = Q \bm{J} Q^{-1}$, where 
\[ \bm{J} = \begin{pmatrix}
\boxed{J_{r_1}(\lambda_1)} & & \\
 & \ddots & \\
  & & \boxed{J_{r_\ell}(\lambda_\ell)}
\end{pmatrix} \]
for some $r_j$ and $\lambda_j$. We have 
\[ \det(sAB^{-1} + t I) = \det(sQ \bm{J} Q^{-1} + t I) = \det(s \bm{J} + t I), \]
so that matters are reduced to the Jordan normal form of $AB^{-1}$. With $I_{r}$ denoting the $r\times r$ identity matrix, we have
\[ s \bm{J} + t I = \begin{pmatrix}
\boxed{s J_{r_1}(\lambda_1) + t I_{r_1}} & & \\
 & \ddots & \\
 & & \boxed{s J_{r_\ell}(\lambda_\ell) + t I_{r_\ell}} 
\end{pmatrix}, \]
where in particular 
\[ s J_{r_j}(\lambda_j) + t I_{r_j} = \begin{pmatrix}
s\lambda_j + t & s & & \\
 & \ddots & \ddots & \\
 & & s \lambda_j + t & s \\
 & & & s \lambda_j + t \\
\end{pmatrix}. \]
We then see that the above has produced the factorisation
\[ \det(sA + t B) = \det(B) \prod_{j=1}^{\ell} (s \lambda_j + t)^{r_j}; \]
we caution the reader that the $\lambda_j$ are not necessarily distinct and therefore the $r_j$ are not exactly the multiplicities. If we want to highlight the correct multiplicities, we let $\lambda^{\ast}_1, \ldots, \lambda^{\ast}_n$ be all the distinct values the $\lambda_j$ take and we write 
\[ \det(sA + tB) = \det(B) \prod_{j=1}^{n} (s\lambda^{\ast}_j + t)^{m_j}, \]
where 
\[ m_j = \sum_{k : \, \lambda_k = \lambda^{\ast}_j} r_k. \]
One of the $m_j$ is larger than $d/2$ by assumption -- let it be $m_1$ for convenience. Then we have deduced that $\bm{J}$, the Jordan form of $AB^{-1}$, has an eigenvalue that is repeated more than $d/2$ times. We will now see how to connect this fact to the original pair $(A,B)$ of symmetric matrices.\par 
Observe that every block $J_r(\lambda)$ can be written as the product of two symmetric matrices: indeed, if we let 
\begin{equation}\label{eq:antisymmetric_Jordan_blocks}
\widetilde{J}_r(\lambda):= \begin{pmatrix}
 & & & 1 & \lambda \\
 & & 1 & \lambda  & \\
 & \text{\reflectbox{$\ddots$}} & \text{\reflectbox{$\ddots$}} & & \\
 1 & \lambda & & &  \\
\lambda  & & & &  \\
\end{pmatrix}, \qquad 
\widetilde{I}_r := \begin{pmatrix}
 & & & & 1 \\
 & & & 1  & \\
 & & \text{\reflectbox{$\ddots$}} & & \\
 & 1 & & &  \\
1 & & & &  \\
\end{pmatrix}, 
\end{equation}
then it is immediate to verify that 
\[ J_r(\lambda) = \widetilde{J}_r(\lambda) \widetilde{I}_r. \]
We can therefore factorise 
\[ \bm{J} = \widetilde{\bm{J}} \widetilde{\bm{I}}, \]
where 
\begin{equation}\label{eq:special_form_A_B}
\widetilde{\bm{J}} = \begin{pmatrix}
\boxed{\widetilde{J}_{r_1}(\lambda_1)} & & \\
 & \ddots & \\
  & & \boxed{\widetilde{J}_{r_\ell}(\lambda_\ell)}
\end{pmatrix}, \qquad
\widetilde{\bm{I}} = \begin{pmatrix}
\boxed{\widetilde{I}_{r_1}} & & \\
 & \ddots & \\
  & & \boxed{\widetilde{I}_{r_\ell}}
\end{pmatrix}.  
\end{equation}
We claim that $(A,B)$ and $(\widetilde{\bm{J}}, \widetilde{\bm{I}})$ belong to the same $(\rho\times\sigma)$-orbit, and therefore they are either both unstable or both semistable. Indeed, since $B$ is invertible we can write 
\[ (A,B) = (AB^{-1} B, B) = (Q \bm{J} Q^{-1} B, B) = (Q \widetilde{\bm{J}} \widetilde{\bm{I}} Q^{-1} B, B); \]
since $B$ is also symmetric, acting with $\rho_{\mu B^{-1}}$ (where $\mu$ is such that $\det (\mu B^{-1}) =1$) we have that the orbit of $(A,B)$ contains
\[ \mu^2 \ (B^{-1}Q \widetilde{\bm{J}} \widetilde{\bm{I}} Q^{-1}, B^{-1}). \]
Acting with $\rho_{\mu' Q^{\top}}$ (where $\mu'$ is such that $\det (\mu' Q^{\top}) =1$) we see that 
\[ \mu^2 {\mu'}^2 \ (Q^{\top}B^{-1}Q \widetilde{\bm{J}} \widetilde{\bm{I}}, Q^{\top} B^{-1} Q) \]
is also in the orbit of $(A,B)$; moreover, since $\widetilde{\bm{I}}$ is symmetric and its own inverse, we have in the orbit of $(A,B)$ also the element 
\[ \mu^2 {\mu'}^2 {\mu''}^2 \ (\widetilde{\bm{I}} (Q^{\top}B^{-1}Q) \widetilde{\bm{J}}, \widetilde{\bm{I}} (Q^{\top}B^{-1}Q) \widetilde{\bm{I}}) \]
(where $\mu''$ is such that $\det(\mu'' \widetilde{\bm{I}})=1$). Letting $N := \mu {\mu'}^2 \mu'' \ \widetilde{\bm{I}} (Q^{\top}B^{-1}Q) \in SL(\mathbb{C}^d)$, we see that the last element is simply $\mu \mu''\ (N \widetilde{\bm{J}}, N \widetilde{\bm{I}})$ (notice that $N\widetilde{\bm{J}}$ and $N \widetilde{\bm{I}}$ are both symmetric). We will show that there exists a matrix $M \in SL(\mathbb{C}^d)$ such that $\rho_{M}(N \widetilde{\bm{J}}, N \widetilde{\bm{I}}) = (\widetilde{\bm{J}},\widetilde{\bm{I}})$, and this will prove the claim at hand. This fact is an immediate consequence of the following lemma.
\begin{lemma}
Let $(A_1,A_2)$ be a pair of symmetric $d\times d$ matrices, of which at least one is invertible, and assume that $N \in SL(\mathbb{C}^d)$ is such that $(NA_1, NA_2)$ is also a pair of symmetric matrices. Then there exists $M \in SL(\mathbb{C}^d)$ such that 
\[  (NA_1,NA_2) = (MA_1 M^\top, MA_2 M^\top). \]
\end{lemma}
We remark that the lemma can be extended to general $n$-tuples of symmetric matrices by essentially the same proof.
\begin{proof}
Assume $A_2$ is invertible, without loss of generality. We will show that it suffices to take $M$ to be a square root of $N$.\par 
Since $NA_2 = (NA_2)^\top = A_2 N^\top$, we have 
\begin{equation} \label{eq:matrix_relationship_A2_N}
N^{\top} = A_2^{-1} N A_2, 
\end{equation}
and therefore $N (A_1 A_2^{-1}) = A_1 N^\top A_2^{-1} = (A_1 A_2^{-1}) N$. In other words, $A_1 A_2^{-1}$ commutes with $N$. Since $N$ is a (complex) invertible matrix, it has a square root $N^{1/2}$ that commutes with $A_1 A_2^{-1}$ too. Indeed, this can be constructed via holomorphic calculus as follows: let $\log z$ denote a branch of the logarithm such that the branch cut does not contain any eigenvalue of $N$; then we define by Cauchy's formula
\[ \operatorname{Log} N := \frac{1}{2\pi i} \int_{\gamma} \log z \, (z I - N)^{-1} \,dz, \]
where $\gamma$ is the boundary of a domain that encloses the spectrum of $N$ and avoids the branch cut of $\log z$; finally, we define 
\[ N^{1/2} := \operatorname{Exp}\Big(\frac{1}{2} \operatorname{Log} N\Big). \]
It is easy to see that $N^{1/2}$ is indeed a square root of $N$ and that, thanks to the formula above, $N^{1/2}$ commutes with $A_1 A_2^{-1}$ as well. Notice that we also have the analogue of \eqref{eq:matrix_relationship_A2_N} for $N^{1/2}$, that is we have $(N^{1/2})^\top = A_2^{-1} N^{1/2} A_2$. As a consequence we have 
\begin{align*}
N^{1/2} A_1 (N^{1/2})^\top &= N^{1/2} A_1 (A_2^{-1} N^{1/2} A_2) \\
& = N^{1/2} N^{1/2} (A_1 A_2^{-1}) A_2 \\
&= N A_1; 
\end{align*}
similarly,
\begin{align*}
N^{1/2} A_2 (N^{1/2})^\top &= N^{1/2} A_2 (A_2^{-1} N^{1/2} A_2) \\
&= N A_2,
\end{align*}
and the lemma follows by taking $M = N^{1/2}$.
\end{proof}
We have therefore proven that $(A,B)$ and $(\widetilde{\bm{J}},\widetilde{\bm{I}})$ belong to the same orbit, and in particular to the same $\rho$-orbit (we omit the constant factor $\mu \mu''$ from now on). Now we take into account the action $\sigma$ as well by observing that $(\widetilde{\bm{J}},\widetilde{\bm{I}})$ is unstable if and only if the element $(\widetilde{\bm{J}} - \lambda_1^\ast \widetilde{\bm{I}},\widetilde{\bm{I}})$ is, since for $N_0 := \begin{pmatrix} 1 & -\lambda_1^{\ast} \\ 0 & 1 \end{pmatrix}$ we have
\[ \sigma_{N_0}(\widetilde{\bm{J}},\widetilde{\bm{I}}) = (\widetilde{\bm{J}} - \lambda_1^\ast \widetilde{\bm{I}},\widetilde{\bm{I}}). \]\par
Evaluating the expression $\widetilde{\bm{J}} - \lambda_1^\ast \widetilde{\bm{I}}$ block by block, we see that the above is a pair of matrices of the same form as $(\widetilde{\bm{J}},\widetilde{\bm{I}})$ but where the eigenvalue of highest multiplicity has been replaced by $0$ (more precisely, each $\widetilde{J}_{r}(\lambda_{1}^{\ast})$ block has been replaced by $\widetilde{J}_{r}(0)$). We will now show that the pair $(\widetilde{\bm{J}} - \lambda_1^\ast \widetilde{\bm{I}},\widetilde{\bm{I}})$ is unstable in two steps: 
\begin{enumerate}[(i)]
\item first we will exhibit a one-parameter subgroup of $SL(\mathbb{C}^d)$ that leaves $\widetilde{\bm{I}}$ fixed but is such that in the limit $\lambda \to 0$ every $\widetilde{J}_r(0)$ block in $\widetilde{\bm{J}} - \lambda_1^\ast \widetilde{\bm{I}}$ is replaced by a block of zeroes; 
\item then we will exhibit a one-parameter subgroup of $SL(\mathbb{C}^d)\times SL(\mathbb{C}^2)$ that shows that the latter is unstable (here is where we finally make use of the fact that $m_1 > d/2$). 
\end{enumerate}
This is enough to conclude: indeed, if $(C,D)$ is $(\rho\times\sigma)$-unstable and for a one-parameter subgroup $((M_\lambda, N_\lambda))_{\lambda \in \mathbb{C}^{\times}}$ we have $\lim_{\lambda \to 0} \rho_{M_\lambda} \sigma_{N_\lambda}(A,B) = (C,D)$, we have by continuity that $Q(A,B) = Q(C,D)$ for all $(\rho\times\sigma)$-invariant polynomials (with $Q(0,0)=0$); but $Q(C,D)=0$ always, and so the same holds for $(A,B)$, which is thus unstable as well.\par 
Consider any $\widetilde{J}_r(0)$ block in $\widetilde{\bm{J}} - \lambda_1^\ast \widetilde{\bm{I}}$, with $r>1$ (if $r=1$ we do not need to do anything); the corresponding block in $\widetilde{\bm{I}}$ is $\widetilde{I}_r$. If we denote 
\[ M_\lambda = \begin{pmatrix}
\lambda^{a_1} & & \\
 & \ddots & \\
 & & \lambda^{a_r}
\end{pmatrix} \]
then we see that 
\begin{align*}
M_\lambda \widetilde{J}_r(0) M_\lambda^\top &= \begin{pmatrix}
 & & & \lambda^{a_1 + a_{r-1}} & 0 \\
 & & \text{\reflectbox{$\ddots$}} & \text{\reflectbox{$\ddots$}} & \\
 & \lambda^{a_{r-2} + a_2}  & 0 & & \\
\lambda^{a_{r-1} + a_1} & 0 & & & \\
0 & & & & 
\end{pmatrix},\\
M_\lambda \widetilde{I}_r M_\lambda^\top &= \begin{pmatrix}
 & & & \lambda^{a_1 + a_r} \\
 & & \text{\reflectbox{$\ddots$}} & \\
  & \lambda^{a_{r-1} + a_2} & & \\
\lambda^{a_r + a_1} & & & 
\end{pmatrix}.
\end{align*}
If $r$ is even we choose
\[ (a_1,\ldots, a_r) = \Big(\frac{r}{2}, \frac{r}{2} - 1, \ldots, 1 - \frac{r}{2}, - \frac{r}{2}\Big) \]
and if $r$ is odd we choose 
\[ a_j := \left\lfloor \frac{r}{2} \right\rfloor - (j-1); \]
these choices satisfy the condition $\sum_{j=1}^{r} a_j =0$, and moreover they satisfy $a_{r - j} + a_j >0$ and $a_{r-j} + a_{j+1}=0$ for every $j$. Thus it is immediate that
\[ \lim_{\lambda \to 0} M_\lambda \widetilde{J}_r(0) M_\lambda^\top = 0, \quad M_\lambda \widetilde{I}_r M_\lambda^{\top} = \widetilde{I}_r. \]
It is then clear that we can construct (block by block) a one-parameter subgroup $(M_\lambda)_{\lambda \in \mathbb{C}^\times} \subset SL_d(\mathbb{C})$ such that 
\[ \lim_{\lambda \to 0} \rho_{M_\lambda}(\widetilde{\bm{J}} - \lambda_1^\ast \widetilde{\bm{I}}, \widetilde{\bm{I}}) = (\bm{J}_{0}, \widetilde{\bm{I}}), \]
where $\bm{J}_{0}$ is the matrix obtained from $\widetilde{\bm{J}} - \lambda_1^\ast \widetilde{\bm{I}}$ by replacing every $\widetilde{J}_r(0)$ block with a block of zeroes of the same $r \times r$ size (notice that we choose $\rho_{M_\lambda}$ to act trivially on the blocks of non-zero eigenvalue).\par 
Finally, we show that $(\bm{J}_{0}, \widetilde{\bm{I}})$ is $(\rho\times\sigma)$-unstable. By reordering the blocks (something that can be easily achieved via $\rho$) we may assume that $\bm{J}_{0}, \widetilde{\bm{I}}$ are of the form
\[ \bm{J}_{0} = \begin{tikzpicture}[baseline=(current bounding box.center)]%
\matrix[matrix of math nodes,
inner sep=0,
nodes={draw,outer sep=0,inner sep=2pt},
every left delimiter/.style={xshift=1ex},
every right delimiter/.style={xshift=-1ex},
left delimiter={(},right delimiter={)},
column sep=-\pgflinewidth,row sep=-\pgflinewidth] (m) {
 \mbox{\huge 0} & \\
 & \bm{J}_1\\
};
\end{tikzpicture}, \quad
\widetilde{\bm{I}} = \begin{tikzpicture}[baseline=(current bounding box.center)]%
\matrix[matrix of math nodes,
inner sep=0,
nodes={draw,outer sep=0,inner sep=2pt},
every left delimiter/.style={xshift=1ex},
every right delimiter/.style={xshift=-1ex},
left delimiter={(},right delimiter={)},
column sep=-\pgflinewidth,row sep=-\pgflinewidth] (m) {
\widetilde{\bm{I}}_1 & \\
 & \widetilde{\bm{I}}_2\\
};
\end{tikzpicture}
\]
where $\bm{J}_1$ is a matrix consisting of the remaining non-zero diagonal blocks of type $\widetilde{J}_r(\lambda_j)$ and $\widetilde{\bm{I}}_1,\widetilde{\bm{I}}_2$ are matrices consisting of the corresponding $\widetilde{I}_r$ diagonal blocks (in particular, $\bm{J}_1$ and $\widetilde{\bm{I}}_2$ have the same size). Observe that $\widetilde{\bm{I}}_1$ has size $m_1 \times m_1$, while $\bm{J}_1, \widetilde{\bm{I}}_2$ have size $(d-m_1) \times (d-m_1)$. If we let $M_\lambda$ denote the matrix
\[ M_\lambda := \begin{tikzpicture}[baseline=(current bounding box.center)]%
\matrix[matrix of math nodes,
inner sep=0,
nodes={draw,outer sep=0,inner sep=2pt},
every left delimiter/.style={xshift=1ex},
every right delimiter/.style={xshift=-1ex},
left delimiter={(},right delimiter={)},
column sep=-\pgflinewidth,row sep=-\pgflinewidth] (m) {
\lambda^{-(d-m_1)} I_{m_1} & \\
 & \lambda^{m_1} I_{d-m_1}\\
};
\end{tikzpicture} \]
then we see that the $M_\lambda$ form a one-parameter subgroup of $SL(\mathbb{C}^d)$ and moreover we have by a direct computation that 
\[ \rho_{M_\lambda} (\bm{J}_{0}, \widetilde{\bm{I}}) = 
\left(
\begin{tikzpicture}[baseline=(current bounding box.center)]%
\matrix[matrix of math nodes,
inner sep=0,
nodes={draw,outer sep=0,inner sep=2pt},
every left delimiter/.style={xshift=1ex},
every right delimiter/.style={xshift=-1ex},
left delimiter={(},right delimiter={)},
column sep=-\pgflinewidth,row sep=-\pgflinewidth] (m) {
\mbox{\huge 0} & \\
 & \lambda^{2m_1} \bm{J}_1 \\
};
\end{tikzpicture}, 
\begin{tikzpicture}[baseline=(current bounding box.center)]%
\matrix[matrix of math nodes,
inner sep=0,
nodes={draw,outer sep=0,inner sep=2pt},
every left delimiter/.style={xshift=1ex},
every right delimiter/.style={xshift=-1ex},
left delimiter={(},right delimiter={)},
column sep=-\pgflinewidth,row sep=-\pgflinewidth] (m) {
\lambda^{-2(d-m_1)} \widetilde{\bm{I}}_1 & \\
 & \lambda^{2m_1} \widetilde{\bm{I}}_2 \\
};
\end{tikzpicture}
\right). \]
Consider also the one-parameter subgroup of $SL(\mathbb{C}^2)$ given by 
\[ N_\lambda := \begin{pmatrix}
\lambda^{-2(d-m_1)-1} & 0 \\
0 & \lambda^{2(d-m_1)+1}
\end{pmatrix} \]
and observe that 
\[ \rho_{M_\lambda} \sigma_{N_\lambda} (\bm{J}_{0}, \widetilde{\bm{I}}) = \left( 
\begin{tikzpicture}[baseline=(current bounding box.center)]%
\matrix[matrix of math nodes,
inner sep=0,
nodes={draw,outer sep=0,inner sep=2pt},
every left delimiter/.style={xshift=1ex},
every right delimiter/.style={xshift=-1ex},
left delimiter={(},right delimiter={)},
column sep=-\pgflinewidth,row sep=-\pgflinewidth] (m) {
\mbox{\huge 0} & \\
 & \lambda^{2(2m_1-d)-1} \bm{J}_1 \\
};
\end{tikzpicture},
\begin{tikzpicture}[baseline=(current bounding box.center)]%
\matrix[matrix of math nodes,
inner sep=0,
nodes={draw,outer sep=0,inner sep=2pt},
every left delimiter/.style={xshift=1ex},
every right delimiter/.style={xshift=-1ex},
left delimiter={(},right delimiter={)},
column sep=-\pgflinewidth,row sep=-\pgflinewidth] (m) {
\lambda \widetilde{\bm{I}}_1 & \\
 & \lambda^{2d+1} \widetilde{\bm{I}}_2\\
};
\end{tikzpicture}
\right).
\]
Since $m_1 > d/2$ we have $2(2m_1 - d) - 1 >0$, and therefore 
\[ \lim_{\lambda \to 0} \sigma_{N_\lambda} \rho_{M_\lambda} (\bm{J}_{0}, \widetilde{\bm{I}}) = (0,0), \]
thus completing the proof that $(A,B)$ is $(\rho\times\sigma)$-unstable if $\Delta_{A,B}$ is not identically vanishing and $\widetilde{\sigma}$-unstable.
\subsection{Case II: $\Delta_{A,B}$ vanishes identically}\label{section:semistability_vanishing_determinant}
Here we assume that $(A,B) \in \operatorname{Sym}^2(\mathbb{C}^d) \times \operatorname{Sym}^2(\mathbb{C}^d)$ is such that 
\[ \det(sA + tB) \equiv 0, \]
or in other words that $\ker(sA + tB) \neq \{0\}$ for all $(s,t) \in \mathbb{C}^2$.\par 
We perform a first reduction. Suppose that for two linearly independent pairs $(s_1,t_1), (s_2, t_2)$ we have that 
\[ \ker (s_1 A + t_1 B) \cap \ker (s_2 A + t_2 B) \neq \{ 0 \} \]
(that is, the kernels have non-trivial intersection); we claim that $(A,B)$ is automatically $(\rho\times\sigma)$-unstable as a consequence. Notice that we can assume for simplicity that $(s_1,t_1) = (1,0)$ and $(s_2,t_2)=(0,1)$ by using the action $\sigma$ (this is essentially the same argument that was given before). Thus we are assuming that there exists a vector $\bm{v} \neq 0$ such that $A \bm{v} = B\bm{v} = 0$. Pick then vectors $\bm{u}_2,\ldots,\bm{u}_d$ so that $\{\bm{v}, \bm{u}_2,\ldots,\bm{u}_d\}$ forms a basis of $\mathbb{C}^d$ and moreover normalise them so that the matrix
\[ M := \begin{pmatrix} 
 & \bm{v}^{\top} & \\
 & \bm{u}_2^\top & \\
 & \vdots & \\
 & \bm{u}_d^{\top} &
\end{pmatrix} \]
is in $SL(\mathbb{C}^d)$. We then see by direct computation that $\rho_M(A,B)$ consists of a pair of matrices each of the form 
\[ \begin{pmatrix}
0 & 0 & \cdots & 0 \\
0 & \bm{\ast} & \cdots & \bm{\ast} \\
\vdots & \vdots & \ddots & \vdots \\
0 & \bm{\ast} & \cdots & \bm{\ast}
\end{pmatrix} \]
(where the asterisks denote possibly non-zero entries). If we consider now the one-parameter subgroup of $SL(\mathbb{C}^d)$ given by 
\[ M_\lambda := \begin{pmatrix}
\lambda^{-(d-1)} & & & \\
 & \lambda & & \\
 & & \ddots & \\
 & & & \lambda
\end{pmatrix}, \]
a computation reveals immediately that the effect of $\rho_{M_\lambda}$ on $\rho_M(A,B)$ is multiplication of every non-zero entry by $\lambda^2$ (because of the particular form of the matrices). Therefore we have 
\[ \lim_{\lambda \to 0} \rho_{M_\lambda}(\rho_M(A,B)) = (0,0) \]
and thus $(A,B)$ is indeed$(\rho\times\sigma)$-unstable.\par
In light of the above, we will assume in the rest of the argument that for every pair of linearly independent $(s_1,t_1), (s_2, t_2) \in \mathbb{C}^2$ we have 
\begin{equation}\label{eq:sA+tB_kernel_intersections_trivial}
\ker (s_1 A + t_1 B) \cap \ker (s_2 A + t_2 B) = \{ 0 \}. 
\end{equation}
Letting $I,J \subset \{1,\ldots, d\}$ with $|I|=|J|$, we denote by $\det_{I,J}M$ the minor of the matrix $M$ obtained by selecting the rows with index in $I$ and the columns with index in $J$. If $(A,B) \neq (0,0)$, some minors of $sA+tB$ will be not identically vanishing. We can then find $I_\ast,J_\ast$ of maximal cardinality such that $\det_{I_\ast,J_\ast}(sA+tB)$ does not vanish identically (and therefore it is non-zero for all $(s,t)$ except for a finite number of directions $as + bt = 0$). We define the set of \emph{generic} $(s,t)$ to be 
\[ \mathscr{G} := \{ (s,t) \in \mathbb{C}^2 : \det\nolimits_{I_\ast,J_\ast}(sA+tB) \neq 0 \}. \]
Notice that for $(s,t)$ generic we have that the dimension of $\ker (sA+tB)$ is constant and equal to $d$ minus the size of the minor; for $(s,t) \not\in \mathscr{G}$ the dimension of the kernel is larger instead. It will be useful to consider the vector space generated by the kernels of $sA + tB$ for generic $(s,t)$, that is
\[ V := \operatorname{Span}\Big\{ \bigcup_{(s,t) \in \mathscr{G}} \ker (sA + tB) \Big\}. \]
We let $k := \dim V$ and notice that by assumption \eqref{eq:sA+tB_kernel_intersections_trivial} we have $k \geq 2$. For convenience, we choose a basis $\{\bm{v}_1, \ldots, \bm{v}_k\}$ of $V$ such that for every $j \in \{1,\ldots, k\}$ 
\[ \bm{v}_j \in \ker(\tilde{s}_j A + \tilde{t}_j B) \]
for some $(\tilde{s}_j,\tilde{t}_j) \in \mathscr{G}$.\par
The first important observation to make is that all the images $(sA + tB)V$ for $(s,t) \in \mathscr{G}$ consist of a same vector space $H$. To begin with, all such images have the same dimension: indeed, for each $(s,t) \in \mathscr{G}$ we have $\ker (sA+tB) \leq V$ and $\dim \ker (sA+tB)$ is a constant; therefore $\dim (sA+tB)V = \dim V - \dim \ker (sA + tB)$ is a constant too. To conclude the claim, it will suffice to verify that for two linearly independent $(s_1,t_1), (s_2, t_2) \in \mathscr{G}$ we have 
\[ (s_1 A + t_1 B) V = (s_2 A + t_2 B) V =: H; \]
for if this is true, then by linear independence we will have $(sA + tB)V \leq H$ for every other $(s,t) \in \mathscr{G}$, and since the dimensions must be the same we will have actually $(sA+tB) V = H$ too. Take then $(s_1,t_1), (s_2, t_2)$ that are linearly independent and not multiples of any of the $(\tilde{s}_j,\tilde{t}_j)$ associated to the basis chosen above. For any $j \in \{1,\ldots,k\}$ there exist coefficients $a_j,b_j$ (both non-zero) such that 
\[ \tilde{s}_j A + \tilde{t}_j B = a_j (s_1 A + t_1 B) + b_j (s_2 A + t_2 B), \]
and since $(\tilde{s}_j A + \tilde{t}_j B)\bm{v}_j = 0$ we have 
\[ a_j (s_1 A + t_1 B)\bm{v}_j = - b_j (s_2 A + t_2 B) \bm{v}_j. \]
Therefore 
\begin{align*}
(s_1 A + t_1 B)V &= \operatorname{span} \{ (s_1 A + t_1 B)\bm{v}_j : 1\leq j \leq k\} \\
&= \operatorname{span} \{ (s_2 A + t_2 B)\bm{v}_j : 1\leq j \leq k\} = (s_2 A + t_2 B)V, 
\end{align*}
as desired.\par
The second observation to make (which is a consequence of the first) is that $V$ and $H$ are actually orthogonal to each other. Indeed, letting $\bm{u} \in H$, it suffices to show that $\langle \bm{u}, \bm{v}_j\rangle = 0 $ for all $j$. This is however easy to see: since $\bm{u} \in H$ and $H = (\tilde{s}_j A + \tilde{t}_j B)V$, there is a vector $\bm{v} \in V$ such that $(\tilde{s}_j A + \tilde{t}_j B)\bm{v} = \bm{u}$, and since the matrices are symmetric we have 
\[ \langle (\tilde{s}_j A + \tilde{t}_j B)\bm{v}, \bm{v}_j\rangle = \langle \bm{v}, (\tilde{s}_j A + \tilde{t}_j B)\bm{v}_j\rangle = \langle \bm{v} , 0 \rangle = 0. \]
Thus $V$ and $H$ are orthogonal, and besides $\dim H < k$ we have therefore $\dim H \leq d-k$ too.\par
We now claim that, as a consequence of the above observations, the $(\sigma\times\rho)$-orbit of $(A,B)$ contains a pair of symmetric matrices both of the form indicated in Figure \ref{eq:special_matrix_form_when_det(sA+tB)_vanishes}.
\begin{figure}[ht]
\centering
\begin{tikzpicture}[baseline=(current  bounding  box.center)]
\matrix [matrix of math nodes,left delimiter=(,right delimiter=),row sep=0.1cm,column sep=0.1cm] (m) {
      & & & & & \bm{\ast} & \cdots & \bm{\ast} \\
      & & & & & \vdots & \ddots & \vdots \\
      & & & & & \vdots & \ddots & \vdots \\
      & & & & & \bm{\ast} & \cdots & \bm{\ast} \\
      & & & & \bm{\ast} & \cdots & \cdots & \bm{\ast} \\
     \bm{\ast} & \cdots & \cdots & \bm{\ast} & \vdots & \ddots & & \vdots \\
     \vdots & \ddots & \ddots & \vdots & \vdots & & \ddots & \vdots \\
     \bm{\ast} & \cdots & \cdots & \bm{\ast} & \bm{\ast} & \cdots & \cdots & \bm{\ast} \\ };

\node[above=1pt of m-1-6] (top-6) {};
\node[above=1pt of m-1-8] (top-8) {};

\node[right=4pt of m-1-8] (right-1) {};
\node[right=4pt of m-4-8] (right-4) {};
\node[right=4pt of m-5-8] (right-5) {};
\node[right=4pt of m-8-8] (right-8) {};

\node[left=4pt of m-6-1] (left-6) {};
\node[left=4pt of m-8-1] (left-8) {};

\node[below=1pt of m-8-1] (below-1) {};
\node[below=1pt of m-8-4] (below-4) {};
\node[below=1pt of m-8-5] (below-5) {};
\node[below=1pt of m-8-8] (below-8) {};

\node[rectangle,above delimiter=\{] (del-top-6) at ($0.5*(top-6.south) +0.5*(top-8.south)$) {\tikz{\path (top-6.south west) rectangle (top-8.north east);}};
\node[above=10pt] at (del-top-6.north) {$\dim H$};

\node[rectangle,right delimiter=\}] (del-right-1) at ($0.5*(right-1.west) +0.5*(right-4.west)$) {\tikz{\path (right-1.north east) rectangle (right-4.south west);}};
\node[right=22pt] at (del-right-1.west) {$k$};

\node[rectangle,right delimiter=\}] (del-right-5) at ($0.5*(right-5.west) +0.5*(right-8.west)$) {\tikz{\path (right-5.north east) rectangle (right-8.south west);}};
\node[right=22pt] at (del-right-5.west) {$d-k$};

\node[rectangle,left delimiter=\{] (del-left-6) at ($0.5*(left-6.east)+0.5*(left-8.east)$) {\tikz{\path (left-6.north east) rectangle (left-8.south west);}};
\node[left=10pt] at (del-left-6.west) {$\dim H$};

\node[rectangle,below delimiter=\}] (del-below-1) at ($0.5*(below-1.north) +0.5*(below-4.north)$) {\tikz{\path (below-1.north west) rectangle (below-4.south east);}};
\node[below=10pt] at (del-below-1.south) {$k$};

\node[rectangle,below delimiter=\}] (del-below-5) at ($0.5*(below-5.north) +0.5*(below-8.north)$) {\tikz{\path (below-5.north west) rectangle (below-8.south east);}};
\node[below=10pt] at (del-below-5.south) {$d-k$};

\end{tikzpicture}
\caption{\footnotesize The special form of the matrices $A,B$ in the appropriate basis.}\label{eq:special_matrix_form_when_det(sA+tB)_vanishes} 
\end{figure}
\begin{remark}
We caution the reader that in the matrix diagram of Figure \ref{eq:special_matrix_form_when_det(sA+tB)_vanishes} the shape of the blocks of non-zero entries could be slightly misleading for large $k$ and $\dim H = d-k$ (more precisely, for $k>d/2$), but the block dimensions as stated are correct for all values of $k\geq 2$. For example, when $k = d-2$ we have $\dim H \leq 2$, and thus if $\dim H = 2$ the matrix looks like 
\[ \begin{pmatrix}
 & & & \bm{\ast} & \bm{\ast} \\
  & & & \vdots & \vdots \\
   & & & \bm{\ast} & \bm{\ast} \\
    \bm{\ast} & \cdots & \bm{\ast} & \bm{\ast} & \bm{\ast} \\
     \bm{\ast} & \cdots & \bm{\ast} & \bm{\ast} & \bm{\ast} 
\end{pmatrix};
 \]
it is evident that the block dimensions here are still as indicated in Figure \ref{eq:special_matrix_form_when_det(sA+tB)_vanishes}.
\end{remark}
By the usual argument using the action $\sigma$, we assume that $(1,0), (0,1) \in \mathscr{G}$ (that is, we can work with $(A,B)$). In order to show that there exists $M \in SL(\mathbb{C}^d)$ such that $\rho_M(A,B) = (MAM^\top, MBM^\top)$ consists of a pair of matrices both of the same form given in Figure \ref{eq:special_matrix_form_when_det(sA+tB)_vanishes}, begin by observing that if we write
\[ M = \begin{pmatrix}
& \bm{u}_1^{\top} & \\
& \vdots & \\
& \bm{u}_d^{\top} &
\end{pmatrix} \]
then the $(i,j)$-entry of $M A M^{\top}$ is $\langle \bm{u}_i, A \bm{u}_j\rangle$ (and the same holds for $B$). We then construct a basis of $\mathbb{C}^d$ in the following way: choose first $\bm{u}_1, \ldots, \bm{u}_k$ to be a basis of $V$, and then complete it to a basis of $H^{\perp}$ by further choosing linearly independent $\bm{u}_{k+1}, \ldots, \bm{u}_{d - \dim H}$; finally, complete the list to a basis of the whole $\mathbb{C}^d$ by choosing $\bm{u}_{d - \dim H + 1}, \ldots, \bm{u}_d$ to be a basis of $H$ (normalised so that $\det M = 1$). We have thus by construction (recall that $H = AV = BV$) that for all $1 \leq i \leq d - \dim H$ and $1 \leq j \leq k$
\[ \langle \bm{u}_i, A \bm{u}_j\rangle = \langle \bm{u}_i, B \bm{u}_j\rangle = 0; \]
since the matrices are symmetric, the fact that $MAM^\top$ and $MBM^\top$ are in the form of Figure \ref{eq:special_matrix_form_when_det(sA+tB)_vanishes} follows.\par 
We assume therefore that $A$ and $B$ are both of the form given in Figure \ref{eq:special_matrix_form_when_det(sA+tB)_vanishes} and proceed to make the block structure more explicit. Letting $\ell := \dim H$ for convenience, we can decompose $A$ and $B$ into (possibly rectangular) blocks as indicated in Figure \ref{eq:special_matrix_form_when_det(sA+tB)_vanishes_BLOCKS}, with dimensions as given in there.
\begin{figure}[ht]
\centering
\begin{tikzpicture}[baseline=(current  bounding  box.center)]
\matrix [matrix of math nodes,left delimiter=(,right delimiter=),row sep=0.1cm,column sep=0.1cm] (m) {
 {} & {} & {} & {} & {} & {} & {} & {} & {} \\
 {} & {} & {} & {} & {} & {} & {} & {A_1} & {} \\
 {} & {} & {} & {} & {} & {} & {} & {} & {} \\
 {} & {} & {} & {} & {} & {} & {} & {} & {} \\
 {} & {} & {} & {} & {A_2} & {} & {} & {A_3} & {} \\
 {} & {} & {} & {} & {} & {} & {} & {} & {} \\
 {} & {} & {} & {} & {} & {} & {} & {} & {} \\
 {} & {A_4} & {} & {} & {A_5} & {} & {} & {A_6} & {} \\
 {} & {} & {} & {} & {} & {} & {} & {} & {} \\
 };

\node[right=4pt of m-1-9] (right-1) {};
\node[right=4pt of m-3-9] (right-3) {};
\node[right=4pt of m-4-9] (right-4) {};
\node[right=4pt of m-6-9] (right-6) {};
\node[right=4pt of m-7-9] (right-7) {};
\node[right=4pt of m-9-9] (right-9) {};

\node[below=0pt of m-9-1] (below-1) {};
\node[below=0pt of m-9-3] (below-3) {};
\node[below=0pt of m-9-4] (below-4) {};
\node[below=0pt of m-9-6] (below-6) {};
\node[below=0pt of m-9-7] (below-7) {};
\node[below=0pt of m-9-9] (below-9) {};

\draw (m-1-6.north east) -- (m-1-9.north east);	
\draw (m-4-3.north east) -- (m-4-9.north east);
\draw (m-7-1.north west) -- (m-7-9.north east);
\draw (m-9-1.south west) -- (m-9-9.south east);
\draw (m-1-9.north east) -- (m-9-9.south east);	
\draw (m-1-6.north east) -- (m-9-6.south east);	
\draw (m-4-3.north east) -- (m-9-3.south east);	
\draw (m-7-1.north west) -- (m-9-1.south west);

\node[rectangle,right delimiter=\}] (del-right-1) at ($0.5*(right-1.west) +0.5*(right-3.west)$) {\tikz{\path (right-1.north east) rectangle (right-3.north west);}};
\node[right=22pt] at (del-right-1.west) {$k$};

\node[rectangle,right delimiter=\}] (del-right-4) at ($0.5*(right-4.west) +0.5*(right-6.west)$) {\tikz{\path (right-4.north east) rectangle (right-6.north west);}};
\node[right=22pt] at (del-right-4.west) {$d-k-\ell$};

\node[rectangle,right delimiter=\}] (del-right-7) at ($0.5*(right-7.west) +0.5*(right-9.west)$) {\tikz{\path (right-7.north east) rectangle (right-9.north west);}};
\node[right=22pt] at (del-right-7.west) {$\ell$};

\node[rectangle,below delimiter=\}] (del-below-1) at ($0.5*(below-1.north) +0.5*(below-3.north)$) {\tikz{\path (below-1.north west) rectangle (below-3.south west);}};
\node[below=10pt] at (del-below-1.south) {$k$};

\node[rectangle,below delimiter=\}] (del-below-4) at ($0.5*(below-4.north) +0.5*(below-6.north)$) {\tikz{\path (below-4.north west) rectangle (below-6.south west);}};
\node[below=10pt] at (del-below-4.south) {$d-k-\ell$};

\node[rectangle,below delimiter=\}] (del-below-7) at ($0.5*(below-7.north) +0.5*(below-9.north)$) {\tikz{\path (below-7.north west) rectangle (below-9.south west);}};
\node[below=10pt] at (del-below-7.south) {$\ell$};

\end{tikzpicture}
\caption{\footnotesize The decomposition of $A$ into rectangular blocks of dimensions as indicated. The decomposition of $B$ has the exact same shape. We remark that it might be the case that $d - k - \ell = 0$, in which case the blocks with the corresponding dimension are omitted (e.g.\ $A$ would contain only blocks $A_1, A_4,A_6$, which would be adjacent to each other).}\label{eq:special_matrix_form_when_det(sA+tB)_vanishes_BLOCKS}
\end{figure}
\par
We will now show that such a pair of matrices is necessarily $\rho$-unstable by producing an explicit one-parameter subgroup of $SL(\mathbb{C}^d)$ that sends $(A,B)$ to $(0,0)$ in the limit $\lambda \to 0$. This subgroup can be taken to be as follows: set
\begingroup
\allowdisplaybreaks
\begin{align*}
a_1 &:= -((d-1)\ell + d - k), \\
a_2 &:= k, \\
a_3 &:= dk,
\end{align*}
\endgroup
then define the block matrix
\[ 
M_\lambda = \begin{tikzpicture}[baseline=(current bounding box.center)]%
\matrix[matrix of math nodes,
inner sep=0,
nodes={draw,outer sep=0,inner sep=2pt},
every left delimiter/.style={xshift=1ex},
every right delimiter/.style={xshift=-1ex},
left delimiter={(},right delimiter={)},
column sep=-\pgflinewidth,row sep=-\pgflinewidth] (m) {
\lambda^{a_1} I_{k} & & \\
  & \lambda^{a_2} I_{d - k - \ell} & \\
   & & \lambda^{a_3} I_{\ell} \\
};
\end{tikzpicture}
\]
(once again, if $d - k - \ell = 0$ the middle block is omitted). The $M_\lambda$'s form a one-parameter subgroup of $SL(\mathbb{C}^d)$ because the sum of all the exponents involved is
\[ a_1 k + a_2 (d - k - \ell) + a_3 \ell = -((d-1)\ell + d - k)k + k (d-k-\ell) + dk \ell = 0. \]
By inspection, the effect of $\rho_{M_\lambda}$ on matrices of the form given in Figure \ref{eq:special_matrix_form_when_det(sA+tB)_vanishes_BLOCKS} is as follows:
\begin{equation*}
\begin{aligned}
& \rho_{M_\lambda}(A,B) = \\
&\left(
\begin{tikzpicture}[baseline=(current bounding box.center)]%
\matrix[matrix of math nodes,
inner sep=0,
nodes={draw,outer sep=0,inner sep=2pt},
every left delimiter/.style={xshift=1ex},
every right delimiter/.style={xshift=-1ex},
left delimiter={(},right delimiter={)},
column sep=-\pgflinewidth,row sep=-\pgflinewidth] (m) {
 & & \lambda^{a_1 + a_3} A_1\\
  & \;\;\lambda^{2 a_2} A_2\;\; & \lambda^{a_2 + a_3} A_3 \\
  \lambda^{a_1 + a_3} A_4 & \lambda^{a_2 + a_3} A_5 & \;\;\lambda^{2 a_3} A_6\;\; \\
};
\end{tikzpicture}, 
\begin{tikzpicture}[baseline=(current bounding box.center)]%
\matrix[matrix of math nodes,
inner sep=0,
nodes={draw,outer sep=0,inner sep=2pt},
every left delimiter/.style={xshift=1ex},
every right delimiter/.style={xshift=-1ex},
left delimiter={(},right delimiter={)},
column sep=-\pgflinewidth,row sep=-\pgflinewidth] (m) {
 & & \lambda^{a_1 + a_3} B_1\\
  & \;\;\lambda^{2 a_2} B_2\;\; & \lambda^{a_2 + a_3} B_3 \\
  \lambda^{a_1 + a_3} B_4 & \lambda^{a_2 + a_3} B_5 & \;\;\lambda^{2 a_3} B_6\;\; \\
};
\end{tikzpicture}
\right). \end{aligned} \end{equation*} 
Notice that $a_2, a_3 > 0$ and moreover, since $k > \ell$,
\[ a_1 + a_3 = -((d-1)\ell + d - k) + dk = d(k-\ell)-(d - k - \ell) \geq k + \ell > 0; \]
therefore we obtain
\[ \lim_{\lambda \to 0} \rho_{M_\lambda}(A,B) = (0,0), \]
and the proof of Proposition \ref{prop::semistability_SL2xSLd_det} (and hence of Theorem \ref{thm:well_curvedness_characterisation}) is concluded.

\section{Proof of Theorems \ref{main_theorem} and \ref{thm:flat_surfaces}}\label{section:main_proof}
We will now prove our main results by a simple instance of Christ's Method of Refinements. The method will reduce matters to sublevel set estimates for the polynomial $\det(s \nabla^2 Q_1 + t\nabla^2 Q_2)$, and these will be proven in Section \ref{section:sublevel_set_estimates}.
\begin{remark}\label{remark:other_proofs}
The result can also be proven by different methods -- in particular, the inflation technique in \cite{Gressman2022} and the testing conditions in \cite{Gressman2022b} (both due to Gressman) can each be employed to provide an alternative proof. Proceeding with either of those methods, the boundedness of the operator $\mathcal{T}$ is reduced to verifying respectively a non-concentration inequality and an integrability condition that explicitly involves $\det(s \nabla^2 Q_1 + t\nabla^2 Q_2)$; Theorem \ref{thm:well_curvedness_characterisation} provides the information needed to conclude either of these. In this paper we have chosen to use Christ's Method of Refinements mainly in the interest of providing a more self-contained exposition and because the condition to be verified (the sublevel set estimate) is slightly simpler.
\end{remark}
\subsection{Preliminaries and refinements} We begin by reformulating the desired estimates in combinatorial fashion. Let $1 \leq p,q <\infty$ be exponents such that $2/q = 1/p$; the restricted weak-type version of inequality $\|\mathcal{T}f\|_{L^q} \lesssim_{p,q} \|f\|_{L^p}$ is then
\[ \langle \mathcal{T}\mathbf{1}_E, \mathbf{1}_F\rangle \lesssim_{q} |E|^{2/q} |F|^{1/{q'}}, \]
where $E \subset \mathbb{R}^d \times [-1,1]^2$ and $F \subset \mathbb{R}^d \times [-1,1]^d$ have finite measure. Introducing the quantities
\begin{equation}\label{eq:defn_alpha_beta}
\alpha := \frac{\langle \mathcal{T}\mathbf{1}_E, \mathbf{1}_F\rangle}{|F|}, \quad \beta:= \frac{\langle \mathbf{1}_E, \mathcal{T}^{\ast}\mathbf{1}_F\rangle}{|E|}, 
\end{equation}
the restricted weak-type inequality above can be rewritten with a little algebra as 
\begin{equation}\label{eq:restricted_2_plane_non_mixed_weak_endpoint}
\alpha^{q-1}\beta \lesssim_{q} |E|.
\end{equation}
The problem has then been reduced to that of providing a lowerbound for the measure of $E$ in terms of $\alpha, \beta$. When the surface $\Sigma(Q_1,Q_2)$ is well-curved, we will prove this lowerbound for $q$ arbitrarily close to the critical value $q_0 = \tfrac{d+4}{2}$ (recall that the strong-type endpoint inequality is $L^{(d+4)/4} \to L^{(d+4)/2}$); and when we are in the situation described in the statement of Theorem \ref{thm:flat_surfaces}, we will prove the lowerbound for $q = m_{\ast}+2$. Theorems \ref{main_theorem} and \ref{thm:flat_surfaces} will then follow by entirely standard interpolation arguments.\par 
We now introduce some ``refinements'' of the sets $E,F$ with improved behaviour (this is what gives the method its name). Observe that if we let 
\[ F' := \Big\{ (x,\xi) \in F : \mathcal{T}\mathbf{1}_{E}(x,\xi) > \frac{\alpha}{2} \Big\} \]
then we have $\langle \mathcal{T}\mathbf{1}_{E}, \mathbf{1}_{F'} \rangle \geq \tfrac{1}{2} \langle \mathcal{T}\mathbf{1}_{E}, \mathbf{1}_{F}\rangle$: indeed, clearly
\[ \langle \mathcal{T}\mathbf{1}_{E}, \mathbf{1}_{F \setminus F'} \rangle \leq \frac{\alpha}{2} |F| = \frac{1}{2} \langle \mathcal{T}\mathbf{1}_{E}, \mathbf{1}_{F} \rangle, \]
and the claim follows; notice that $F' \neq \varnothing$, as a consequence. Thus in $F'$ we have enforced a lowerbound on $\mathcal{T}\mathbf{1}_{E}$. Next we observe that we can enforce an analogous lowerbound in a refinement of $E$ (but with respect to $\mathcal{T}^{\ast} \mathbf{1}_{F'}$ instead): we let 
\[ E' := \Big\{(y,s,t) \in E : \mathcal{T}^{\ast} \mathbf{1}_{F'}(y,s,t) > \frac{\beta}{4} \Big\}, \]
and by a repetition of the argument above we see that we have 
\[ \langle \mathbf{1}_{E'}, \mathcal{T}^{\ast}\mathbf{1}_{F'} \rangle \geq \langle \mathcal{T}\mathbf{1}_{E}, \mathbf{1}_{F'} \rangle - \tfrac{1}{4} \langle \mathcal{T}\mathbf{1}_{E}, \mathbf{1}_{F} \rangle \geq \tfrac{1}{4} \langle \mathcal{T}\mathbf{1}_{E}, \mathbf{1}_{F} \rangle \]
(so that $E' \neq \varnothing$ too). Summarising, we have shown the following lemma.
\begin{lemma}\label{lemma:refinements}
Let $E \subset \mathbb{R}^d \times [-1,1]^2$ and $F \subset \mathbb{R}^d \times [-1,1]^d$ be sets of finite positive measure, and let $\alpha,\beta$ be as in \eqref{eq:defn_alpha_beta}. Then there exist non-empty subsets $E' \subseteq E, F' \subseteq F$ such that 
\begin{enumerate}[(i)]
\item for every $(x,\xi) \in F'$ we have $\mathcal{T}\mathbf{1}_{E}(x,\xi) \gtrsim \alpha$, 
\item for every $(y,s,t) \in E'$ we have $\mathcal{T}^{\ast} \mathbf{1}_{F'}(y,s,t) \gtrsim \beta$.
\end{enumerate}
\end{lemma}
The reason why these properties are remarkable is that they translate into (uniform) lowerbounds for the size of certain sets. To see this, let us introduce some notation: we let
\[ \gamma(({x},{\xi}),(s,t)):= ({x} - s \nabla Q_1(\xi) - t \nabla Q_2(\xi), s, t), \]
so that $\mathcal{T}f(x,\xi) = \iint_{|s|,|t|\leq 1} f(\gamma(({x},{\xi}),(s,t)))\,ds\,dt$; moreover, we let 
\[ \gamma^{\ast}((y,s,t),\eta)=(y + s \nabla Q_1(\eta) + t \nabla Q_2(\eta), \eta), \]
so that $\mathcal{T}^{\ast} g(y,s,t) := \int_{[-1,1]^d} g(\gamma^{\ast}((y,s,t),\eta))\,d\eta$. Now observe that
\[ \mathcal{T}^{\ast}\mathbf{1}_{F'}(y,s,t) = |\{ \eta \in [-1,1]^d : \gamma^{\ast}((y,s,t), \eta) \in F'\}|, \]
so that if we pick $(y_0,s_0,t_0) \in E'$ and we let
\[ \mathcal{B}:= \{ \eta \in [-1,1]^d : \gamma^{\ast}((y_0,s_0,t_0), \eta) \in F'\}, \]
we have by Lemma \ref{lemma:refinements}
\[ |\mathcal{B}| \gtrsim \beta. \]
Similarly, we see that if $(x,\xi) \in F'$ we have (again by Lemma \ref{lemma:refinements}) 
\[ |\{ (s,t) \in [-1,1]^2 : \gamma((x,\xi),(s,t)) \in E\} | \gtrsim \alpha; \]
we can then define for $\eta \in \mathcal{B}$
\[ \mathcal{A}_{\eta} := \{ (s,t) \in [-1,1]^2 : \gamma( \gamma^{\ast}((y_0,s_0,t_0), \eta) ,(s,t)) \in E\} \]
and have uniformly 
\[ |\mathcal{A}_{\eta}| \gtrsim \alpha. \]
\subsection{Change of variables and conclusion}
We can see from the above discussion that the function
\[ \Psi(\eta,s,t):= \gamma( \gamma^{\ast}((y_0,s_0,t_0), \eta) ,(s,t)) \]
maps the set 
\[ \bigcup_{\eta \in \mathcal{B}} (\{ \eta \} \times \mathcal{A}_{\eta}) \]
into the set $E$, thus providing a way to obtain lowerbounds on $|E|$; moreover, it is a map from $\mathbb{R}^{d+2}$ into itself, which will enable us to use the change of variables formula to obtain explicit lowerbounds. To make use of these ideas and in anticipation of the technical challenges, we introduce for every $\eta \in \mathcal{B}$ subsets $\mathcal{A}'_{\eta} \subseteq \mathcal{A}_{\eta}$, which will be specified later; these are assembled into the set 
\begin{equation}\label{eq:tower_of_parameters}
S := \bigcup_{\eta \in \mathcal{B}} (\{ \eta\} \times \mathcal{A}'_{\eta}), 
\end{equation}
and we stress that we have $\Psi(S) \subset E$. By the change of variables formula we have then
\[ |E| \geq \mu_{\Psi}^{-1} \int_{S} |J\Psi(\eta,s,t)| \,d\eta\,ds\,dt, \]
where $\mu_{\Psi} = \max_{(\eta,s,t) \in S} \# \Psi^{-1}(\eta,s,t)$ is the multiplicity of the map $\Psi$ and $J\Psi$ its Jacobian determinant, which we will now calculate. Observe that 
\[ \Psi(\eta,s,t) = (y_0 - (s-s_0) \nabla Q_1(\eta) - (t-t_0) \nabla Q_2(\eta),s,t), \]
so that the Jacobian of $\Psi$ is given by 
\[ - \begin{pmatrix}
(s-s_0)\nabla^2 Q_1(\eta) + (t-t_0)\nabla^2 Q_2(\eta) &  \nabla Q_1(\eta) & \nabla Q_2(\eta) \\
\begin{matrix}
0 & \cdots & \cdots & 0 
\end{matrix} & -1 & 0 \\
\begin{matrix}
0 & \cdots & \cdots & 0 
\end{matrix} & 0 & -1
\end{pmatrix} \]
and it is immediate that\footnote{Notice that when $Q_1,Q_2$ are quadratic forms the Jacobian determinant is independent of $\eta$.}
\begin{equation}\label{eq:Jacobian} 
J\Psi(\eta,s,t) = (-1)^{d} \det( (s-s_0)\nabla^2 Q_1(\eta) + (t-t_0)\nabla^2 Q_2(\eta)); 
\end{equation}
crucially, this is the same object that characterises the well-curvedness of $\Sigma(Q_1,Q_2)$. As for $\mu_{\Psi}$, we have $\Psi(\eta,s,t) = \Psi(\eta',s',t')$ only if $s=s', t=t'$; moreover, $Q_1,Q_2$ are quadratic forms and therefore we must have (switching again to Hessian matrices $A,B$)
\[ (s-s_0) A (\eta - \eta') + (t-t_0) B (\eta - \eta') = 0. \]
If we choose $S$ so as to impose $\det((s-s_0)A + (t-t_0)B)\neq 0$ (which we will), we see that the above equation is solved only by $\eta = \eta'$, and thus we will have $\mu_{\Psi}=1$.\par 
Assume now that the surface $\Sigma(Q_1,Q_2)$ is well-curved and fix $\epsilon >0$ arbitrarily small. We claim that we can choose subsets $\mathcal{A}'_{\eta}$ so that 
\begin{enumerate}[(i)]
\item\label{condition:lowerbound_A_eta} $|\mathcal{A}'_{\eta}| \gtrsim \alpha$ for every $\eta \in \mathcal{B}$, 
\item\label{condition:lowerbound_Jacobian} for every $(\eta,s,t) \in S$ we have $|J\Psi(\eta,s,t)| \gtrsim_{\epsilon} \alpha^{d/2 + \epsilon}$.
\end{enumerate}
If these conditions are satisfied we see immediately from \eqref{eq:tower_of_parameters} that $|S| \gtrsim \alpha \beta$ and moreover that 
\[ |E| \geq \int_{S} |J\Psi(\eta,s,t)|\,d\eta\,ds\,dt \gtrsim_{\epsilon} \alpha^{\frac{d+2}{2} + \epsilon} \beta, \]
which is precisely the desired inequality \eqref{eq:restricted_2_plane_non_mixed_weak_endpoint} for $q = \frac{d+4}{2} + \epsilon$; since $\epsilon$ is arbitrary, this proves Theorem \ref{main_theorem}. To obtain the conditions above, simply choose 
\[ \mathcal{A}'_{\eta} := \mathcal{A}_{\eta} \setminus \{ (s,t) \in [-1,1]^2 : |\det((s-s_0) A + (t - t_0)B)| < C_{\epsilon} \alpha^{d/2 + \epsilon} \} \]
for $C_{\epsilon}>0$; then by \eqref{eq:Jacobian} we see that condition \eqref{condition:lowerbound_Jacobian} is automatically satisfied. As for condition \eqref{condition:lowerbound_A_eta}, Theorem \ref{thm:well_curvedness_characterisation} and Proposition \ref{prop:sublevel_set_estimate_1} (which will be proven in Section \ref{section:sublevel_set_estimates}) imply the sublevel set estimate
\[ |\{ (s,t) \in [-1,1]^2 : |\det((s-s_0) A + (t - t_0)B)| < C_{\epsilon} \alpha^{d/2 + \epsilon} \}| \ll \alpha \]
(provided $C_{\epsilon}$ is chosen sufficiently small), from which condition \eqref{condition:lowerbound_A_eta} follows at once.\par
Suppose instead that the surface $\Sigma(Q_1,Q_2)$ is flat, but $\det(sA + tB)$ does not vanish identically and has a root of multiplicity $m_{\ast} > d/2$ (these are the hypotheses of Theorem \ref{thm:flat_surfaces}). In this case we claim that we can find subsets $\mathcal{A}'_{\eta}$ so that 
\begin{enumerate}[(i)]
\item $|\mathcal{A}'_{\eta}| \gtrsim \alpha$ for every $\eta \in \mathcal{B}$ (as before), 
\item for every $(\eta,s,t) \in S$ we have $|J\Psi(\eta,s,t)| \gtrsim \alpha^{m_{\ast}}$.
\end{enumerate}
This is achieved in exactly the same way, with the only difference being that we appeal to Proposition \ref{prop:sublevel_set_estimate_2} instead to obtain the sublevel set estimate 
\[  |\{ (s,t) \in [-1,1]^2 : |\det((s-s_0) A + (t - t_0)B)| < C \alpha^{m_{\ast}} \}| \ll \alpha. \]
Then the same argument as before shows that 
\[ |E| \gtrsim \alpha^{m_{\ast} + 1} \beta, \]
which is inequality \eqref{eq:restricted_2_plane_non_mixed_weak_endpoint} for $q = m_{\ast} + 2$, as claimed. The proofs of Theorems \ref{main_theorem} and \ref{thm:flat_surfaces} are thus concluded, conditionally on Propositions \ref{prop:sublevel_set_estimate_1} and \ref{prop:sublevel_set_estimate_2} (recall also that the negative parts of the statements will be proven in Section \ref{section:flat_surfaces}). 
\begin{remark}
In Theorem \ref{thm:flat_surfaces} and in certain cases of Theorem \ref{main_theorem} it is possible to refine the restricted weak-type inequalities to restricted strong-type inequalities by using the Inflation Method instead (also originating in M.\@ Christ's work, see \cite{Christ02,Christ05}); however, the range of exponents obtained by interpolation is the same in either case.
\end{remark}
\section{Sublevel set estimates}\label{section:sublevel_set_estimates}
In this section we will prove the sublevel set estimates that are needed to close the argument of Section \ref{section:main_proof}. There are two types of estimates (one for the well-curved case, one for the flat case), which are encapsulated in the two propositions below, stated for general homogeneous polynomials of two variables. Recall that by a root of a homogeneous polynomial in $\mathbb{R}[s,t]$ we mean a homogeneous linear divisor in $\mathbb{C}[s,t]$.
\begin{proposition}\label{prop:sublevel_set_estimate_1}
Let $P(s,t)$ be a real homogeneous polynomial of degree $d$. If all the roots of $P$ have multiplicity $\leq d/2$ then we have for every $\delta > 0$
\begin{equation}\label{eq:sublevel_set_est_well_curved}
|\{ (s,t) : |s|,|t|\lesssim 1, \, |P(s,t)|<\delta\}| \lesssim_P \, \delta^{2/d} \log^{+} 1/\delta. 
\end{equation}
\end{proposition}
\begin{proposition}\label{prop:sublevel_set_estimate_2}
Let $P(s,t)$ be a real homogeneous polynomial of degree $d$. If $P$ has a root of multiplicity $m_{\ast} > d/2$ then we have for every $\delta > 0$
\begin{equation}\label{eq:sublevel_set_est_flat}
|\{ (s,t) : |s|,|t|\lesssim 1, \, |P(s,t)|<\delta\}| \lesssim_P \, \delta^{1/m_{\ast}}.
\end{equation}
\end{proposition}
These sublevel set estimates are sharp in several ways. First of all, it is not possible to improve the exponent $2/d$ in \eqref{eq:sublevel_set_est_well_curved}: indeed, if $|s|,|t| \lesssim \delta^{1/d}$ then each monomial in $P(s,t)$ is $\lesssim \delta$, and therefore the sublevel set contains the set $\{(s,t) : |s|,|t| \lesssim \delta^{1/d}\}$, which has measure $\gtrsim \delta^{2/d}$. Secondly, if the root multiplicity assumption of Proposition \ref{prop:sublevel_set_estimate_1} is violated, \eqref{eq:sublevel_set_est_well_curved} can no longer hold: since we can write $P(s,t) = (as + bt)^m Q(s,t)$ for some $a,b \in \mathbb{C}$ and some homogeneous polynomial $Q$ of degree $d-m$, we see that $|Q(s,t)|\lesssim 1$ and therefore the sublevel set contains the set $\{(s,t) : |s|,|t|\lesssim 1, |as+bt|^m \lesssim \delta\}$, which is seen to have measure $\gtrsim \delta^{1/m} \gg \delta^{2/d}$. This also shows that it is not possible to improve the exponent $1/m_{\ast}$ in \eqref{eq:sublevel_set_est_flat}. Finally, it is not possible in general to remove the logarithmic factor in \eqref{eq:sublevel_set_est_well_curved}: consider for example polynomials $P(s,t) = s^{d/2} t^{d/2}$ when $d$ is even.\footnote{This polynomial can be realised as $\det(sA+tB)$ for block matrices $A = \begin{pmatrix}
I & 0 \\ 0 & 0 
\end{pmatrix}$, $B = \begin{pmatrix}
0 & 0 \\ 0 & I
\end{pmatrix}$; thus the log-loss cannot be avoided even in our case of interest.}\par 
There is a rich and well-developed theory of sublevel set estimates for polynomials (and more in general for analytic functions) which runs in parallel to an analogous theory of oscillatory integral estimates with polynomial phases. The two are intimately related: indeed, it is well-known that it is possible to deduce sublevel set estimates from estimates for the corresponding oscillatory integrals (see e.g. Section 1 of \cite{CarberyChristWright}). For multivariable phases, the oscillatory integrals theory has been developed by A.\@ N.\@ Varchenko in his foundational work \cite{Varchenko}. The main takeaway of this theory is that the rate of decay is controlled by the \emph{height} of the phase, which is the supremum of the Newton distance\footnote{The Newton distance of an analytic function $f$ is the smallest $d\geq 0$ such that $(d,\ldots,d)$ belongs to the Newton diagram of $f$.} taken over all locally smooth (or analytic) coordinate systems. One could therefore prove Propositions \ref{prop:sublevel_set_estimate_1} and \ref{prop:sublevel_set_estimate_2} from the corresponding oscillatory integral estimates of Varchenko by computing the height of $P$, given the multiplicity assumption. This computation has been carried out already by I.\@ A.\@ Ikromov and D.\@ M\"{u}ller in \cite{IkromovMuller} (Corollary 3.4), in which they showed that in our case the height is $\max\{m_{\ast}, d/2\}$, where $m_{\ast}$ denotes the largest root multiplicity; thus one obtains the desired proofs. Alternatively, one could use the same corollary of \cite{IkromovMuller} and an integration argument in polar coordinates to obtain a direct proof that does not require the oscillatory integrals theory of Varchenko.\footnote{The argument proceeds by rewriting $ |\{(s,t): s^2 + t^2\leq 1, \, |P(s,t)|<\delta\}| = \int_{0}^{2\pi} \int_{0}^{1} \mathbf{1}_{[-\delta,\delta]}(r^d |P(\cos \alpha, \sin \alpha)|)\, r\,dr \,d\alpha,$ which is then equal to (letting $Q(\alpha):= P(\cos \alpha, \sin \alpha)$) $\frac{1}{2} \delta^{2/d} \int_{\{\alpha : |Q(\alpha)|>\delta\}} |Q(\alpha)|^{-2/d} \,d\alpha + \frac{1}{2} |\{ \alpha : |Q(\alpha)|<\delta\}|; $ both terms can be estimated by factoring $Q(\alpha)$ and using \cite{IkromovMuller}. The argument was pointed out to us by J.\@ Wright.} Here however we will offer our own independent proofs that rely on a simple but interesting linear programming argument (that such arguments are powerful enough to deal with sublevel set and oscillatory integral estimates was already observed in \cite{Gilula}). Besides the inherent interest, the method we employ is conveniently stable under perturbations of $P$, due to the fact that the constants involved are sufficiently explicit; this will come in handy when we prove Theorem \ref{thm:general_well_curved_surfaces} in Appendix \ref{appendix:modification_general_surface}. The estimates of Varchenko are also stable under analytic perturbations in the case of two variables, as was shown by V.\@ N.\@ Karpushkin in \cite{Karpushkin}. By contrast, the aforementioned integration argument in polar coordinates produces a constant that depends on the separation between the roots, which is not stable under perturbations.
\begin{proof}[Proof of Proposition \ref{prop:sublevel_set_estimate_1}]
Since $P \in \mathbb{R}[s,t]$ is homogeneous of degree $d$, it can be factored over $\mathbb{C}$ as 
\[ P(s,t) = C \prod_{j=1}^{d}\theta_j(s,t), \]
where the $\theta_j$ are homogeneous linear forms (that is, $\theta_j(s,t) = a_j s + b_j t$). Since $P$ is a real polynomial, we can arrange things so that the $\theta_j$ are either real or occur in complex conjugate pairs. We furthermore choose a normalisation of the $\theta_j$'s so that if $[a_j : b_j] = [a_k : b_k]$ (as points of $\mathbb{P}(\mathbb{C}^2)$) then $\theta_j = \theta_k$; thus the multiplicity of a root of $P(s,t)$ is simply the number of occurrences of a same factor $\theta$ in the product above. Notice that $C$ ends up depending on $P$. If the distinct factors are $\theta_1, \ldots, \theta_\ell$ (in particular, they are all pairwise linearly independent) and the respective multiplicities are $m_j$ (thus $\sum_{j=1}^{\ell} m_j = d$ and $m_j \leq d/2$), we can write
\[ P(s,t) = C \prod_{j = 1}^{\ell} \theta_j(s,t)^{m_j}. \]\par 
First of all, we will need to control sublevel sets of polynomials with only two distinct roots; this is achieved by the next lemma.
\begin{lemma}\label{lemma:sublevel_set_estimate_two_factors}
Let $\mu, \nu >0$ and let $\theta,\theta' \in \mathbb{C}[s,t]$ be linear forms that are $\mathbb{C}$-linearly independent. Then for every $\delta>0$
\[ |\{ s,t : |s|,|t|\lesssim 1, |\theta(s,t)^{\mu}\theta'(s,t)^{\nu}|\lesssim \delta\}| \lesssim_{\theta,\theta'} \begin{cases} \delta^{1/\max\{\mu,\nu\}} & \quad \text{ if } \mu \neq \nu \\
\delta^{1/\mu} \log^{+} 1/\delta & \quad \text{ if } \mu = \nu.
\end{cases} \]
\end{lemma}
\begin{proof}[Proof of Lemma \ref{lemma:sublevel_set_estimate_two_factors}]
From $\mathbb{C}$-linear independence we see in fact that we can pick real linear forms $\hat{\theta} \in \{ \operatorname{Re}\theta, \operatorname{Im}\theta\}$ and $\hat{\theta}' \in \{ \operatorname{Re}\theta', \operatorname{Im}\theta'\}$ so that $\hat{\theta}, \hat{\theta}'$ are $\mathbb{R}$-linearly independent. Since $|\hat{\theta}|\leq |\theta|$ and $|\hat{\theta}'|\leq |\theta'|$ we have then
\[ |\{ s,t : |s|,|t|\lesssim 1, |\theta(s,t)^{\mu}\theta'(s,t)^{\nu}|\lesssim \delta\}| \leq |\{ s,t : |s|,|t|\lesssim 1, |\hat{\theta}(s,t)|^{\mu}|\hat{\theta}'(s,t)|^{\nu} \lesssim \delta\}|, \] 
and by a linear change of variables the latter is 
\[ \lesssim_{\theta,\theta'} |\{ s,t : |s|,|t|\lesssim_{\theta,\theta'} 1, |s|^{\mu} |t|^{\nu}\lesssim \delta\}|. \]
By a simple integration we see that if $\mu \neq \nu$ then the last expression is dominated by $\lesssim \delta^{1/\max\{\mu,\nu\}}$, and if $\mu = \nu$ then it is dominated by $\lesssim \delta^{1/\mu} \log^{+} 1/\delta$.
\end{proof}
\begin{remark}\label{remark:implicit_constant_change_of_var}
The implicit constant in the estimate of Lemma \ref{lemma:sublevel_set_estimate_two_factors} can be made explicit: it is simply $O(|\det \begin{pmatrix} \hat{\theta} & \hat{\theta}' \end{pmatrix}|^{-1})$, where $\det \begin{pmatrix} \hat{\theta} & \hat{\theta}' \end{pmatrix}$ denotes the Jacobian determinant of the map $(s,t) \mapsto (\hat{\theta}(s,t), \hat{\theta}'(s,t))$, and $\hat{\theta}, \hat{\theta}'$ are as in the proof just given.
\end{remark}
We will show that for a general polynomial that satisfies the multiplicity assumption of Proposition \ref{prop:sublevel_set_estimate_1}, we can always reduce at least to the second case of the lemma.\par
As a step in the direction indicated, we claim that we can always rewrite the polynomial $P$ as a product of pairs of the form $(\theta_j\theta_k)^\mu$: more precisely, we will show that there exist quantities $\mu_{jk} \geq 0$ such that 
\begin{equation}\label{eq:alternative_factorisation}
\prod_{j=1}^{\ell} \theta_j(s,t)^{m_j} = \prod_{j=1}^{\ell} \prod_{j < k \leq \ell} (\theta_j(s,t) \theta_k(s,t))^{\mu_{jk}}. 
\end{equation}
Indeed, looking at the exponents, the equality translates immediately into the existence of a non-negative solution $(\mu_{jk})_{1\leq j < k \leq \ell}$ to the linear equations\footnote{Notice that the resulting system of equations has $\ell(\ell-1)/2$ variables and $\ell$ equations, and is therefore severely underdetermined.}
\begin{equation}\label{eq:system_positive_mu_jk}
L_j : \quad \sum_{i:\; i < j} \mu_{ij} + \sum_{k:\; k>j} \mu_{jk} = m_j, \qquad j \in \{1,\ldots, \ell\}. 
\end{equation}
In order to treat such a system of linear equations, we recall the following fundamental linear programming lemma. For convenience, given a vector $\bm{v}$ we write $\bm{v} \geq 0$ to denote the fact that all components of $\bm{v}$ are non-negative. 
\begin{lemma}[Farkas' lemma; \cite{Rockafellar}]\label{lemma:Farkas}
If $M$ is an $m \times n$ real matrix and $\bm{b} \in \mathbb{R}^m$, then exactly one of the following mutually exclusive cases holds:
\begin{enumerate}[(i)]
\item there exists $\bm{x} \in \mathbb{R}^n$ such that $M\bm{x} = \bm{b}$ with $\bm{x}\geq 0$, or
\item\label{item:Farkas_item_2} there exists $\bm{y} \in \mathbb{R}^m$ such that $M^\top \bm{y} \geq 0$ and $\bm{b} \cdot \bm{y} < 0$. 
\end{enumerate}
\end{lemma}
\begin{remark}
The statement might appear somewhat cryptic at first, but the geometric content is actually elementary: if we let $\Gamma_{+} := \{ \bm{x} \in \mathbb{R}^n : \bm{x} \geq 0\}$, we observe that $\Gamma_{+}$ is a closed convex cone and therefore so is $M \Gamma_{+}$; then Farkas' lemma simply states that either $\bm{b}$ belongs to $M \Gamma_{+}$ or not, in which case the two can be separated by a hyperplane ($\bm{y}$ is an element orthogonal to this hyperplane and on the opposite side to $\bm{b}$).
\end{remark}
We will show that case \eqref{item:Farkas_item_2} of Lemma \ref{lemma:Farkas} is impossible in our situation (in which $\bm{b} = (m_1, \ldots, m_\ell)$ and $M$ can be read off of the system of equations \eqref{eq:system_positive_mu_jk}), and thus the desired $(\mu_{jk})_{j<k}$ exist. Assume by contradiction that there is such a vector $\bm{y} = (y_1,\ldots, y_\ell)$ as in case \eqref{item:Farkas_item_2}. Inspecting the system \eqref{eq:system_positive_mu_jk} we see that the condition $M^\top \bm{y} \geq 0$ translates into the system of inequalities
\begin{equation} \label{eq:Farkas_inequalities_yj_yk}
y_j + y_k \geq 0 
\end{equation}
for all $1 \leq j < k \leq \ell$ (indeed, observe that each variable $\mu_{jk}$ appears only in equations $L_j$ and $L_k$, always with coefficient $+1$); the condition $\bm{b} \cdot \bm{y} < 0$ is simply the statement that 
\[ y_1 m_1 + \ldots + y_\ell m_\ell < 0. \]
On the one hand, since the $m_j$ are all positive, from the last inequality we see that at least one of the $y_j$ must be negative. On the other hand, from inequalities \eqref{eq:Farkas_inequalities_yj_yk} we see that there can be at most a single index $j_{\ast}$ such that $y_{j_{\ast}} < 0$ and that all other $y_j$ must be strictly positive instead; in particular, $y_j \geq |y_{j_{\ast}}| >0$. But then we have 
\[ \sum_{j \neq j_{\ast}} m_j \leq \sum_{j \neq j_{\ast}} \frac{y_j}{|y_{j_{\ast}}|} m_j < m_{j_{\ast}}, \]
and this implies that $m_{j_{\ast}} > d/2$, which is a contradiction.\par 
The above has shown that the desired structural factorisation of $P$ can be achieved -- and notice in particular that we have necessarily $\sum_{j<k} \mu_{jk} = d/2$. Now consider only those indices $j,k$ such that $\mu_{jk} > 0$. By the pigeonhole principle and factorisation \eqref{eq:alternative_factorisation} we have that if $|P(s,t)|\leq \delta$ then for at least one pair of indices $j < k$ we have 
\[ |\theta_j(s,t)\theta_k(s,t)|^{\mu_{jk}} \lesssim_P \,\delta^{2\mu_{jk}/d}; \]
it follows that $|\{ s,t : |s|,|t|\lesssim 1, |P(s,t)|\leq \delta\}|$ is dominated by the sum in indices $j<k$ of 
\[ |\{ s,t : |s|,|t|\lesssim 1, |\theta_j(s,t)\theta_k(s,t)|^{\mu_{jk}} \lesssim_P \, \delta^{2\mu_{jk}/d} \}|. \]
However, since the $\theta_j$'s are normalised and distinct, they are linearly independent in pairs; by Lemma \ref{lemma:sublevel_set_estimate_two_factors} this measure is dominated by $\lesssim_P \, (\delta^{2\mu_{jk}/d})^{1/\mu_{jk}}\log^{+} (1/\delta^{2\mu_{jk}/d}) \allowbreak \sim \delta^{2/d} \log^{+} 1/\delta$, and we are done.  
\end{proof}
The proof of \eqref{eq:sublevel_set_est_flat} follows similar lines but is much simpler.
\begin{proof}[Proof of Proposition \ref{prop:sublevel_set_estimate_2}]
As in the proof of \eqref{eq:sublevel_set_est_well_curved}, we can factorise $P$ as 
\[ P(s,t) = C \theta_{\ast}(s,t)^{m_{\ast}} \prod_{j=1}^{\ell} \theta_{j}(s,t)^{m_j}, \]
where $m_{\ast} > d/2$ is the largest multiplicity and $\theta_{\ast}, \theta_1, \ldots, \theta_{\ell}$ are linearly independent linear forms. Since $m_{\ast} > \sum_{j=1}^{\ell} m_j$, we can find $\mu_{j}$ such that $\mu_j > m_j$ and $\sum_{j=1}^{\ell} \mu_j = m_{\ast}$; as a consequence, we can rearrange the factorisation of $P$ as 
\[ P(s,t) = C \prod_{j=1}^{\ell} ( \theta_{\ast}(s,t)^{\mu_j} \theta_j(s,t)^{m_j}). \]
By the pigeonhole principle, if $|P(s,t)| < \delta$ then for at least one index $j$ we have 
\[ |\theta_{\ast}(s,t)^{\mu_j} \theta_j(s,t)^{m_j}| \lesssim_P \delta^{\mu_j / m_{\ast}}; \]
therefore the sublevel set $\{s,t : |s|,|t|\lesssim 1, |P(s,t)|< \delta\}$ is contained in the union over $j$ of sublevel sets 
\[ \{ s,t : |s|,|t| \lesssim 1, |\theta_{\ast}(s,t)^{\mu_j} \theta_j(s,t)^{m_j}| \lesssim_P \delta^{\mu_j / m_{\ast}} \}. \]
By Lemma \ref{lemma:sublevel_set_estimate_two_factors}, each of these has measure $\lesssim_P (\delta^{\mu_j / m_{\ast}})^{1/ \max\{\mu_j, m_j\}} = \delta^{1/{m_{\ast}}}$, and thus the proof is concluded.
\end{proof}
\begin{remark}\label{remark:surfaces_without_log_loss}
While it is not possible in general to remove the logarithmic loss in \eqref{eq:sublevel_set_est_well_curved} even in the case of polynomials $P(s,t) = \det(sA + tB)$, the class of polynomials for which we incur such a loss can be narrowed down significantly. Indeed, with a more precise argument (such as e.g.\@ the aforementioned integration argument in polar coordinates using Corollary 3.4 of \cite{IkromovMuller}) one incurs logarithmic losses only when the polynomial $P$ has a root of multiplicity exactly equal to $d/2$. It follows that for well-curved surfaces $\Sigma(Q_1,Q_2)$ we can always obtain the restricted weak-type endpoint $L^{(d+4)/4} \to L^{(d+4)/2}$, provided all the roots have multiplicity strictly smaller than $d/2$. In particular, one recovers in these cases the critical line that is missing from the statement of Theorem \ref{main_theorem}.
\end{remark}
\section{Flat surfaces}\label{section:flat_surfaces}
In this final section we will give counterexamples that show the necessity of the curvature assumptions of Theorems \ref{main_theorem} and \ref{thm:flat_surfaces}. More specifically, for flat $\Sigma(Q_1,Q_2)$ surfaces:
\begin{itemize}
\item We will show that if $\det(s \nabla^2 Q_1 + t \nabla^2 Q_2)$ does not vanish identically but has a root of multiplicity $m_{\ast} >d/2$, then for any $(p,q)$ sufficiently close to the endpoint $\big(\tfrac{d+4}{4}, \tfrac{d+4}{2}\big)$ the $L^p \to L^q$ estimate for operator $\mathcal{T}$ given by \eqref{eq:definition_restricted_2-plane_transform} is false; in particular, we will show that any estimate with $2/q = 1/p$ and $q < m_{\ast} + 2$ is false.
\item We will show that if $\det(s \nabla^2 Q_1 + t \nabla^2 Q_2)$ vanishes identically then any estimate with $2/q = 1/p$ is false (except for $p = q = \infty$); more in general, we will rule out every estimate for which $(2 - \epsilon)/q < 1/p$ for some $\epsilon > 0$ (this range intersects non-trivially the conjectural non-mixed range given by \eqref{eq:mixed_necessary_conditions}).
\end{itemize}
We will deal with each case in a separate subsection. Once again we resort to writing $A,B$ for $\nabla^2 Q_1, \nabla^2 Q_2$.
\subsection{Case I: $\det(sA+tB)$ is not identically vanishing}
In order to allow for a cleaner argument, we begin by making some reductions that are entirely analogous to those operated in Section \ref{section:semistability}; some care is needed because of the local nature of $\mathcal{T}$. For added precision, we introduce operators 
\[ \mathcal{T}_{\Omega}^{A,B}f(x,\xi) := \iint_{\Omega} f(x - (sA + tB)\xi, s,t) \,ds\,dt, \]
in which the subscript $\Omega$ specifies the integration domain; thus for the operator given by \eqref{eq:definition_restricted_2-plane_transform} we have $\mathcal{T} = \mathcal{T}_{[-1,1]^2}^{A,B}$.\par  First of all, we claim that we can assume that $B$ is invertible. Indeed, otherwise there exists some $\tau_0$ such that $B_0 := -A - \tau_0 B$ is invertible, and we can write 
\[ sA + tB = s' A_0 + t' B_0 \]
for $A_0 := B$ and $s',t'$ given by 
\[ \begin{pmatrix} s' \\ t' \end{pmatrix} = N \begin{pmatrix}
s \\ t
\end{pmatrix}, \qquad N = \begin{pmatrix}
-\tau_0 & 1 \\ -1 & 0
\end{pmatrix} \in SL(\mathbb{R}^2). \]
If for any function $f$ we let $f_{\tau_0}(y,s,t):= f(y, t - s \tau_0, -s)$, we see by a change of variables that 
\[ \mathcal{T}_{[-1,1]^2}^{A,B} f_{\tau_0} = \mathcal{T}_{N([-1,1]^2)}^{A_0,B_0} f; \]
therefore it will suffice to show that $\mathcal{T}_{N([-1,1]^2)}^{A_0,B_0}$ is unbounded, where now $B_0$ is invertible. Notice that since the operators are positive it will suffice to show that $\mathcal{T}_{[-\epsilon,\epsilon]^2}^{A_0,B_0}$ is unbounded for some $\epsilon > 0$ such that $[-\epsilon,\epsilon]^2 \subset N([-1,1]^2)$; by a rescaling, it then suffices to show that $\mathcal{T}_{[-1,1]^2}^{\epsilon A_0, \epsilon B_0}$ is unbounded.\par 
Assuming then that $B$ is invertible, we further claim that we can assume that $(A,B)$ is in the form $(\widetilde{\bm{J}}, \widetilde{\bm{I}})$ given by \eqref{eq:special_form_A_B} of Section \ref{section:semistability}. Indeed, using the notation of that section, we see that 
\begin{align*}
sA + tB &= (s AB^{-1} + t I) B = (s Q J Q^{-1} + t I)B \\
&= Q(s \widetilde{\bm{J}} \widetilde{\bm{I}} + t \widetilde{\bm{I}}^2) Q^{-1} B = Q(s \widetilde{\bm{J}} + t \widetilde{\bm{I}}) \widetilde{\bm{I}} Q^{-1} B
\end{align*}
(recall that $Q$ is an invertible matrix such that $Q^{-1}AB^{-1}Q$ is in Jordan normal form). If for any function $f$ we let $f_{Q^{-1}}(y,s,t) := f(Q^{-1} y, s,t)$, we see by a straightforward calculation that 
\[ \mathcal{T}_{[-1,1]^2}^{A,B} f_{Q^{-1}} (Qx, B^{-1} Q\widetilde{\bm{I}} \xi) = \mathcal{T}_{[-1,1]^2}^{\widetilde{\bm{J}}, \widetilde{\bm{I}}} f(x,\xi). \]
As a consequence, it will suffice to show that $\mathcal{T}_{[-1,1]^2}^{\widetilde{\bm{J}}, \widetilde{\bm{I}}}$ is unbounded from $L^p(B(0,C)\times [-1,1]^2)$ to $L^q(\mathbb{R}^d \times [-\epsilon',\epsilon']^d)$, where $\epsilon'>0$ is chosen sufficiently small to ensure $B^{-1} Q \widetilde{\bm{I}}([-\epsilon',\epsilon']^d) \allowbreak \subset [-1,1]^d$.\par
Finally, assuming that the matrices are of the form $(\widetilde{\bm{J}}, \widetilde{\bm{I}})$, we can further assume that the eigenvalue of $\widetilde{\bm{J}}$ of highest multiplicity is $\lambda_{\ast} = 0$: this can be achieved by a repetition of the argument given to show that we could assume $B$ to be invertible, and thus we omit the details. Associated to eigenvalue $0$ we have the generalised eigenspaces of $\widetilde{\bm{J}}$: let $V_0$ be the span of all the generalised eigenspaces of dimension $1$, and let $V_1, \ldots, V_\ell$ be the generalised eigenspaces of dimension larger than $1$. For $j \in \{0,\ldots, \ell\}$ we let $\bm{e}^{(j)}_{1}, \ldots, \bm{e}^{(j)}_{n_j}$ be the generalised eigenvectors that span $V_j$, where $n_j := \dim V_j$. Moreover, we let $W$ denote the span of the generalised eigenspaces of non-zero eigenvalue -- thus $\mathbb{R}^d = V_0 \oplus \ldots \oplus V_{\ell} \oplus W$. In the resulting basis of generalised eigenvectors the matrix $\widetilde{\bm{J}}$ has the form
\[ \widetilde{\bm{J}} = \begin{pmatrix}
0 & & & & & & & & \\
 & \ddots & & & & & & & \\
 & & 0 & & & & & & \\
 & & & \widetilde{J}_{n_1}(0) & & & & & \\
 & & & & \ddots & & & & \\
 & & & & & \widetilde{J}_{n_\ell}(0) & & & \\
 & & & & & & \bm{\ast} & \cdots & \bm{\ast} \\
 & & & & & & \vdots & \ddots & \vdots \\
  & & & & & & \bm{\ast} & \cdots & \bm{\ast} 
\end{pmatrix}, \]
where $\widetilde{J}_{r}(0)$ is given by \eqref{eq:antisymmetric_Jordan_blocks}; matrix $\widetilde{\bm{I}}$ has analogous form but $\widetilde{J}_r(0)$ is replaced by $\widetilde{I}_r$ (also given by \eqref{eq:antisymmetric_Jordan_blocks}). In particular, we have
\begin{equation}\label{eq:J_I_acting_on_V_0} \widetilde{\bm{J}} \bm{e}^{(0)}_k = 0, \qquad \widetilde{\bm{I}}\bm{e}^{(0)}_k = \bm{e}^{(0)}_k 
\end{equation}
for every $k \leq n_0$, and for $1 \leq j \leq \ell$
\begin{equation}\label{eq:J_I_acting_on_V_j}
\begin{aligned}
\widetilde{\bm{J}}\bm{e}^{(j)}_k &= \bm{e}^{(j)}_{n_j - k}, \\
 \widetilde{\bm{J}}\bm{e}^{(j)}_{n_j} &=0,
 \end{aligned} 
 \quad 
 \begin{aligned}
 \widetilde{\bm{I}}\bm{e}^{(j)}_k &= \bm{e}^{(j)}_{n_j +1 - k}, \quad \text{ for } 1 \leq k \leq n_j - 1, \\
 \widetilde{\bm{I}}\bm{e}^{(j)}_{n_j} &= \bm{e}^{(j)}_{1}. 
 \end{aligned}
\end{equation} \par 
We introduce two types of parabolic boxes adapted to the generalised eigenspaces: for any $j \in \{0,\ldots, \ell\}$ and $\delta > 0$ (an arbitrarily small parameter) we let 
\[ R(\delta,V_j) := \Big\{ \sum_{k=1}^{n_j} \alpha^{(j)}_k \delta^{n_j - k} \bm{e}^{(j)}_k : |\alpha^{(j)}_k| < \epsilon' \text{ for all } k \Big\}, \]
and for $\epsilon >0$ we let also
\[ \widetilde{R}(\delta, \epsilon, V_j) := \Big\{ \sum_{k=1}^{n_j} \beta^{(j)}_k \delta^{k} \bm{e}^{(j)}_k : |\beta^{(j)}_k| < \epsilon \text{ for all } k \Big\} \]
(notice how $R(\delta, V_j)$ is a parabolic box of dimensions $\sim \delta^{n_j - 1} \times \ldots \times \delta \times 1$, whereas $\widetilde{R}(\delta, \epsilon, V_j)$ is a parabolic box of dimensions $\sim \delta \times \delta^2 \times  \ldots \times \delta^{n_j}$). Consider now parameters $(s,t)$ restricted to the strip 
\[ S_{\delta} := \{ (s,t) \in [-1,1]^2 : |t|<\delta\} \]
and let us study how $s \widetilde{\bm{J}} + t \widetilde{\bm{I}}$ acts on the parabolic boxes. If $\bm{v} \in R(1,V_0)$ we have $\bm{v} = \sum_{k=1}^{n_0} \alpha^{(0)}_k  \bm{e}^{(0)}_k$ and thus by \eqref{eq:J_I_acting_on_V_0}
\[ (s \widetilde{\bm{J}} + t \widetilde{\bm{I}})\bm{v} = \sum_{k=1}^{n_0} t \alpha^{(0)}_{k}  \bm{e}^{(0)}_k; \]
as a consequence, we have $(s \widetilde{\bm{J}} + t \widetilde{\bm{I}}) R(1, V_0) \subset \widetilde{R}(1,\delta \epsilon', V_0)$. If $\bm{v} \in R(\delta, V_j)$ for $1 \leq j \leq \ell$ we have $\bm{v} = \sum_{k=1}^{n_j} \alpha^{(j)}_k \delta^{n_j - k} \bm{e}^{(j)}_k$ and thus by \eqref{eq:J_I_acting_on_V_j}
\[ (s \widetilde{\bm{J}} + t \widetilde{\bm{I}})\bm{v} = t \alpha^{(j)}_1 \delta^{n_j - 1} \bm{e}^{(j)}_{n_j} + \sum_{k=1}^{n_j - 1} (s \alpha^{(j)}_{n_j - k} \delta + t \alpha^{(j)}_{n_j + 1 - k}) \delta^{k-1} \bm{e}^{(j)}_k; \]
therefore $(s \widetilde{\bm{J}} + t \widetilde{\bm{I}}) R(\delta, V_j) \subset \widetilde{R}(\delta, 2\epsilon', V_j)$. Such inclusions have the following consequences: define (with a little abuse of notation) subsets of $\mathbb{R}^d$
\begin{align*}
E_{\delta} &:= R(1,V_0) \times \Big(\prod_{j=1}^{\ell} R(\delta, V_j)\Big) \times \{ \bm{w} \in W : \|\bm{w}\|_{\ell^{\infty}} < \epsilon'\}, \\
F_{\delta} &:= \widetilde{R}(1,\delta\epsilon',V_0) \times \Big( \prod_{j=1}^{\ell} \widetilde{R}(\delta, 2\epsilon',V_j)\Big)\times \{\bm{w} \in W : \|\bm{w}\|_{\ell^{\infty}} \lesssim_{\widetilde{\bm{J}},\widetilde{\bm{I}}} \epsilon'\};
\end{align*}
then we have $(s \widetilde{\bm{J}} + t \widetilde{\bm{I}}) E_\delta \subset F_\delta$ and $F_\delta - F_\delta \subset 2F_{\delta}$, which in particular implies 
\[ \mathcal{T} \mathbf{1}_{2F_{\delta} \times S_{\delta}} \geq |S_{\delta}| \mathbf{1}_{F_{\delta} \times E_{\delta}} \]
(where we wrote $\mathcal{T}$ for $\mathcal{T}_{[-1,1]^2}^{\widetilde{\bm{J}},\widetilde{\bm{I}}}$ to ease the notation a little). If $\mathcal{T}$ were $L^p \to L^q$ bounded, the last inequality would imply (with some rearranging)
\[ |S_{\delta}|^{1/{p'}} |E_{\delta}|^{1/q} \lesssim |F_{\delta}|^{1/p - 1/q}. \]
However, it is easy to see that in terms of $\delta$
\[ |S_{\delta}| \sim \delta, \quad |E_{\delta}| \sim \delta^{\sum_{j=1}^{\ell} n_j(n_j - 1)/2}, \quad |F_{\delta}| \sim \delta^{n_0 + \sum_{j=1}^{\ell} n_j(n_j + 1)/2}, \]
and letting $\delta \to 0$ we obtain the necessary condition (after further rearranging)
\begin{equation}\label{eq:flat_surfaces_necessary_condition_1}
1 + \Big(n_0 + \sum_{j=1}^{\ell} n_j^2\Big) \frac{1}{q} \geq \Big( 1 + n_0 + \sum_{j=1}^{\ell} \frac{n_j(n_j + 1)}{2} \Big) \frac{1}{p}. 
\end{equation}
Observe that $m_{\ast} = n_0 + \sum_{j=1}^{\ell} n_j$, so that if we restrict ourselves to exponents such that $2/q = 1/p$ we see with some algebra that \eqref{eq:flat_surfaces_necessary_condition_1} yields the same set of exponents as the condition 
\[ 1 + \frac{m_{\ast}}{q} \geq \frac{m_{\ast} + 1}{p} \]
stated in Theorem \ref{thm:flat_surfaces}. On the other hand, the general condition excludes a range of exponents beyond those strictly on the critical line $2/q = 1/p$, as illustrated in Figure \ref{figure:flat_surfaces_range}. The figure also illustrates that the reduced range provided by \eqref{eq:flat_surfaces_necessary_condition_1} does not quite coincide with the range of true estimates afforded by Theorem \ref{thm:flat_surfaces}; notice however that the two ranges coincide when $m_{\ast} = n_0$, that is, when the generalised eigenspaces of eigenvalue $\lambda_{\ast}$ are all of dimension $1$ (Theorem \ref{thm:flat_surfaces} is then sharp in such cases, save perhaps for the endpoint).
%
%
\begin{figure}[ht]
\centering
\begin{tikzpicture}[line cap=round,line join=round,>=Stealth,x=1cm,y=1cm, scale=7]
\clip(-0.15,-0.1) rectangle (1.55,1.1);
\draw [->,line width=0.5pt] (0,-0.06) -- (0,1.1);
\draw [->,line width=0.5pt] (-0.06,0) -- (1.1,0);
\draw [line width=0.5pt] (-0.02,1)-- (1,1);
\draw [line width=0.5pt,domain=-0.1:1.2] plot(\x,{(-0--1*\x)/2});
\draw [line width=0.5pt,domain=-0.1:1.65] plot(\x,{(-2--7*\x)/5});
\draw [line width=0.5pt,domain=-0.1:0.95] plot(\x,{(-1--8*\x)/10});
\draw [line width=0.5pt] (1/3,1/6)-- (1,1);
\draw [line width=0.5pt] (1,-0.02)-- (1,0.02);
\draw (0.4,0.27) node[anchor=south east] {$\Big(\frac{2}{m_{\ast} + 2}, \frac{1}{m_{\ast}+2}\Big)$};
\draw (1.05,0.84) node[anchor=north west] {$2+\frac{d}{q} = \frac{d+2}{p}$};
\draw (0.48,0.21) node[anchor=north west] {$1 + \Big(n_0 + \sum_{j} n_j^2\Big)\frac{1}{q} = \Big(1 + n_0 + \sum_{j} \frac{n_j(n_j+1)}{2}\Big)\frac{1}{p}$};
\draw (0.532,-0.015) node[anchor=north west] {$1/p$};
\draw (-0.13,0.6) node[anchor=north west] {$1/q$};
\draw (0.95,0.36) node[anchor=north west] {$\frac{2}{q} = \frac{1}{p}$};
\draw (0.5,0.45) node[anchor=south east] {$\Big(\frac{4}{d+4}, \frac{2}{d+4}\Big)$};
\draw (-0.08,1.035) node[anchor=north west] {$1$};
\draw (0.97,-0.04) node[anchor=north west] {$1$};
\begin{scriptsize}
\draw [fill=black] (4/9,2/9) circle (0.2pt);
\draw [fill=black] (1/3,1/6) circle (0.2pt);
\fill[fill=black,fill opacity=0.1] (0,0) -- (1/3,1/6) -- (1,1) -- (0,1) -- cycle;
\draw [->,line width=0.5pt] (0.64,0.2) -- (0.58,0.364);
\draw [->,line width=0.5pt] (0.94,0.31) -- (0.8,0.4);
\draw [->,line width=0.5pt] (1.05,0.78) -- (0.76,0.664);
\draw [->,line width=0.5pt] (0.24,0.276) -- (0.3295,0.168);
\draw [->,line width=0.5pt] (0.385,0.45) -- (0.4435,0.225);
\end{scriptsize}
\end{tikzpicture}
\caption{\footnotesize The shaded area corresponds to the range of boundedness afforded by Theorem \ref{thm:flat_surfaces}, that is, when the surface $\Sigma(Q_1,Q_2)$ is flat but $\det(s\nabla^2 Q_1 + t \nabla^2 Q_2)$ does not vanish identically. The critical lines given by \eqref{eq:mixed_necessary_conditions} and \eqref{eq:flat_surfaces_necessary_condition_1} are indicated: as one can see, the range of Theorem \ref{thm:flat_surfaces} is sharp when $2/q = 1/p$. The endpoint $\big(\tfrac{4}{d+4}, \tfrac{2}{d+4}\big)$ for the well-curved case is also indicated, and one can see that for these surfaces all $L^p \to L^q$ estimates for $(1/p,1/q)$ close to this endpoint are false.} \label{figure:flat_surfaces_range}
\end{figure}
\subsection{Case II: $\det(sA+tB)$ vanishes identically}
We consider first the case in which $\ker (s_1 A + t_1 B) \cap \ker (s_2 A + t_2 B) = \{0\}$ for any linearly independent $(s_1,t_1), (s_2, t_2)$ (equivalently, $\ker A \cap \ker B = \{0\}$). As in Section \ref{section:semistability_vanishing_determinant}, we can locate a maximal non-vanishing minor $\det_{I_{\ast}, J_{\ast}}(sA + tB)$ (where $I_{\ast},J_{\ast} \subset \{1,\ldots, d\}$ and $|I_{\ast}|=|J_{\ast}|$) and use it to define the set of generic $(s,t)$:
\[ \mathscr{G} := \{ (s,t) \in \mathbb{R}^2 : \det\nolimits_{I_{\ast}, J_{\ast}}(sA + tB) \neq 0 \} \]
(notice that, unlike in Section \ref{section:semistability_vanishing_determinant}, we are considering real parameters only). Observe that we can find a set $S \subset [-1,1]^2 \cap \mathscr{G}$ such that $|S|>1/2$, since $\mathscr{G}$ is simply $\mathbb{R}^2$ with some lines removed. We define then the subspace of $\mathbb{R}^d$
\[ V := \operatorname{Span}\Big\{ \bigcup_{(s,t) \in \mathscr{G}} \ker (sA + tB) \Big\}; \]
by the same arguments given in Section \ref{section:semistability_vanishing_determinant} we have that for every $(s,t) \in \mathscr{G}$ the image $(sA + tB)V$ consists of a common subspace $H$, which is a strict subspace of $V$. As a consequence, if $\xi \in \mathcal{N}_{\delta}(V)$ (the $\delta$-neighbourhood of $V$) we see that for $(s,t) \in S$ we have $(sA + tB)\xi \in \mathcal{N}_{K \delta}(H)$, where $K := \|A\|+ \|B\|$. Define then sets 
\begin{align*}
E_{\delta} &:= \mathcal{N}_{\delta}(V) \cap [-1,1]^d, \\
F_{\delta} &:= \mathcal{N}_{K \delta}(H) \cap [-K,K]^d;
\end{align*}
by the discussion above we have that
\[ \mathcal{T} \mathbf{1}_{2F_{\delta} \times S} \gtrsim \mathbf{1}_{F_{\delta} \times E_{\delta}}, \]
and therefore if $\mathcal{T}$ is $L^p \to L^q$ bounded we have from the last inequality (after some rearranging)
\[ |E_{\delta}|^{1/q} \lesssim |F_{\delta}|^{1/p - 1/q}. \]
It is easy to see that 
\[ |E_{\delta}| \sim \delta^{d - \dim V}, \quad |F_{\delta}| \sim \delta^{d - \dim H}, \]
so that letting $\delta \to 0$ we obtain the necessary condition
\[ \frac{d - \dim V}{q} \geq (d - \dim H)\Big(\frac{1}{p} - \frac{1}{q}\Big), \]
which after some rearranging is rewritten as 
\[ \Big( 2 - \frac{\dim V - \dim H}{d - \dim H}\Big) \frac{1}{q} \geq \frac{1}{p}, \]
as claimed in Theorem \ref{thm:flat_surfaces}. Since $\dim V > \dim H$, the condition shows that every $L^p \to L^q$ estimate with $2/q = 1/p$ is false in this case (with the exclusion of $(p,q) = (\infty, \infty)$). \par
It remains to treat the case in which $\ker A \cap \ker B \neq \{ 0\}$, in which case $\mathcal{T}$ does not satisfy any non-trivial estimate. Indeed, there exists a strict subspace $W \subsetneq \mathbb{R}^d$ such that $(sA + tB) \mathbb{R}^d \subset W$ for all $(s,t)$. If we let 
\[ F_{\delta} := \mathcal{N}_{\delta}(W) \cap [-K,K]^d \] 
we see easily that 
\[ \mathcal{T} \mathbf{1}_{2F_{\delta} \times [-1,1]^2} \geq \mathbf{1}_{F_{\delta} \times [-1,1]^d}; \]
if $\mathcal{T}$ is $L^p \to L^q$ bounded we have then
\[ |F_{\delta}|^{1/q} \lesssim |F_{\delta}|^{1/p}, \]
and since $|F_{\delta}| \sim \delta^{d - \dim W}$ it is immediate to deduce the necessary condition $1/q \geq 1/p$. Thus every estimate beyond those obtained from interpolation of the trivial estimates of Remark \ref{remark:trivial_estimates} is false.
\appendix
\section{General well-curved surfaces}\label{appendix:modification_general_surface}
In this appendix we sketch the modifications of the arguments presented in this paper that allow to extend Theorem \ref{main_theorem} to Theorem \ref{thm:general_well_curved_surfaces}, that is, to general well-curved surfaces $\Sigma(\varphi_1,\varphi_2)$ of the form 
\[ (\xi, \varphi_1(\xi), \varphi_2(\xi)), \qquad \xi \in [-\epsilon, \epsilon]^d, \]
where $\varphi_1,\varphi_2$ are $C^2$ functions such that $\nabla \varphi_1(0) = \nabla \varphi_2(0) = 0$, and $\epsilon$ will be taken sufficiently small depending on $\varphi_1,\varphi_2$. We will borrow heavily from other sections and their notations to keep the appendix short.\par 
The first observation is that if $\Sigma(\varphi_1,\varphi_2)$ is well-curved at $\xi = 0$ then it is well-curved in a neighbourhood of $0$ as well. Indeed, this is a consequence of the fact that condition \eqref{condition:well-curvedness} is stable under small perturbations: observe that the coefficients of the polynomial $\det(s \nabla^2\varphi_1(\xi) + t \nabla^2 \varphi_2(\xi))$ are continuous functions of $\xi$. It is well-known that the roots of a univariate polynomial are continuous functions of its coefficients, and it is not hard to see that this fact extends to homogeneous polynomials of two variables (for example, by passing to the projectivisation). Thus the roots  of $\det(s \nabla^2\varphi_1(\xi) + t \nabla^2 \varphi_2(\xi))$ are continuous functions of $\xi$ and we see that if \eqref{condition:well-curvedness} is satisfied at $\xi = 0$ then it is satisfied for $\xi \in [-\epsilon,\epsilon]^d$ for some $\epsilon >0$ (this is because the maximal algebraic multiplicity of the roots of $\det(sA+tB)$ is an upper semicontinuous function of the matrices $A,B$).\par 
The bulk of the argument of Section \ref{section:main_proof} goes through without major changes: in particular, the Jacobian determinant of the map $\Psi$ is still given by \eqref{eq:Jacobian} -- that is, by $\det((s-s_0) \nabla^2\varphi_1(\eta) + (t-t_0) \nabla^2 \varphi_2(\eta))$, which unlike  the quadratic case is now a function of $\eta$ too. For $\epsilon$ sufficiently small, the multiplicity $\mu_{\Psi}$ of the map $\Psi$ is still $1$. Indeed, we see that $\Psi(\eta,s,t) = \Psi(\eta',s',t')$ only if $s=s', t=t'$ and 
\[ \hat{s}(\nabla \varphi_1(\eta) - \nabla \varphi_1(\eta')) + \hat{t} (\nabla \varphi_2(\eta) - \nabla \varphi_2(\eta')) = 0 \]
(where $\hat{s} = s - s_0$, $\hat{t} = t - t_0$ for shortness); this can be rewritten as 
\[ \Big( \int_{0}^{1} [\hat{s} \nabla^2 \varphi_1 + \hat{t} \nabla^2 \varphi_2](\theta \eta + (1 - \theta)\eta') \,d\theta\Big) (\eta - \eta') = 0, \]
so that the matrix in brackets must have determinant zero if $\eta \neq \eta'$. However, expanding the determinant we see that it equals 
\[ \int_{[0,1]^d} \sum_{\sigma \in S_d} \operatorname{sgn}\sigma \prod_{j=1}^{d} \partial_j \partial_{\sigma(j)} (\hat{s} \varphi_1 + \hat{t} \varphi_2)(\theta_j \eta + (1 - \theta_j) \eta') \,d\theta_1 \ldots \,d\theta_d; \]
the integrand is seen to be the determinant of a matrix that is a small perturbation of $\hat{s} \nabla^2 \varphi_1(\eta) + \hat{t} \nabla^2 \varphi_2(\eta)$. If we impose -- as we do -- that for $(\eta,s,t) \in S$ (where $S$ is given by \eqref{eq:tower_of_parameters}) this is non-zero, then the integrand is never zero and in particular single-signed (provided $\epsilon$ is small), and therefore the determinant above is not zero and $\eta = \eta'$. \par 
To complete the proof given in Section \ref{section:main_proof} all that remains to show is that we can make the sublevel set estimate \eqref{eq:sublevel_set_est_well_curved} uniform in $\eta$; this is the most delicate part. First of all, recall as observed in Remark \ref{remark:implicit_constant_change_of_var} that the implicit constant in Lemma \ref{lemma:sublevel_set_estimate_two_factors} can be made explicit: with $\theta,\theta'$ normalised linear forms (which for simplicity we assume real, without loss of generality), we have
\[ |\{ (s,t) : |s|,|t|\leq 1, |\theta(s,t)^{\mu} \theta'(s,t)^{\nu}|<\delta\}| \lesssim \frac{|\{ (s,t) : |s|,|t|\lesssim 1, |s^{\mu} t^{\nu}|<\delta\}|}{|\det \begin{pmatrix} \theta & \theta' \end{pmatrix}|}, \]
where $\det \begin{pmatrix} \theta & \theta' \end{pmatrix}$ is the Jacobian determinant of the map $(s,t) \mapsto (\theta(s,t), \theta'(s,t))$; thus the implicit constant is $O(|\det \begin{pmatrix} \theta & \theta' \end{pmatrix}|^{-1})$. Secondly, by continuity of the roots we have the following: if 
\[ \theta_1^{m_1}, \ldots, \theta_{\ell}^{m_{\ell}} \]
are the distinct normalised roots of $\det(s \nabla^2 \varphi_1 (0) + t \nabla^2 \varphi_2 (0))$ with respective multiplicities, then for a fixed $\eta \in [-\epsilon, \epsilon]^d$ and $\epsilon$ sufficiently small the distinct normalised roots of $\det(s \nabla^2 \varphi_1 (\eta) + t \nabla^2 \varphi_2 (\eta))$ are 
\[ \tilde{\theta}_{1 1}^{m_{1 1}}, \ldots, \tilde{\theta}_{1  n_1}^{m_{1 n_1}}, \ldots, \tilde{\theta}_{\ell 1}^{m_{\ell 1}}, \ldots, \tilde{\theta}_{\ell n_\ell}^{m_{\ell n_{\ell}}}, \]
where each $\tilde{\theta}_{ji}$ for $1\leq i \leq n_j$ is a small perturbation of $\theta_{j}$ and for each $j$ we have $\sum_{i=1}^{n_j} m_{ji} = m_j$. In particular, for any $j,k$ we have 
\[ |\det \begin{pmatrix} \theta_j & \theta_k \end{pmatrix}| \sim |\det \begin{pmatrix} \tilde{\theta}_{ji} & \tilde{\theta}_{k i'} \end{pmatrix}| \]
for all $1\leq i \leq n_j$ and $1 \leq i' \leq n_k$. To obtain a sublevel set estimate that is uniform in $\eta \in [-\epsilon, \epsilon]^d$ it will then suffice to show that we can find coefficients $\mu_{j i k i'} \geq 0$ such that we have the structural factorisation
\[ \prod_{j=1}^{\ell}\prod_{i=1}^{n_j} \tilde{\theta}_{j i}^{m_{j i}} = \prod_{j=1}^{\ell} \prod_{j < k \leq \ell} \prod_{i=1}^{n_j} \prod_{i'=1}^{n_k} (\tilde{\theta}_{j i} \tilde{\theta}_{k i'})^{\mu_{j i k i'}} \]
(in this way in our constants we will avoid terms like $|\det(\begin{pmatrix} \tilde{\theta}_{j i} & \tilde{\theta}_{j i'} \end{pmatrix}|^{-1}$, which could be arbitrarily large). This can be achieved by a variation of the argument used in the proof of Proposition \ref{prop:sublevel_set_estimate_1}, as we now illustrate. As in there, the existence of such a factorisation translates into the existence of a non-negative solution to the equations
\[ L_{ji} : \quad \sum_{k < j} \sum_{i' = 1}^{n_k} \mu_{k i' j i} + \sum_{k'>j}\sum_{i''=1}^{n_{k'}} \mu_{j i k' i''} = m_{j i}, \]
for $1\leq j \leq \ell$ and $1\leq i \leq n_j$. Appealing once again to Lemma \ref{lemma:Farkas}, it suffices to show that there is no simultaneous solution $(y_{j i})_{j\leq \ell, i\leq n_j}$ to the inequalities
\[ \left\{ \begin{aligned}
y_{j i} + y_{k i'} &\geq 0 \quad \text{ for all } 1\leq j < k \leq \ell,\, 1\leq i \leq n_j, \, 1\leq  i' \leq n_k, \\
\sum_{j=1}^{\ell} \sum_{i=1}^{n_j} m_{j i} y_{j i} & < 0.
\end{aligned} \right. \]
Since $m_{ji} >0$ the second inequality implies that for some $j_{\ast}$ one coefficient $y_{j_{\ast} i}$ is negative; let $y_{j_{\ast} i_{\ast}}$ be the most negative of such coefficients. From the first inequality we see that for every $j \neq j_{\ast}$ we must have $y_{ji} \geq |y_{j_{\ast} i_{\ast}}| > 0$ for all $1\leq i\leq n_j$, and therefore we have 
\begin{align*}
&\sum_{j \neq j_{\ast}} m_j = \sum_{j \neq j_{\ast}} \sum_{i=1}^{n_j} m_{ji} \leq  \sum_{j \neq j_{\ast}} \sum_{i=1}^{n_j} m_{ji} \frac{y_{ji}}{|y_{j_{\ast} i_{\ast}}|} \\
&= \frac{1}{|y_{j_{\ast} i_{\ast}}|}\sum_{j=1}^{\ell} \sum_{i=1}^{n_j} m_{ji}y_{ji} - \frac{1}{|y_{j_{\ast} i_{\ast}}|}\sum_{i=1}^{n_{j_{\ast}}} m_{j_{\ast} i} (y_{j_{\ast} i} - y_{j_{\ast} i_{\ast}}) + \sum_{i=1}^{n_{j_{\ast}}} m_{j_{\ast} i} \\
& < - \frac{1}{|y_{j_{\ast} i_{\ast}}|}\sum_{i=1}^{n_{j_{\ast}}} m_{j_{\ast} i} (y_{j_{\ast} i} - y_{j_{\ast} i_{\ast}}) + \sum_{i=1}^{n_{j_{\ast}}} m_{j_{\ast} i} \leq \sum_{i=1}^{n_{j_{\ast}}} m_{j_{\ast} i} = m_{j_{\ast}};
\end{align*}
this would imply $m_{j_{\ast}} > d/2$, a contradiction because $\Sigma(\varphi_1,\varphi_2)$ is well-curved at $\xi = 0$. This concludes the proof.
\subsection*{Acknowledgements} The third author was supported in part by the Irish Research Council via the IRC Postdoctoral Fellowship GOIPD/2019/434. The authors are indebted to P.\@ Gressman and M.\@ Christ for enlightening conversations about their work, and to J.\@ Bennett and M.\@ Iliopoulou for equally enlightening conversations on the Mizohata-Takeuchi conjecture.
%
%
%
\bibliographystyle{abbrv}
\bibliography{bibliography}

\end{document}